\documentclass{amsart}

\usepackage{mathtools}
\usepackage{tikz}
\usetikzlibrary{cd}
\usepackage{amssymb}
\usepackage{MnSymbol}
\usepackage{float}
\usepackage[normalem]{ulem}
\usepackage{comment}

\tikzset{
  symbol/.style={
    draw=none,
    every to/.append style={
      edge node={node [sloped, allow upside down, auto=false]{$#1$}}}
  }
}

\newtheorem{thm}{Theorem}[section]
\newtheorem{lem}[thm]{Lemma}
\newtheorem{prop}[thm]{Proposition}
\newtheorem{cor}[thm]{Corollary}
\newtheorem*{thm*}{Theorem}
\newtheorem{defn}[thm]{Definition}
\newtheorem{example}[thm]{Example}
\newtheorem*{definition*}{Definition}
\newtheorem{remark}[thm]{Remark}

\numberwithin{equation}{section}

\newcommand{\cont}{{\mathrm C }}
\newcommand{\contb}{{\mathrm C}_{\mathrm b}}

\newcommand{\contX}{\cont(X)}
\newcommand{\contbX}{\contb(X)}

\newcommand{\vlat}[1]{\mathrm{#1}}

\newcommand{\borel}[1]{\mathfrak{B}_{#1}}

\newcommand{\orderdual}[1]{{#1}^\sim}
\newcommand{\ordercontn}[1]{{#1}^\sim_{\mathrm n}}

\newcommand{\ordercontnbidual}[1]{{#1}^{\sim \sim}_{\mathrm{n n}}}

\newcommand{\onefunction}{{\mathbf 1}}
\newcommand\supp{\mathord{\mathrm{supp}}}
\newcommand{\zeroset}[1]{{\bf Z}_{#1}}
\newcommand{\cozeroset}[1]{{\bf Z}_{#1}^c}

\newcommand{\R}{\mathbb{R}}
\newcommand{\N}{\mathbb{N}}

\newcommand{\vlc}{{\bf VL}}
\newcommand{\nvlc}{{\bf NVL}}
\newcommand{\vlic}{{\bf IVL}}
\newcommand{\nvlic}{{\bf NIVL}}

\newcommand{\nullid}[1]{\mathrm{N}_{#1}}
\newcommand{\carrier}[1]{\mathrm{C}_{#1}}
\newcommand{\preann}[1]{\prescript{\circ}{}{#1}}
\newcommand{\ann}[1]{{#1}^\circ}
\newcommand{\bands}[1]{{\rm {\bf B}}_{#1}}

\newcommand{\precc}{\preccurlyeq}
\newcommand{\scc}{\succcurlyeq}
\newcommand{\defeq}{\ensuremath{\mathop{:}\!\!=}}
\newcommand{\proj}[1]{\underleftarrow{\lim}\hspace{1.5pt} #1 }
\newcommand{\ind}[1]{\underrightarrow{\lim}\hspace{1.5pt} #1 }
\newcommand{\cal}[1]{\mathcal{#1}}
\newcommand{\dsprod}{\displaystyle\prod}
\newcommand{\loc}{\ell oc}

\makeatletter
\newcommand*\bigcdot{\mathpalette\bigcdot@{.5}}
\newcommand*\bigcdot@[2]{\mathbin{\vcenter{\hbox{\scalebox{#2}{$\m@th#1\bullet$}}}}}
\makeatother

\newcommand{\mysetminusD}{\hbox{\tikz{\draw[line width=0.6pt,line cap=round] (3pt,0) -- (0,6pt);}}}
\newcommand{\mysetminusT}{\mysetminusD}
\newcommand{\mysetminusS}{\hbox{\tikz{\draw[line width=0.45pt,line cap=round] (2pt,0) -- (0,4pt);}}}
\newcommand{\mysetminusSS}{\hbox{\tikz{\draw[line width=0.4pt,line cap=round] (1.5pt,0) -- (0,3pt);}}}

\newcommand{\mysetminus}{\mathbin{\mathchoice{\mysetminusD}{\mysetminusT}{\mysetminusS}{\mysetminusSS}}}

\usepackage[UKenglish]{babel}

\usepackage{currfile}
\usepackage[UKenglish]{datetime}
\usepackage{color}

\begin{document}

\title[Limits of vector lattices]{Limits of vector lattices}

\author{Walt van Amstel}

\author{Jan Harm van der Walt}

\address{Department of Mathematics and Applied Mathematics, University of Pretoria, Cor\-ner of Lynnwood Road and Roper Street,
Hatfield 0083, Pretoria, South Africa and DSI-NRF Centre of Excellence in
Mathematical and Statistical Sciences (CoE-MaSS), South Africa}
\email{sjvdwvanamstel@gmail.com}

\address{Department of Mathematics and Applied Mathematics, University of Pretoria, Cor\-ner of Lynnwood Road and Roper Street,
Hatfield 0083, Pretoria, South Africa}
\email{janharm.vanderwalt@up.ac.za}

\thanks{The first author was supported by a grant from the DSI-NRF Centre of Excellence in Mathematical and Statistical Sciences (CoE-MaSS),
South Africa. Opinions expressed and conclusions arrived at are those of the authors and are not necessarily to be attributed to the CoE-MaSS.  The second author was supported by the NRF of South Africa, grant number 115047.  The results in this paper were obtained, in part, while both authors visited Leiden University from September 2021 to January 2022.  This visit was funded by the European Union Erasmus+ ICM programme.  The authors thank Prof. Marcel de Jeu and the Mathematical Institute at Leiden University for their hospitality.  The authors thank the reviewer for a meticulous reading of the paper along with a number of helpful suggestions.}


\date{\tt {\today}}




\begin{abstract}
If $K$ is a compact Hausdorff space so that the Banach lattice $\cont(K)$ is isometrically lattice isomorphic to a dual of some Banach lattice, then $\cont(K)$ can be decomposed as the $\ell^\infty$-direct sum of the carriers of a maximal singular family of order continuous functionals on $\cont(K)$.  In order to generalise this result to the vector lattice $\contX$ of continuous, real valued functions on a realcompact space $X$, we consider direct and inverse limits in suitable categories of vector lattices.  We develop a duality theory for such limits and apply this theory to show that $\contX$ is lattice isomorphic to the order dual of some vector lattice $\vlat{F}$ if and only if $\contX$ can be decomposed as the inverse limit of the carriers of all order continuous functionals on $\contX$.  In fact, we obtain a more general result:  A Dedekind complete vector lattice $\vlat{E}$ is perfect if and only if it is lattice isomorphic to the inverse limit of the carriers of a suitable family of order continuous functionals on $\vlat{E}$.  A number of other applications are presented, including a decomposition theorem for order dual spaces in terms of spaces of Radon measures.
\end{abstract}

\maketitle

\section{Introduction}\label{Section:  Introduction}

Let $K$ be a compact Hausdorff space. A basic question concerning the Banach lattice $\cont(K)$ is the following:  Does there exist a Banach space (lattice) $\vlat{E}$ so that $\cont(K)$ is isometrically (lattice) isomorphic to the dual $\vlat{E}^\ast$ of $\vlat{E}$?  That is, does $\cont(K)$ have a Banach space (lattice) predual?  In general, the answer to this question is `no'. The unit ball of $\cont[0,1]$ has only two extreme points, but the unit ball of the dual of an infinite dimensional Banach space has infinitely many extreme points.  Hence $\cont[0,1]$ is not the dual of any Banach space; hence also not of any Banach lattice. On the other hand, $\cont(\beta \N)$ is the dual of $\ell^1$.  The problem is therefore to characterise those spaces $K$ for which $\cont(K)$ is a dual Banach space (lattice). Combining two classic results of Dixmier \cite{Dixmier1951} and Grothendieck \cite{Grothendieck1955}, respectively, gives an answer to this question in the setting of Banach spaces, see also \cite{DalesDashiellLauStrass2016} for a recent presentation.  The Banach lattice case is treated in \cite{Schaefer1974}.

In order to formulate this result we recall the following.  A Radon measure $\mu$ on $K$ is called \emph{normal} if $|\mu|(B)=0$ for every closed nowhere dense subset $B$ of $K$.  The space of all normal Radon measures on $K$ is denoted $\vlat{N}(K)$.  The space $K$ is called \emph{Stonean} if it is extremally disconnected; that is, the closure of every open set is open.  $K$ is \emph{hyper-Stonean}\footnote{We feel obligated to recall Kelley's remark \cite{Kelley1959}: `In spite of my affection and admiration for Marshall Stone, I find the notion of a Hyper-Stone downright appalling.'} if it is Stonean and the union of the supports of the normal Radon measures on $K$ is dense in $K$.

\begin{thm}\label{Thm:  C(K) Dual space char}
Let $K$ be a compact Hausdorff space.  Consider the following statements. \begin{itemize}
    \item[(i)] $\cont(K)$ has a Banach lattice predual.
    \item[(ii)] $\cont(K)$ has a Banach space predual.
    \item[(iii)] $K$ is hyper-Stonean.
    \item[(iv)] Let $\cal{F}$ be a maximal singular family of normal probability measures on $K$, and for each $\mu\in \cal{F}$ let $S_\mu$ denote its support.  Then
    \[
    \cont(K)\ni u\longmapsto \left( \left.u\right|_{S_\mu} \right)_{\mu\in \cal{F}}\in \bigoplus_{\infty} \cont(S_\mu)
    \]
    is an isometric lattice isomorphism.
\end{itemize}
Statements (i), (ii) and (iii) are equivalent, and each implies (iv).  If $K$ is Stonean, then all four statements are equivalent.

Furthermore, in case $\cont(K)$ has a Banach space predual $\vlat{E}$, this predual is also a Banach lattice predual and is unique up to isometric lattice isomorphism.  In particular, $\vlat{E}$ is isometrically lattice isomorphic to $\vlat{N}(K)$.
\end{thm}

This result can be reformulated by identifying $\vlat{N}(K)$ with the order continuous dual of $\cont(K)$, via the isometric lattice isomorphism between the dual of $\cont(K)$ and the space of Radon measures on $K$, and $\cont(S_\mu)$ with the carrier of the corresponding functional on $\cont(K)$.

\begin{thm}\label{Thm:  C(K) Dual space char order dual version}
Let $K$ be a compact Hausdorff space.  Consider the following statements. \begin{itemize}
    \item[(i)] $\cont(K)$ has a Banach lattice predual.
    \item[(ii)] $\cont(K)$ has a Banach space predual.
    \item[(iii)] $\cont(K)$ is Dedekind complete and has a separating order continuous dual.
    \item[(iv)] Let $\cal{F}$ be a maximal singular family of order continuous functionals on $\cont(K)$, and for each $\varphi\in\cal{F}$ let $\vlat{C}_\varphi$ denote its carrier and $P_\varphi$ the band projection onto $\vlat{C}_\varphi$.  Then
    \[
    \cont(K)\ni u \longmapsto (P_\varphi u)_{\varphi\in \cal{F}}\in \bigoplus_{\infty} \vlat{C}_\varphi
    \]
    is an isometric lattice isomorphism.
\end{itemize}
Statements (i), (ii) and (iii) are equivalent, and each implies (iv).  If $K$ is Stonean, then all four statements are equivalent.

Furthermore, in case $\cont(K)$ has a Banach space predual $\vlat{E}$, this predual is also a Banach lattice predual and is unique up to isometric lattice isomorphism.  In particular, $\vlat{E}$ is isometrically lattice isomorphic to the order continuous dual $\ordercontn{\cont(K)}$ of $\cont(K)$.
\end{thm}

The above problem may be generalised to the class of realcompact spaces.  Recall that a \emph{realcompact} space is a Tychonoff space $X$ which is homeomorphic to a closed subset of some product of $\R$.  Equivalently, $X$ is realcompact if it is a Tychonoff space and for every point $x \in \beta X\setminus X$ (where $\beta X$ denotes the Stone-\v{C}ech compactification of $X$) there exists a real-valued, continuous function $u$ on $X$ which does not extend to a continuous, real-valued function on $X\cup\{x\}$.  For every Tychonoff space $X$ there exists a unique (up to homeomorphism) realcompact space $\upsilon X$ so that $\contX$ and $\cont(\upsilon X)$ are isomorphic vector lattices, see for instance \cite{Hewitt1948}, \cite[Chapter 8]{GillmanJerison1960} and \cite[\S 3.11]{Engelking1989}. The realcompact space $\upsilon X$ is called the \emph{realcompactification} of $X$.\label{RC-ification}

Let $X$ be a realcompact space.  Then $\cont(X)$ is a vector lattice but, in general, not a Banach lattice.  Hence we ask the following question: Does there exist a vector lattice $\vlat{E}$ so that $\orderdual{\vlat{E}}$ is lattice isomorphic to $\cont(X)$?  That is, does $\contX$ have an \emph{order predual}?  Xiong \cite{Xiong1983} obtained the following answer to this question.

\begin{thm}
Let $X$ be a realcompact space.  Denote by $S$ the union of the supports of all compactly supported normal Radon measures\footnote{See Section \ref{Subsection:  Realcompact spaces preliminaries}.} on $X$.  The following statements are equivalent. \begin{itemize}
    \item[(i)] There exists a vector lattice $\vlat{E}$ so that $\orderdual{\vlat{E}}$ is lattice isomorphic to $\cont(X)$.
    \item[(ii)] $\contX$ is lattice isomorphic to $\orderdual{(\ordercontn{\cont(X)})}$.
    \item[(iii)] $X$ is extremally disconnected and $\upsilon S=X$.
\end{itemize}
\end{thm}


This result differs from the corresponding result for compact spaces in the following respects.  Unlike in the Banach lattice setting, $\cont(X)$ may have more than one order predual, see \cite{Xiong1983}.  Secondly, the condition that $\contX$ is Dedekind complete and has a separating order continuous dual does not imply that $\contX$ has an order predual.  Indeed, in \cite[p. 620]{Mazon1986} an example is provided of a realcompact space $X$ so that $\contX$ is Dedekind complete and has a separating order continuous dual, but is not the order dual of any vector lattice.  Furthermore, we have no counterpart of the decomposition
\[
\cont(K)\ni u \longmapsto (P_\varphi u)_{\varphi\in \cal{F}}\in \bigoplus_{\infty} \vlat{C}_\varphi.
\]
The naive extension of this decomposition to the class of extremally disconnected realcompact spaces does not provide a characterization of those spaces $\contX$ which admit an order predual.  It will be shown in Section \ref{Subsection:  Structure theorems for C(X) as a dual space}, Proposition \ref{Prop:  Partial decomposition result for order dual C(X)}, that if $X$ is an extremally disconnected realcompact space and $\cal{F}$ is a maximal singular family in $\ordercontn{\contX}$ so that
\[
\cont(X)\ni u \longmapsto (P_\varphi u)_{\varphi\in \cal{F}}\in \prod_{\varphi\in\cal{F}} \vlat{C}_\varphi
\]
is a lattice isomorphism, then $\ordercontn{\contX}$ is an order predual for $\contX$.  The converse, however, is false, see Example \ref{Exm:  C(X) decomponsition counterexample}.

In view of the above, we formulate the following problem.  Let $X$ be an extremally disconnected realcompact space.  Can the property `$\contX$ admits an order predual' be characterised in terms of a suitable decomposition of $\contX$ in terms of the carriers of order continuous functionals on $\contX$?  We solve this problem using direct and inverse limits in suitable categories of vector lattices.\footnote{In the literature, direct and inverse limits are also referred to as \emph{inductive} and \emph{projective} limits, respectively.}

Such limits are common in analysis, see for instance \cite{BeattieButzmann2002}, \cite[Chapter IV, \S 5]{Conway1990}, \cite[Chapter 5]{Bochner1955} and \cite{Choksi1958}.  Direct limits of vector lattices were introduced by Filter \cite{Filter1988} and inverse limits of vector lattices have appeared sporadically in the literature, see for instance \cite{Dettweiler1979,Kuller1958}, but no systematic study of this construction has been undertaken in the context of vector lattices.  We therefore take the opportunity to clarify the question of existence of inverse limits in certain categories of vector lattices. We also establish the permanence of a number of vector lattice properties under the inverse limit construction. Our treatment of direct and inverse limits of vector lattices is found in Sections \ref{Section:  Inductive limits} and \ref{Section:  Projective limits}, respectively.  Inspired by results in the theory of convergence spaces \cite{BeattieButzmann2002} we obtain duality results for direct and inverse limits of vector lattices, see Section \ref{Section:  Dual spaces}.  These results are roughly of the following form:  If a vector lattice $\vlat{E}$ can be expressed as the direct (inverse) limit of some system of vector lattices, then the order (continuous) dual of $\vlat{E}$ can be expressed in a natural way as the inverse (direct) limit of a system of order (continuous) duals. In addition to a solution of the mentioned decomposition problem, a number of applications of the general theory of direct and inverse limits of vector lattices are presented in Section \ref{Section:  Applications}.  These include the computations of order (continuous) duals of function spaces and a structural characterisation of order dual spaces in terms of spaces of Radon measures.

In the next section, we state some preliminary definitions and results which are used in the rest of the paper.

\section{Preliminaries}\label{Section:  Preliminaries}

\subsection{Vector lattices}\label{Subsection:  Vector lattice preliminaries}

In order to make the paper reasonably self-contained we recall a few concepts and facts from the theory of vector lattices.  For undeclared terms and notation we refer to the reader to any of the standard texts in the field, for instance \cite{AliprantisBurkinshaw78,AliprantisBurkinshaw2006,LuxemburgZaanen1971RSI,Zaanen1983RSII}.  Let $\vlat{E}$ and $\vlat{F}$ be real vector lattices. For $u,v\in \vlat{E}$ we write $u<v$ if $u\leq v$ and $u\neq v$. In particular, $0 < u$ means $u$ is positive but not zero. We note that if $\vlat{E}$ is a space of real-valued functions on a set $X$, then $0 < v$ does not mean that $0 < v(x)$ for every $x\in X$.

For sets $A,B\subseteq \vlat{E}$ let $A\vee B \defeq \{u\vee v ~:~ u\in A,~ v\in B\}$. The sets $A\wedge B$, $A^+$, $A^-$ and $|A|$ are defined similarly.  Lastly, $A^d \defeq \{u\in \vlat{E} ~:~ |u|\wedge |v|=0 \text{ for all } v\in A\}$.  We write $A\downarrow u$ if $A$ is downward directed and $\inf A=u$. Similarly, we write $B\uparrow u$ if $B$ is upward directed and $\sup B = u$.

Let $T:\vlat{E}\to \vlat{F}$ be a linear operator.  Recall that $T$ is \emph{positive} if $T\left[ \vlat{E}^+ \right] \subseteq \vlat{F}^+$, and \emph{regular} if $T$ is the difference of two positive operators.  $T$ is \emph{order bounded} if $T$ maps order bounded sets in $\vlat{E}$ to order bounded sets in $\vlat{F}$.  If $\vlat{F}$ is Dedekind complete, $T$ is order bounded if and only if $T$ is regular \cite[Theorem 20.2]{Zaanen1997Introduction}.  Further, $T$ is \emph{order continuous} if $\inf |T[A]|=0$ whenever $A\downarrow 0$ in $\vlat{E}$.  Every order continuous operator is necessarily order bounded \cite[Theorem 1.54]{AliprantisBurkinshaw2006}.  $T$ is a \emph{lattice homomorphism} if it preserves suprema and infima of finite sets, and a \emph{normal lattice homomorphism} if it preserves suprema and infima of arbitrary sets; equivalently, if it is an order continuous lattice homomorphism, see \cite[p. 103]{LuxemburgZaanen1971RSI}.  A \emph{lattice isomorphism} is a bijective lattice homomorphism $T:\vlat{E}\to\vlat{F}$.  An operator $T$ is a lattice isomorphism if and only if it is bijective and both $T$ and $T^{-1}$ are positive \cite[Theorem 19.3]{Zaanen1997Introduction}.  We say that $T$ is \emph{interval preserving} if for all $0\leq u\in\vlat{E}$, $T[[0,u]]=[0,T(u)]$.  An interval preserving map need not be a lattice homomorphism, nor is a (normal) lattice homomorphism in general interval preserving, see for instance \cite[p. 95]{AliprantisBurkinshaw2006}.  However, the following holds.  We have not found this result in the literature, and therefore we include the simple proof.

\begin{prop}\label{Prop:  Interval Preserving vs Lattice Homomorphism}
Let $\vlat{E}$ and $\vlat{F}$ be vector lattices and $T:\vlat{E}\to\vlat{F}$ a positive operator.  The following statements are true. \begin{itemize}
    \item[(i)] If $T$ is injective and interval preserving then $T$ is a lattice isomorphism onto an ideal in $\vlat{F}$, hence a normal lattice homomorphism into $\vlat{F}$.
    \item[(ii)] If  $T$ is a lattice homomorphism and $T[\vlat{E}]$ is an ideal in $\vlat{F}$ then $T$ is interval preserving.
\end{itemize}
\end{prop}

\begin{proof}[Proof of (i)]
Assume that $T$ is injective and interval preserving.  $T[\vlat{E}]$ is an ideal in $\vlat{F}$ by \cite[Proposition 14.7]{Kaplan1985}.  Therefore, because $T$ is injective, it suffices to show that $T$ is a lattice homomorphism.  To this end, consider $u,v\in \vlat{E}^+$.  Then $0\leq T(u) \wedge T(v) \leq T(u)$ and $0\leq T(u)\wedge T(v) \leq T(v)$.  Since $T$ is interval preserving and injective there exists $w\in [0,u]\cap[0,v] = [0, u\wedge v]$ so that $T(w)=T(u)\wedge T(v)$.  We have
\[
T(w)\leq T(u\wedge v)\leq T(u) \text{ and } T(w)\leq T(u\wedge v)\leq T(v).
\]
Hence $T(u)\wedge T(v) = T(w) \leq T(u\wedge v) \leq T(u)\wedge T(v)$ so that $T(u\wedge v) = T(w) = T(u)\wedge T(v)$.

To see that $T$ is a normal lattice homomorphism, let $A\downarrow 0$ in $\vlat{E}$.  Then $T[A]\downarrow 0$ in $T[\vlat{E}]$ because $T$ is a lattice isomorphism onto $T[\vlat{E}]$.  But $T[\vlat{E}]$ is and ideal in $\vlat{F}$, so $T[A]\downarrow 0$ in $\vlat{F}$.
\end{proof}

\begin{proof}[Proof of (ii)]
Assume that $T$ is a lattice homomorphism and $T[\vlat{E}]$ is an ideal in $\vlat{F}$.  Let $0\leq u \in \vlat{E}$ and $0\leq v\leq T(u)$.  Because $T[\vlat{E}]$ is an ideal in $\vlat{F}$ there exists $w\in \vlat{E}$ so that $T(w) = v$.  Let $w'= (w\vee 0)\wedge u$.  Then $0\leq w'\leq u$ and $T(w') = (v\vee 0)\wedge T(u) = v$.
\end{proof}

\begin{prop}\label{Prop: Properties of band projections}
Let $\vlat{E}$ be a vector lattice, $\vlat{A}$ and $\vlat{B}$ projection bands in $\vlat{E}$, $P_\vlat{A}$ and $P_{\vlat{B}}$ the band projections of $\vlat{E}$ onto $\vlat{A}$ and $\vlat{B}$, respectively, and $I_{\vlat{E}}$ the identity operator on $\vlat{E}$.  Assume that $\vlat{A}\subseteq \vlat{B}$. The following statements are true.
\begin{enumerate}
	\item[(i)] $P_\vlat{A}$ is an order continuous lattice homomorphism.
	\item[(ii)] $P_\vlat{A} \leq I_{\vlat{E}}$.
	\item[(iii)] $P_\vlat{A} P_\vlat{B} = P_\vlat{B} P_\vlat{A} = P_\vlat{A}$.
	\item[(iv)] $P_\vlat{A}$ is interval preserving.
\end{enumerate}
\end{prop}

\begin{proof}
For (i), see \cite[Theorem 24.6 \& Exercise 24.11]{LuxemburgZaanen1971RSI}. For (ii) and (iii), see \cite[Theorems 24.5 (ii) \& 30.1 (i)]{LuxemburgZaanen1971RSI}. Lastly, (iv) follows from Proposition \ref{Prop:  Interval Preserving vs Lattice Homomorphism} (ii), since $P_{\vlat{A}}[\vlat{E}]=\vlat{A}$ is a band, hence an ideal, in $\vlat{E}$.
\end{proof}


The \emph{order dual} of $\vlat{E}$ is $\orderdual{\vlat{E}}\defeq \{\varphi: \vlat{E}\to\R ~:~ \varphi \text{ is order bounded}\}$, and the order continuous dual of $\vlat{E}$ is $\ordercontn{\vlat{E}} \defeq \{\varphi\in \orderdual{\vlat{E}} ~:~ \varphi \text{ is order continuous}\}$.  If $A\subseteq \vlat{E}$ and $B\subseteq \orderdual{\vlat{E}}$ we set
\[
\ann{A} \defeq \{\varphi\in \orderdual{\vlat{E}} ~:~ \varphi(u)=0,~u\in A\},~~ \preann{B} \defeq \{u\in\vlat{E} ~:~ \varphi(u) = 0,~ \varphi\in B\}.
\]
For $\varphi\in \orderdual{\vlat{E}}$ the \emph{null ideal} (or absolute kernel) of $\varphi$ is
\[
\nullid{\varphi} \defeq \{u\in \vlat{E} ~:~ |\varphi|(|u|)=0\}.
\]
The \emph{carrier} of $\varphi$ is $\carrier{\varphi} \defeq \nullid{\varphi}^d$.  The null ideal $\nullid{\varphi}$ of $\varphi$ is an ideal in $\vlat{E}$ and its carrier $\carrier{\varphi}$ is a band; if $\varphi$ is order continuous then $\nullid{\varphi}$ is also a band in $\vlat{E}$, see for instance \cite[\S 90]{Zaanen1983RSII}.

Define $\sigma:\vlat{E}\ni u\mapsto \Psi_u\in \ordercontnbidual{\vlat{E}}$ by setting $\Psi_u(\varphi)\defeq \varphi(u)$ for all $u\in\vlat{E}$ and $\varphi\in \ordercontn{\vlat{E}}$.  Then $\sigma$ is a lattice homomorphism, and, if $\preann{\ordercontn{\vlat{E}}}=\{0\}$, $\sigma$ is injective, see \cite[p.~404~-~405]{Zaanen1983RSII}.   We call $\vlat{E}$ \emph{perfect} if $\sigma$ is a lattice isomorphism onto $\ordercontnbidual{\vlat{E}}$.

In the following theorem, we briefly recall some basic facts concerning the order adjoint of a positive operator $T:E\to F$ which we make use of in the sequel.


\begin{thm}\label{Thm:  Adjoints of interval preserving vs lattice homomorphisms}
Let $\vlat{E}$ and $\vlat{F}$ be vector lattices and $T:\vlat{E}\to\vlat{F}$ a positive operator. Denote by $T^\sim :\vlat{F}^\sim \to \vlat{E}^\sim$ its order adjoint, $\varphi\mapsto \varphi\circ T$.  The following statements are true. \begin{itemize}
    \item[(i)] $T^\sim$ is positive and order continuous.
    \item[(ii)] If $T$ is order continuous then $T^\sim[\ordercontn{\vlat{F}}]\subseteq \ordercontn{\vlat{E}}$.
    \item[(iii)] If $T$ is interval preserving then $T^\sim$ is a lattice homomorphism.
    \item[(iv)] If $T$ is a lattice homomorphism then $T^\sim$ is interval preserving.  The converse is true if $\preann{\orderdual{\vlat{F}}}=\{0\}$.
\end{itemize}
\end{thm}

\begin{proof}
For (i), see \cite[14.2 \& 14.5]{Kaplan1985}. The statement in (ii) follows directly from the fact that composition of order continuous operators is order continuous. For (iii), see \cite[14.13]{Kaplan1985}. The first statement in (iv) is proven in \cite[Theorem 2.16 (1)]{AliprantisBurkinshaw2006}. The second statement is proven in \cite[Theorem~2.20]{AliprantisBurkinshaw2006}. We note that although \cite{AliprantisBurkinshaw2006} declares a blanket assumption at the start of the book that all vector lattices under consideration are Archimedean, the proofs of \cite[Theorems 2.16 \& 2.20]{AliprantisBurkinshaw2006} do not make use of this assumption.
\end{proof}

\begin{prop}\label{Prop:  Image of adjoint of lattice homomorphism}
Let $\vlat{E}$ and $\vlat{F}$ be vector lattices and $T:\vlat{E}\to\vlat{F}$ a linear lattice homomorphism onto $\vlat{F}$.  The following statements are true. \begin{itemize}
    \item[(i)] $T^\sim[\vlat{F}^\sim] = \ann{\ker(T)}$.
    \item[(ii)] If $\vlat{E}$ is Archimedean and $T$ is order continuous then $T^\sim[\ordercontn{\vlat{F}}] = \ann{\ker(T)}\cap\ordercontn{\vlat{E}}$.
\end{itemize}
\end{prop}

\begin{proof}[Proof of (i)]
Let $\varphi\in\vlat{F}^\sim$.  If $u\in\ker(T)$ then $T^\sim(\varphi) (u) = \varphi (T(u)) = \varphi(0) = 0$.  Hence $T^\sim(\varphi)\in\ann{\ker(T)}$.  For the reverse inclusion, let $\psi\in \ann{\ker(T)}$.  Define $\varphi:\vlat{F}\to\R$ by setting $\varphi(v) = \psi(u)$ if $v=T(u)$.  Then $\varphi\in \vlat{F}^\sim$ and $T^\sim(\varphi) = \psi$.
\end{proof}

\begin{proof}[Proof of (ii)]
It follows from (i) and Theorem \ref{Thm:  Adjoints of interval preserving vs lattice homomorphisms} (ii) that $T^\sim[\ordercontn{\vlat{F}}] \subseteq \ann{\ker(T)}\cap\ordercontn{\vlat{E}}$.  We show that if $T^\sim(\varphi) \in \ordercontn{\vlat{E}}$ for some $\varphi\in \vlat{F}^\sim$ then $\varphi\in \ordercontn{\vlat{F}}$.  From this and (i) it follows that $T^\sim[\ordercontn{\vlat{F}}] = \ann{\ker(T)}\cap\ordercontn{\vlat{E}}$.  We observe that it suffices to consider positive $\varphi\in\vlat{F}^\sim$.  Indeed, $T$ is a surjective lattice homomorphism and therefore, by Proposition \ref{Prop:  Interval Preserving vs Lattice Homomorphism} (ii), also interval preserving.  Hence by Theorem \ref{Thm:  Adjoints of interval preserving vs lattice homomorphisms} (iii), $T^\sim$ is a lattice homomorphism.

Suppose that $0\leq \varphi\in\vlat{F}^\sim$ and that $T^\sim(\varphi) \in \ordercontn{\vlat{E}}$.  Let $A\downarrow 0$ in $\vlat{F}$.  Define $B\defeq T^{-1}[A]\cap \vlat{E}^+$.  Then $B$ is downward directed and $T[B]=A$.  In particular, $\varphi[A] = T^\sim(\varphi)[B]$.  Let $C\defeq \{w\in \vlat{E} ~:~ 0\leq w\leq v \text{ for all } v\in B\}$.  If $w\in C$ then $0\leq T(w) \leq u$ for all $u\in A$ so that $T(w) = 0$.  Hence $C\subseteq \ker(T)$.  Since $\vlat{E}$ is Archimedean, we have $B-C\downarrow 0$ in $\vlat{E}$, see \cite[Theorem~$22.5$]{LuxemburgZaanen1971RSI}.  Since $T^\sim(\varphi)$ is order continuous, $T^\sim(\varphi)[B-C]\downarrow 0$; that is, for every $\epsilon>0$ there exists $v\in B$ and $w\in C$ so that $\varphi(T(v)) = \varphi (T(v-w)) = T^\sim(\varphi)(v-w)<\epsilon$.  Hence, for every $\epsilon>0$ there exists $u\in A$ so that $\varphi(u)<\epsilon$.  This shows that $\varphi[A]\downarrow 0$ so that $\varphi\in \ordercontn{\vlat{F}}$ as required.
\end{proof}Let $I$ be a non-empty set and let $\vlat{E}_\alpha$ be a vector lattice for every $\alpha\in I$. Then $\dsprod_{\alpha\in I} \vlat{E}_\alpha$ is a vector lattice with respect to the coordinate-wise operations.  If the index set is clear form the context, we omit it and write $\dsprod \vlat{E}_\alpha$.  For $\beta\in I$ let $\pi_\beta:\dsprod\vlat{E}_\alpha\to \vlat{E}_\beta$ be the coordinate projection onto $\vlat{E}_\beta$ and $\iota_\beta:\vlat{E}_\beta\to\dsprod \vlat{E}_\alpha$ the right inverse of $\pi_\beta$ given by
\[
\pi_\alpha(\iota_\beta(u))= \left\{\begin{array}{lll}
u & \text{if} & \alpha=\beta \smallskip \\
0 & \text{if} & \alpha\neq \beta.\\
\end{array}\right.
\]
We denote by $\displaystyle\bigoplus \vlat{E}_\alpha$ the ideal in $\dsprod \vlat{E}_\alpha$ consisting of $u\in \dsprod \vlat{E}_\alpha$ for which $\pi_\alpha(u)\neq 0$ for only finitely many $\alpha\in I$.  The following properties of $\dsprod \vlat{E}_\alpha$ and $\displaystyle \bigoplus \vlat{E}_\alpha$ are used frequently in the sequel.

\begin{thm}\label{Thm:  Properties of product of vector lattices.}
Let $I$ be a non-empty set and $\vlat{E}_\alpha$ a vector lattice for every $\alpha\in I$.  The following statements are true. \begin{itemize}
    \item[(i)] The coordinate projections $\pi_\beta$ and their right inverses $\iota_\beta$ are normal, interval preserving lattice homomorphisms.\smallskip
    \item[(ii)] $\dsprod \vlat{E}_\alpha$ is Archimedean if and only if each $\vlat{E}_\alpha$ is Archimedean.\smallskip
    \item[(iii)] $\dsprod \vlat{E}_\alpha$ is Dedekind complete if and only if each $\vlat{E}_\alpha$ is Dedekind complete.\smallskip
    \item[(iv)] If $I$ has non-measurable cardinal, then the order dual of $\dsprod \vlat{E}_\alpha$ is $\displaystyle \bigoplus \vlat{E}_\alpha^\sim$.\smallskip
    \item[(v)] The order continuous dual of $\dsprod\vlat{E}_\alpha$ is $\displaystyle \bigoplus \ordercontn{(\vlat{E}_\alpha)}$.\smallskip
    \item[(vi)] The order dual of $\displaystyle \bigoplus \vlat{E}_\alpha$ is $\dsprod \vlat{E}_\alpha^\sim$.\smallskip
    \item[(vii)] The order continuous dual of $\displaystyle \bigoplus \vlat{E}_\alpha$ is $\dsprod \ordercontn{\left( \vlat{E}_\alpha \right)}$.
\end{itemize}
\end{thm}

We leave the straightforward proofs of (i), (ii), (iii), (vi) and (vii) to the reader.

\begin{proof}[Proof of (iv)]
Assume that $I$ has non-measurable cardinal.  By (i) of this theorem and Theorem \ref{Thm:  Adjoints of interval preserving vs lattice homomorphisms} (iii) and (iv), $\iota_\beta^\sim : \orderdual{\left(\dsprod \vlat{E}_\alpha\right)}\to \orderdual{\vlat{E}}_\beta$ is an interval preserving normal lattice homomorphism for every $\beta\in I$.  Because each $\varphi\in \orderdual{\left(\dsprod \vlat{E}_\alpha\right)}$ is linear and order bounded, the set $I_\varphi \defeq \{\beta \in I ~:~ \iota_\beta^\sim(\varphi)\neq 0 \}$ is finite for every $\varphi\in \orderdual{\left(\dsprod \vlat{E}_\alpha\right)}$.  Define $S:\orderdual{\left(\dsprod \vlat{E}_\alpha\right)} \to \displaystyle \bigoplus \orderdual{\vlat{E}}_\alpha$ by setting
\[
S(\varphi) \defeq (\iota_\alpha^\sim(\varphi))_{\alpha\in I},~~ \varphi\in \orderdual{\left(\dsprod \vlat{E}_\alpha\right)}.
\]Then $S$ is a lattice homomorphism. It remains to verify that $S$ is bijective.

We show that $S$ is injective.  Let $0\neq \varphi \in \orderdual{\left(\dsprod \vlat{E}_\alpha\right)}$.  Fix $0\leq u\in\dsprod\vlat{E}_\alpha$ so that $\varphi(u)\neq 0$.  For $f\in \R^I$ let $fu\in\dsprod \vlat{E}_\alpha$ be defined by $\pi_\alpha (fu) = f(\alpha)\pi_\alpha (u)$, $\alpha\in I$.  Define $\hat\varphi : \R^I\to \R$ by setting
\[
\hat \varphi (f) \defeq \varphi(fu),~~ f\in \R^I.
\]
Then $\hat\varphi$ is a non-zero order bounded linear functional on $\R^I$.  Because $I$ has nonmeasurable cardinal, $I$ equipped with the discrete topology is realcompact, see \cite[\S 12.2]{GillmanJerison1960}.  Therefore there exists a non-zero finitely supported and countably additive measure $\mu$ on the powerset $2^I$ of $I$ so that
\[
\hat\varphi (f) = \int_I f \thinspace d\mu = \sum_{\alpha \in I} f(\alpha)\mu(\alpha),~~ f\in \R^I,
\]
see \cite[Theorem 4.5]{GouldMahowald1962}.  Let $\alpha$ be in the support of $\mu$, and let $g$ be the indicator function of $\{\alpha\}$.  Then $0\neq \mu(\alpha)=\hat \varphi(g) = \varphi(gu) = \iota_\alpha^\sim(\varphi)(\pi_\alpha(u))$.  Therefore $S(\varphi)\neq 0$ so that $S$ is injective.

To see that $S$ is surjective, observe that for every $\beta\in I$, $\pi_\beta^\sim :\orderdual{\vlat{E}}_\beta\to \orderdual{\left(\dsprod \vlat{E}_\alpha\right)}$ is an interval preserving normal lattice homomorphism by (i) of this theorem and Theorem \ref{Thm:  Adjoints of interval preserving vs lattice homomorphisms} (iii) and (iv).  Define $T: \displaystyle \bigoplus \orderdual{\vlat{E}}_\alpha\to \orderdual{\left(\dsprod \vlat{E}_\alpha\right)}$ by setting
\[
T(\psi) \defeq \sum \pi_\alpha^\sim(\psi_\alpha),~~ \psi=(\psi_\alpha)\in \bigoplus \orderdual{\vlat{E}}_\alpha.
\]
Then $T$ is a positive operator.  We claim that $S\circ T$ is the identity on $\displaystyle \bigoplus \orderdual{\vlat{E}}_\alpha$.  Indeed, for any $\psi\in \displaystyle \bigoplus \orderdual{\vlat{E}}_\alpha$ we have
\[
S(T(\psi)) = \sum_{\alpha\in I}(\iota_\beta^\sim(\pi_\alpha^\sim(\psi_\alpha)))_{\beta\in I} = \sum_{\alpha\in I}(\psi_\alpha\circ \pi_\alpha\circ\iota_\beta )_{\beta\in I}.
\]By definition of the $\iota_\beta$ it follows that $S(T(\psi)) = \psi$ which verifies our claim. Therefore $S$ is a lattice isomorphism.
\end{proof}

\begin{proof}[Proof of (v)]
Define $S:\orderdual{\left(\dsprod \vlat{E}_\alpha\right)} \to \displaystyle \bigoplus \orderdual{\vlat{E}}_\alpha$ as in the proof of (iv).  By (i) of this theorem and Theorem \ref{Thm:  Adjoints of interval preserving vs lattice homomorphisms} (ii), $S$ maps $\ordercontn{\left(\dsprod \vlat{E}_\alpha\right)}$ into $\displaystyle \bigoplus \ordercontn{(\vlat{E}_\alpha)}$.  A similar argument to that given in the proof of (iv) shows that $S$ is a surjective lattice homomorphism.  Hence it remains to show that $S$ is injective.

Let $0\leq \varphi \in \ordercontn{\left(\dsprod \vlat{E}_\alpha\right)}$ and suppose that $S(\varphi) = 0$.  Then $\iota_\beta^\sim(\varphi) = 0$ for every $\beta\in I$.  But for any $0\leq u \in \dsprod\vlat{E}_\alpha$,
\[
u = \sup\left\lbrace \sum_{\alpha \in F}\iota_\alpha(u) ~:~F\subseteq I \text{ is finite} \right\rbrace.
\]
Therefore by the order continuity of $\varphi$,
\[
\varphi(u) = \sup \left\lbrace \sum_{\alpha \in F}\iota_\alpha^\sim (\varphi)(u) ~:~ F\subseteq I \text{ is finite}\right\rbrace = 0
\]
for all $0\leq u \in \dsprod\vlat{E}_\alpha$; hence $\varphi=0$.  Because $S$ is a lattice homomorphism it follows that, for all $\varphi\in \ordercontn{\left(\dsprod \vlat{E}_\alpha\right)}$, if $S(\varphi) = 0$ then $\varphi=0$; that is, $S$ is injective.
\end{proof}

\begin{remark}
We note that, in general, the statement in Theorem 2.5 (iv) is not true if $I$ has measurable cardinal:  In this case the map $S$ in the proof of Theorem 2.5 (iv) may fail to be injective.  To see this, suppose that $I$ has measurable cardinal.  Then $I$ equipped with the discrete topology is not realcompact.  We identify $\R^I$ with $\cont (\upsilon I)$. Let $x\in \upsilon I\setminus I$.  Then $\delta_x :\R^I \ni u\mapsto u(x)\in \R$ is a non-zero, positive linear functional on $\R^I$, but $S(\delta_x)=0$.
\end{remark}


We now define the categories which are the setting of this paper.  It is readily verified that these are indeed categories.

\begin{table}[H]

\begin{tabular}{ |l|l|l| }
\hline
${}$ & \quad\textsc{Objects}\quad & \quad\textsc{Morphisms}\quad \\
\hline
$\vlc$ & Vector lattices & Lattice homomorphisms\\
\hline
$\nvlc$ & Vector lattices & Normal lattice homomorphisms\\
\hline
$\vlic$ & Vector lattices & Interval preserving lattice homomorphisms\\
\hline
$\nvlic$ & Vector lattices & Normal, interval preserving lattice homomorphisms\\
\hline
\end{tabular}

\end{table}

We refer to these four categories as \emph{categories of vector lattices}.  If ${\bf C}$ is a category of vector lattices, then a ${\bf C}$-morphism is a morphism within the category ${\bf C}$. Below we depict the subcategory relationships among the categories of vector lattices under consideration.

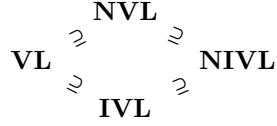
\begin{figure}[H]

\begin{tikzcd}[row sep = 4pt , column sep = 7pt]
{}   & \nvlc \arrow[symbol=\supseteq]{rd} & {} \\
\vlc \arrow[symbol=\supseteq]{ru} \arrow[symbol=\supseteq]{rd} & {} & \nvlic \\
{}   & \vlic \arrow[symbol=\supseteq]{ru} & {}\\
\end{tikzcd}
\end{figure}

\subsection{Measures on topological spaces}\label{Subsection:  Realcompact spaces preliminaries}

Because the terminology related to measures on topological spaces varies across the literature, we declare our conventions.  Let $X$ be a Hausdorff topological space.  For a function $u:X\to\R$ we denote by $\zeroset{u}$ the \emph{zero set} of $u$ and by $\cozeroset{u}$ its \emph{co-zero set}, that is, the complement of $\zeroset{u}$.  If $A\subseteq X$ then $\onefunction_A$ denotes the indicator function of $A$.

Denote by $\borel{X}$ the Borel $\sigma$-algebra generated by the open sets in $X$. A \emph{(signed) Borel measure} on $X$ is a real-valued and $\sigma$-additive function on $\borel{X}$.  We denote the space of all signed Borel measures on $X$ by $\vlat{M}_\sigma(X)$.  This space is a Dedekind complete vector lattice with respect to the pointwise operations and order \cite[Theorem 27.3]{Zaanen1997Introduction}.  In particular, for $\mu,\nu\in \vlat{M}_\sigma(X)$,
\[
(\mu\vee \nu)(B) = \sup\left\lbrace \mu(A)+\nu(B\setminus A) ~:~ A\subseteq B,~~ A \in \borel{X} \right\rbrace,~~ B \in \borel{X}.
\]For any upward directed set $D\subseteq \vlat{M}_\sigma(X)^+$ with $\sup D=\nu$ in $\vlat{M}_\sigma(X)$,
\begin{eqnarray}
\nu(B) = \sup\{\mu(B) ~:~ \mu\in D\},~~ B \in \borel{X}.\label{EQ:  Sup of set of measures}
\end{eqnarray}
Following Bogachev \cite{Bogachev2007}, we call a Borel measure $\mu$ on $X$ a \emph{Radon measure} if for every $B \in \borel{X}$,
\[
|\mu|(B)=\sup\{|\mu|(K) ~:~ K\subseteq B \text{ is compact}\}.
\]
Equivalently, $\mu$ is Radon if for every $B \in \borel{X}$ and every $\epsilon>0$ there exists a compact set $K\subseteq B$ so that $|\mu|(B\setminus K)<\epsilon$.  Observe that if $\mu$ is Radon, then also
\[
|\mu|(B)=\inf\{ |\mu|(U) ~:~ U\supseteq B \text{ is open}\}.
\]
Denote the space of Radon measures on $X$ by $\vlat{M}(X)$.

Recall that the \emph{support} of a Borel measure $\mu$ on $X$ is defined as
\[
S_\mu \defeq \{x\in X ~:~ |\mu|(U)>0 \text{ for all } U\ni x \text{ open}\}.
\]
A non-zero Borel measure $\mu$ may have empty support, and even if $S_\mu\neq \emptyset$, it may have measure zero \cite[Vol. II, Example 7.1.3]{Bogachev2007}.  However, if $\mu$ is a non-zero Radon measure, then $S_\mu \neq \emptyset$ and $|\mu|(S_\mu)=|\mu|(X)$; in fact, for every $B \in \borel{X}$, $|\mu|(B)=|\mu|(B\cap S_\mu)$.  We list the following useful properties of the support of a measure.  These are well known for measures on locally compact spaces, see for instance \cite[Chapter 4]{DalesDashiellLauStrass2016}, with proofs that also apply in our setting.  The proofs are therefore omitted.

\begin{prop}\label{Prop:  Properties of support of a measure}
Let $\mu$ and $\nu$ be Radon measures on $X$.  The following statements are true. \begin{enumerate}
    \item[(i)] If $|\mu|\leq |\nu|$ then $S_\mu\subseteq S_\nu$.
    \item[(ii)] $S_{\mu+\nu}\subseteq S_{|\mu|+|\nu|}$
    \item[(iii)] $S_{|\mu|+|\nu|} = S_\mu \cup S_\nu$.
\end{enumerate}
\end{prop}

A Radon measure $\mu$ is called \emph{compactly supported} if $S_\mu$ is compact.  We denote the space of all compactly supported Radon measures on $X$ as $\vlat{M}_c(X)$. Further, a Radon measure $\mu$ on $X$ is called a \emph{normal measure} if $|\mu|(L)=0$ for all closed nowhere dense sets $L$ in $X$.  The space of all normal Radon measures on $X$ is denoted $\vlat{N}(X)$, and the space of compactly supported normal Radon measures by $\vlat{N}_c(X)$.

\begin{thm}
The following statements are true. \begin{enumerate}
    \item[(i)] $\vlat{M}(X)$ is an band in $\vlat{M}_\sigma(X)$
    \item[(ii)] $\vlat{M}_c(X)$ is an ideal in $\vlat{M}(X)$.
    \item[(iii)] $\vlat{N}(X)$ is a band in $\vlat{M}(X)$.
    \item[(iv)] $\vlat{N}_c(X)$ is a band in $\vlat{M}_c(X)$.
\end{enumerate}
\end{thm}
\begin{proof}
For the proof of (i), let $\mu,\nu\in\vlat{M}(X)$.  Consider a Borel set $B$ and a real number $\epsilon>0$.  There exists a compact set $K\subseteq B$ so that $|\mu|(B\setminus K)<\epsilon/2$ and $|\nu|(B\setminus K)<\epsilon/2$.  We have $|\mu+\nu|(B\setminus K)\leq |\mu|(B\setminus K) + |\nu|(B\setminus K)<\epsilon$.  Therefore $\mu+\nu \in \vlat{M}(X)$.  A similar argument shows that $a\mu\in \vlat{M}(X)$ for all $a\in\R$.  It also follows in this way that for all $\nu\in\vlat{M}_\sigma(X)$ and $\mu\in \vlat{M}(X)$, if $|\nu|\leq |\mu|$ then $\nu\in \vlat{M}(X)$.  By definition of a Radon measure, $|\mu|\in\vlat{M}(X)$ whenever $\mu\in\vlat{M}(X)$.  Therefore $\vlat{M}(X)$ is an ideal in $\vlat{M}_\sigma(X)$.

To see that $\vlat{M}(X)$ is a band in $\vlat{M}_\sigma(X)$, consider an upward directed subset $D$ of $\vlat{M}(X)^+$ so that $\sup D=\nu$ in $\vlat{M}_\sigma(X)$.  Fix a Borel set $B$ and a real number $\epsilon>0$.  There exists $\mu\in D$ so that $\nu(B)-\epsilon/2<\mu(B)$.  But $\mu$ is a Radon measure, so there exists a compact subset $K$ of $B$ so that $\mu(K)>\mu(B)-\epsilon/2$.  Therefore $\nu(K)\geq \mu(K)>\mu(B)-\epsilon/2> \nu(B)-\epsilon$.  Therefore $\nu \in \vlat{M}(X)$ so that $\vlat{M}(X)$ is a band in $\vlat{M}_\sigma(X)$.

The statement in (ii) follows immediately from the definition of the support of a measure and Proposition \ref{Prop:  Properties of support of a measure}. It is clear that $\vlat{N}(X)$ is an ideal in $\vlat{M}(X)$, and that it is a band follows from (\ref{EQ:  Sup of set of measures}).  Hence (iii) is true.  That (iv) is true follows immediately from (ii) and (iii).
\end{proof}

Unsurprisingly, there is a close connection between Radon measures on $X$ and order bounded linear functionals on $\contX$.  Theorem \ref{Thm:  Riesz Representation Theorem for C(X)} to follow is implicit in \cite[Corollary 1 (p. 106) \& Theorems 4.2, 4.5]{GouldMahowald1962}, see also \cite{Hewitt1950} where a treatment is given in terms of Baire measures.  In order to facilitate the discussion of order continuous functionals to follow, we include the proof.

\begin{thm}\label{Thm:  Riesz Representation Theorem for C(X)}
Let $X$ be a realcompact space.  There is a lattice isomorphism $\orderdual{\contX}\ni \varphi \longmapsto \mu_\varphi \in \vlat{M}_c(X)$ so that for every $\varphi\in \orderdual{\contX}$,
\[
\varphi(u) = \int_X u \thinspace d\mu_\varphi,~~ u\in\contX.
\]
\end{thm}

\begin{proof}
We identify the space $\contbX$ with $\cont\left(\beta X\right)$. Because $\contbX$ is an ideal in $\contX$, the restriction map from $\orderdual{\contX}$ to $\orderdual{\contbX}$ is a lattice homomorphism \cite[Section 1.3, Exercise 1]{AliprantisBurkinshaw2006}.  It follows from \cite[Theorem 1]{Hewitt1950} that this map is injective.  Thus by the Riesz Representation Theorem \cite[Theorem 18.4.1]{Semadeni1971}, for every $\varphi\in\orderdual{\contX}$ there exists a unique Radon measure $\nu_\varphi$ on $\beta X$ so that
\[
\varphi(u) = \int_{\beta X}u \thinspace d\nu_\varphi,~~ u\in\contbX.
\]
Furthermore, the map $\varphi\mapsto \nu_\varphi$ is a lattice isomorphism onto its range.

We claim that the range of this map is $\vlat{M}_0(\beta X)\defeq \{\nu \in \vlat{M}(\beta X) ~:~ S_\nu \subseteq X\}$.  According to \cite[Theorem 4.4]{GouldMahowald1962}, $S_{\nu_\varphi}\subseteq X$ for every $\varphi\in \orderdual{\contX}$.  Hence $\nu_\varphi\in \vlat{M}_0(\beta X)$.  Conversely, let $\nu\in \vlat{M}_0(\beta X)$.  Since $S_\nu \subseteq X$ is compact in $\beta X$, hence also in $X$,
\[
\psi(u) \defeq \int_{S_{\nu}} u \thinspace d\nu,~~ u\in\contX
\]
defines an order bounded functional on $\contX$.  For every $u\in \contbX$ we have
\[
\int_{\beta X} u \thinspace d\nu  = \int_{S_{\nu}} u \thinspace d\nu   = \psi(u) = \int_{\beta X}u \thinspace d\nu_\varphi.
\]
Therefore $\nu = \nu_\psi$ which establishes our claim.

We have shown that $\orderdual{\contX}\ni \varphi\mapsto \nu_\varphi \in \vlat{M}_0(\beta X)$ is a lattice isomorphism.  We now show that $\vlat{M}_0(\beta X)$ is isomorphic to $\vlat{M}_c(X)$.

Let $\nu\in \vlat{M}_0(\beta X)$.  The Borel sets in $X$ are precisely the intersections with $X$ of Borel sets in $\beta X$ \cite[p. 108]{GouldMahowald1962}.  Furthermore, if $B', B'' \in \borel{\beta X}$ so that $B'\cap X = B''\cap X$ then $\nu(B')=\nu(B'\cap S_\nu)=\nu(B''\cap S_\nu)=\nu(B'')$.  For $B \in \borel{X}$ define
\[
\nu^\ast (B) \defeq \nu(B') \text{ with }B'\in \borel{\beta X} \text{ so that }B'\cap X=B.
\]
It follows from the previous observation that $\nu^\ast$ is well-defined.  It follows easily that $\nu^\ast\in \vlat{M}_c(X)$, and that the map $\vlat{M}_0(\beta X)\ni \nu\mapsto \nu^\ast\in \vlat{M}_c(X)$ is injective, linear, and bipositive.  Let $\mu\in \vlat{M}_c(X)$.  For every $B\in \borel{\beta X}$ let $\nu(B) \defeq \mu(B\cap X)$.  Then $\nu\in \vlat{M}_0(\beta X)$ and $\nu^\ast = \mu$.  Therefore $\vlat{M}_0(\beta X)\ni \nu\mapsto \nu^\ast\in \vlat{M}_c(X)$ is a lattice isomorphism.

For $\varphi \in \orderdual{\contX}$ let $\mu_\varphi \defeq (\nu_\varphi)^\ast$.  Then $\orderdual{\contX}\ni \varphi\mapsto \mu_\varphi \in \vlat{M}_c(X)$ is a lattice isomorphism. It remains to show that, for every $\varphi \in \orderdual{\contX}$,
\[
\varphi(u) = \int_{X} u \thinspace d\mu_\varphi,~~ u\in\contX.
\]
Fix $0\leq \varphi \in \orderdual{\contX}$ and $u\in\contX^+$.  A minor modification of the proof of \cite[Theorem 3.1]{GouldMahowald1962} shows that there exists a natural number $N$ so that $\varphi( u )=\varphi( u \wedge n\onefunction_X)$ for every $n\geq N$.  But
\[
\int_X u \thinspace d\mu_\varphi = \sup_{n\in\N} \int_X u \wedge n\onefunction_X \thinspace d\mu_\varphi,
\]
and, for every $n\in\N$,
\[
\int_X u \wedge n\onefunction_X \thinspace d\mu_\varphi = \int_{\beta X} u \wedge n\onefunction_X \thinspace d\nu_\varphi = \varphi(u\wedge n\onefunction_X).
\]
Therefore
\[
\varphi(u) = \int_X u \thinspace d\mu_\varphi,
\]
as desired.
\end{proof}

\begin{thm}\label{Thm:  Order continuous functionals on C(X) are normal measures}
Let $X$ be a realcompact space.  Let $\varphi$ be an order bounded functional on $\contX$.  Then $\varphi$ is order continuous if and only if $\mu_\varphi$ is a normal measure.  The map
\[
\ordercontn{\contX}\ni \varphi \longmapsto \mu_{\varphi} \in \vlat{N}_c(X)
\]is a lattice isomorphism onto $\vlat{N}_c(X)$.
\end{thm}
\begin{proof}
We make use of the notation introduced in the proof of Theorem \ref{Thm:  Riesz Representation Theorem for C(X)}.  It suffices to show that for any $0\leq \varphi\in \orderdual{\contX}$, $\varphi$ is order continuous if and only if $\mu_\varphi$ is normal.  Let $0\leq \varphi \in \ordercontn{\contX}$.  Because $\contbX$ is an ideal in $\contX$ the restriction of $\varphi$ to $\contb(X)$ is order continuous.  Hence the measure $\nu_\varphi\in\vlat{M}_0(\beta X)$ so that
\[
\varphi(u)=\int_{\beta X}u \thinspace d\nu_\varphi,~~ u\in\contb(X)
\]
is a normal measure on $\beta X$, see for instance \cite[Definition 4.7.1, Theorem 4.7.4]{DalesDashiellLauStrass2016}.  It therefore follows that the measure $\mu_\varphi = (\nu_\varphi)^\ast \in\vlat{M}_c(X)$ is a normal measure on $X$.

Conversely, let $0\leq \varphi \in \orderdual{\contX}$ be such that $\mu_\varphi$ is a normal measure on $X$.  Then the Borel measure $\nu$ on $\beta X$ given by
\[
\nu(B) = \mu_\varphi(B\cap X),~~ B\in \borel{\beta X}
\]is a normal measure on $\beta X$.  Hence $S_{\nu}$ is regular-closed in $\beta X$, see \cite[Proposition 4.7.9]{DalesDashiellLauStrass2016}.  But $S_\nu = S_{\mu_\varphi}\subseteq X$ so that $S_{\mu_\varphi}$ is regular-closed in $X$.  Therefore, if $D\downarrow 0$ in $\contX$ then $\left.D\right|_{S_{\mu_\varphi}} = \{  \left. u \right|_{S_{\mu_\varphi}} ~:~ u\in D \} \downarrow 0$ in $\cont(S_{\mu_\varphi})$, see \cite[Theorem 3.4]{KandicVavpeticPositivity2019}.  Also, $\mu_\varphi$ restricted to the Borel sets in $S_{\mu_\varphi}$ is a normal measure on $S_{\mu_\varphi}$.  Hence
\[
\inf_{u\in D} \varphi(u) = \inf_{u\in D}\int_{S_{\mu_\varphi}}u \thinspace d\mu_\varphi = 0.
\]
Therefore $\varphi$ is order continuous.
\end{proof}

\section{Direct limits}\label{Section:  Inductive limits}

We recall the definitions of a direct system in a category of vector lattices, and of the direct limit of such a system.  These definitions are specializations of the corresponding definitions in general categories, see for instance \cite[Chapter 5]{Awodey2010} and \cite[Chapter III]{MacLane1998} where direct limits are referred to as \emph{colimits}. We summarise some existence results and list vector lattice properties that have permanence under the direct limit construction.  Additional results are found in \cite{Filter1988}.  Lastly, we give a number of examples of direct limits which we will make use of later.

\begin{defn}\label{Defn:  Inductive system}
Let ${\bf C}$ be a category of vector lattices, $I$ a directed set, $\vlat{E}_\alpha$ a vector lattice for each $\alpha\in I$, and $e_{\alpha, \beta}: \vlat{E}_\alpha\to \vlat{E}_\beta$ a ${\bf C}$-morphism for all $\alpha \precc \beta$ in $I$.  The ordered pair $\cal{D} \defeq \left( (\vlat{E}_\alpha)_{\alpha\in I}, (e_{\alpha, \beta})_{\alpha\precc\beta}\right)$ is called a \emph{direct system} in ${\bf C}$ if, for all $\alpha \precc \beta \precc \gamma$ in $I$, the diagram
\[
\begin{tikzcd}[cramped]
\vlat{E}_\alpha \arrow[rd, "e_{\alpha, \beta}"'] \arrow[rr, "e_{\alpha, \gamma}"] & & \vlat{E}_\gamma\\
& \vlat{E}_\beta\arrow[ru, "e_{\beta, \gamma}"']
\end{tikzcd}
\]
commutes in ${\bf C}$.
\end{defn}

\begin{defn}\label{Defn:  Compatible system for inductive limits}
Let ${\bf C}$ be a category of vector lattices and $\cal{D} \defeq \left( (\vlat{E}_\alpha)_{\alpha\in I}, (e_{\alpha, \beta})_{\alpha\precc\beta}\right)$ a direct system in ${\bf C}$. Let $\vlat{E}$ be a vector lattice and for every $\alpha\in I$, let $e_\alpha: \vlat{E}_\alpha \to \vlat{E}$ be a ${\bf C}$-morphism.  The ordered pair $\cal{S} \defeq (\vlat{E}, (e_\alpha)_{\alpha\in I})$ is a  \emph{compatible system} of $\cal{D}$ in ${\bf C}$ if, for all $\alpha\precc\beta$ in $I$, the diagram
\[
\begin{tikzcd}[cramped]
\vlat{E}_\alpha \arrow[rd, "e_{\alpha, \beta}"'] \arrow[rr, "e_{\alpha}"] & & \vlat{E}\\
& \vlat{E}_\beta\arrow[ru, "e_{\beta}"']
\end{tikzcd}
\]
commutes in ${\bf C}$.
\end{defn}

\begin{defn}\label{Defn:  Inductive limit}
Let ${\bf C}$ be a category of vector lattices and $\cal{D} \defeq \left( (\vlat{E}_\alpha)_{\alpha\in I}, (e_{\alpha, \beta})_{\alpha\precc\beta}\right)$ a direct system in ${\bf C}$.  The \emph{direct limit} of $\cal{D}$ in ${\bf C}$ is a compatible system $\cal{S}\defeq(\vlat{E}, (e_\alpha)_{\alpha\in I})$ of $\cal{D}$ in ${\bf C}$ so that for any compatible system $\tilde{S} \defeq (\tilde{\vlat{E}}, (\tilde{e}_\alpha)_{\alpha\in I})$ of $\cal{D}$ in ${\bf C}$ there exists a unique ${\bf C}$-morphism $r: \vlat{E}\to \tilde{\vlat{E}}$ so that, for every $\alpha\in I$, the diagram
\[
\begin{tikzcd}[cramped]
\vlat{E} \arrow[rr, "r"] & & \tilde{\vlat{E}}\\
& \vlat{E}_\alpha \arrow[lu, "e_{\alpha}"] \arrow[ru, "\tilde{e}_{\alpha}"']
\end{tikzcd}
\]
commutes in ${\bf C}$.  We denote the direct limit of a direct system $\cal{D}$ by $\ind{\cal{D}}$ or $\ind{\vlat{E}_\alpha}$.
\end{defn}

Since the direct limit of a direct system is in fact an initial object in a certain derived category, it follows that the direct limit, when it exists, is unique up to a unique isomorphism, see for instance \cite[p. 54]{BucurDeleanu1968}.

\subsection{Existence and permanence properties of direct limits}\label{Subsection:  Existence and permanence properties of inductive limits}

Filter \cite{Filter1988} shows that any direct system $\cal{D} \defeq ((\vlat{E}_\alpha)_{\alpha\in I},(e_{\alpha , \beta})_{\alpha \precc\beta})$ in $\vlc$ has a direct limit in $\vlc$.\footnote{The results in \cite{Filter1988} are not formulated in these terms.}  In particular, the set-theoretic direct limit \cite[Chapter III, $\S7.5$]{Bourbakie_Theory_of_sets} of $\cal{D}$ equipped with suitable vector space and order structures is also the direct limit of $\cal{D}$ in $\vlc$.  We briefly recall the details.

For $u$ in the disjoint union $\biguplus \vlat{E}_\alpha$ of the collection $\left( \vlat{E}_\alpha \right)_{\alpha\in I}$, denote by $\alpha(u)$ that element of $I$ so that $u\in \vlat{E}_{\alpha(u)}$.  Define an equivalence relation on $\biguplus \vlat{E}_\alpha$ by setting $u\sim v$ if and only if there exists $\beta \scc \alpha (u),\alpha (v)$ in $I$ so that $e_{\alpha(u) , \beta}(u) = e_{\alpha(v) , \beta}(v)$.  Let $\vlat{E} \defeq  \biguplus \vlat{E}_\alpha / \sim$ and denote the equivalence class generated by $u\in \biguplus\vlat{E}_\alpha$ by $\dot u$.

Let $\dot u,\dot v\in \vlat{E}$.  We set $\dot u\leq \dot v$ if and only if there exists $\beta \scc \alpha (u),\alpha(v)$ in $I$ so that $e_{\alpha(u) , \beta}(u) \leq e_{\alpha(v) , \beta}(v)$.  Further, for $a,b\in\R$ define
\[
a\dot u + b\dot v \defeq \dot{\overbrace{a e_{\alpha(u) , \beta}(u) + b e_{\alpha(v) , \beta}(v)}},
\]
where $\beta \scc \alpha(u),\alpha(v)$ in $I$ is arbitrary.  With addition, scalar multiplication and the partial order so defined, $\vlat{E}$ is a vector lattice.  The lattice operations are given by
\[
	\dot u\wedge\dot v\ = \ \dot{\overbrace{e_{\alpha(u) , \beta}(u) \wedge e_{\alpha(v) , \beta}(v)}}
\]
and
\[
	\dot u\vee\dot v\ = \ \dot{\overbrace{e_{\alpha(u) , \beta}(u) \vee e_{\alpha(v) , \beta}(v)}},
\]with $\beta \scc \alpha(u),\alpha(v)$ in $I$ arbitrary.

For each $\alpha\in I$ define $e_\alpha : \vlat{E}_\alpha \to \vlat{E}$ by setting $e_\alpha (u) \defeq \dot u$ for $u\in \vlat{E}_\alpha$. Each $e_\alpha$ is a lattice homomorphism and the diagram
\[
\begin{tikzcd}[cramped]
\vlat{E}_\alpha \arrow[rd, "e_{\alpha, \beta}"'] \arrow[rr, "e_{\alpha}"] & & \vlat{E}\\
& \vlat{E}_\beta\arrow[ru, "e_{\beta}"']
\end{tikzcd}
\]
commutes in $\vlc$ for all $\alpha \precc\beta$ in $I$ so that $\cal{S}\defeq (\vlat{E},(e_\alpha)_{\alpha\in I})$ is a compatible system of $\cal{D}$ in $\vlc$.  Further, if $\tilde{\cal{S}} = (\tilde{\vlat{E}},(\tilde{e}_\alpha)_{\alpha\in I})$ is another compatible system of $\cal{D}$ in $\vlc$ then
\[
r: \vlat{E}\ni \dot u \longmapsto \tilde{e}_{\alpha(u)}(u)\in \tilde{\vlat{E}}
\]
is the unique lattice homomorphism so that the diagram
\[
\begin{tikzcd}[cramped]
\vlat{E} \arrow[rr, "r"] & & \tilde{\vlat{E}}\\
& \vlat{E}_\alpha \arrow[lu, "e_{\alpha}"] \arrow[ru, "\tilde{e}_{\alpha}"']
\end{tikzcd}
\]
commutes for every $\alpha\in I$.  Hence $\cal{S}$ is indeed the direct limit of $\cal{D}$ in $\vlc$.

We give two further existence results for direct limits of direct systems in other categories of vector lattices.

\begin{thm}\label{Thm:  Existence of Inductive Limits in VLI}
Let $\cal{D} \defeq \left( (\vlat{E}_\alpha)_{\alpha\in I}, (e_{\alpha, \beta})_{\alpha\precc\beta}\right)$ be a direct system in $\vlic$, and let $\cal{S} \defeq (\vlat{E},(e_\alpha)_{\alpha \in I})$ be the direct limit of $\cal{D}$ in $\vlc$. Then $\cal{S}$ is the direct limit of $\cal{D}$ in $\vlic$.
\end{thm}

\begin{proof}
We show that each $e_\alpha$ is interval preserving.  To this end, fix $\alpha\in I$ and $0\leq u\in \vlat{E}_\alpha$.  Suppose that $\dot 0 \leq \dot v \leq e_\alpha (u)=\dot u$.  Then there exists a $\beta \scc \alpha , \alpha(v)$ in $I$ so that $0\leq e_{\alpha(v) , \beta}(v) \leq e_{\alpha , \beta} (u)$.  But $e_{\alpha , \beta}$ is interval preserving, so there exists $0\leq w\leq u$ in  $\vlat{E}_\alpha$ so that $e_{\alpha , \beta}(w)=e_{\alpha(v) , \beta}(v)$.  Therefore $e_\alpha(w) = \dot w = \dot v$.  Hence $e_\alpha$ is interval preserving. Therefore $\cal{S}$ is a compatible system of $\cal{D}$ in $\vlic$.

Let $\tilde{\cal{S}}\defeq (\tilde{\vlat{E}},(\tilde{e}_\alpha)_{\alpha\in I})$ be a compatible system of $\cal{D}$ in $\vlic$, thus also in $\vlc$.  We show that the canonical lattice homomorphism $r:\vlat{E}\to \tilde{\vlat{E}}$ is interval preserving.  Consider $\dot u \in \vlat{E}^+$.  Let $0\leq v \leq r(\dot u)$ in $\tilde{\vlat{E}}$, that is, $0\leq v\leq \tilde{e}_{\alpha(u)}(u)$.  But $\tilde{e}_{\alpha(u)}$ is interval preserving so there exists $0\leq w\leq u$ in $\vlat{E}_{\alpha(u)}$ so that $v=\tilde{e}_{\alpha(u)}(w)$.  Thus $\dot 0 \leq \dot w \leq \dot u$ in $\vlat{E}$ and $r(\dot w) = v$.  Therefore $r$ is interval preserving.
\end{proof}

\begin{thm}\label{Thm:  Existence of Inductive Limits in NVLI}
Let $\cal{D} \defeq \left( (\vlat{E}_\alpha)_{\alpha\in I}, (e_{\alpha, \beta})_{\alpha\precc\beta}\right)$ be a direct system in $\nvlic$, and let $\cal{S} \defeq (\vlat{E},(e_\alpha)_{\alpha \in I})$ be the direct limit of $\cal{D}$ in $\vlc$.  Assume that $e_{\alpha , \beta}$ is injective for all $\alpha \precc \beta$ in $I$.  Then $\cal{S}$ is the direct limit of $\cal{D}$ in $\nvlic$.
\end{thm}

\begin{proof}
We start by proving that $e_\alpha :\vlat{E}_\alpha\to \vlat{E}$ is injective for every $\alpha\in I$.  Fix $\alpha \in I$ and $u\in \vlat{E}_\alpha$ so that $e_\alpha (u ) = \dot 0$ in $\vlat{E}$.  Then there exists $\beta \scc \alpha$ in $I$ so that $e_{\alpha , \beta} (u)=0$.  But $e_{\alpha , \beta}$ is injective, so $u=0$.  Hence $e_\alpha$ is injective.

By Theorem \ref{Thm:  Existence of Inductive Limits in VLI}, $e_\alpha :\vlat{E}_\alpha\to \vlat{E}$ is an injective interval preserving lattice homomorphism for every $\alpha\in I$. It follows from Proposition \ref{Prop:  Interval Preserving vs Lattice Homomorphism} (i) that $e_\alpha$ is a $\nvlic$-morphism for every $\alpha\in I$. Therefore $\cal{S}$ is a compatible system of $\cal{D}$ in $\nvlic$.

Let $\tilde{S} \defeq (\tilde{\vlat{E}},(\tilde{e}_\alpha)_{\alpha\in I})$ be a compatible system of $\cal{D}$ in $\nvlic$.  By Theorem \ref{Thm:  Existence of Inductive Limits in VLI} the canonical map $r:\vlat{E}\to \tilde{\vlat{E}}$ is an interval preserving lattice homomorphism.  We claim that $r$ is a normal lattice homomorphism.  To this end, let $A\downarrow \dot 0$ in $\vlat{E}$.  Without loss of generality we may suppose that $A$ is bounded from above in $\vlat{E}$, say by $\dot u_0$.  There exists $\alpha \in I$ and $u_0 \in \vlat{E}_\alpha$ so that $\dot u_0 = e_\alpha (u_0)$.  Because $e_\alpha$ is injective and interval preserving, there exists for every $\dot u\in A$ a unique $u\in [0,u_0]\subseteq \vlat{E}_\alpha$ so that $e_\alpha (u) = \dot u$.  In particular, $e_\alpha^{-1}[A]\subseteq [0,u_0]$.  We claim that $\inf e_{\alpha}^{-1}[A] = 0$ in $\vlat{E}_\alpha$.  Let $0\leq v\in \vlat{E}_\alpha $ be a lower bound for $e_{\alpha}^{-1}[A]$.  Then $e_{\alpha}(v)\geq0$ is a lower bound for $A$ in $\vlat{E}$, hence $e_{\alpha}(v)=0$.  But $e_\alpha$ is injective, so $v=0$.  This verifies our claim.  By definition, $r[A]=\tilde{e}_\alpha[e_{\alpha}^{-1}[A]]$.  Because $\tilde{e}_{\alpha}$ is a normal lattice homomorphism it follows that $\inf r[A] = 0$ in $\tilde{\vlat{E}}$.
\end{proof}


We recall the following result on permanence of vector lattice properties under the direct limit construction from \cite{Filter1988}.

\begin{thm}\label{Thm:  Inductive Limit Permanence}
Let $\cal{D} \defeq \left( (\vlat{E}_\alpha)_{\alpha\in I}, (e_{\alpha, \beta})_{\alpha\precc\beta}\right)$ be a direct system in a category ${\bf C}$ of vector lattices.  Assume that $e_{\alpha , \beta}$ is injective for all $\alpha \precc \beta$ in $I$.  Let $\cal{S} \defeq (\vlat{E},(e_\alpha)_{\alpha \in I})$ be the direct limit of $\cal{D}$ in $\vlc$.  Then the following statements are true. \begin{itemize}
    \item[(i)] $\vlat{E}$ is Archimedean if and only if $\vlat{E}_\alpha$ is Archimedean for all $\alpha\in I$.
    \item[(ii)] If ${\bf C}$ is $\vlic$ then $\vlat{E}$ is order separable if and only if $\vlat{E}_\alpha$ is order separable for every $\alpha\in I$.
    \item[(iii)] If ${\bf C}$ is $\vlic$ then $\vlat{E}$ has the (principal) projection property if and only if $\vlat{E}_\alpha$ has the (principal) projection property for every $\alpha\in I$.
    \item[(iv)] If ${\bf C}$ is $\vlic$ then $\vlat{E}$ is ($\sigma$-)Dedekind complete if and only if $\vlat{E}_\alpha$ is \linebreak ($\sigma$-)Dedekind complete for every $\alpha\in I$.
    \item[(v)] If ${\bf C}$ is $\vlic$ then $\vlat{E}$ is relatively uniformly complete if and only if $\vlat{E}_\alpha$ is relatively uniformly complete for every $\alpha\in I$.
\end{itemize}
\end{thm}

Before we proceed to discuss examples of direct limits we make some clarifying remarks about the structure of the direct limit of vector lattices.

\begin{remark}\label{Remark:  Inductive limit notation}
Let $\cal{D} \defeq ((\vlat{E}_\alpha)_{\alpha\in I},(e_{\alpha , \beta})_{\alpha \precc\beta})$ be a direct system in $\vlc$ and let $\cal{S} \defeq (\vlat{E},(e_\alpha)_{\alpha \in I})$ be the direct limit of $\cal{D}$ in $\vlc$.  \begin{itemize}
    \item[(i)] Unless clarity demands it, we henceforth cease to explicitly express elements of $\vlat{E}$ as equivalence classes; that is, we write $u\in\vlat{E}$ instead of $\dot u\in\vlat{E}$.
    \item[(ii)] For every $u\in \vlat{E}$ there exists at least one $\alpha\in I$ and $u_\alpha \in\vlat{E}_\alpha$ so that $u = e_\alpha (u_\alpha)$.  If $u = e_\beta (u_\beta)$ for some other $\beta\in I$ and $u_\beta\in \vlat{E}_\beta$ then there exists $\gamma \scc \alpha,\beta$ in $I$ so that $e_{\alpha , \gamma}(u_\alpha) = e_{\beta , \gamma}(u_\beta)$, and hence
        \[
        e_\gamma ( e_{\alpha , \gamma}(u_\alpha)) = u = e_\gamma ( e_{\beta , \gamma}(u_\beta)).
        \]
    \item[(iii)] It is proven in Theorem \ref{Thm:  Existence of Inductive Limits in NVLI} that if $e_{\alpha , \beta}$ is injective for all $\alpha \precc \beta$ in $I$ then $e_\alpha$ is injective for all $\alpha\in I$. In this case we identify $\vlat{E}_\alpha$ with the sublattice $e_\alpha[\vlat{E}_\alpha]$ of $\vlat{E}$.
    \item[(iv)] An element $u\in \vlat{E}$ is positive if and only if there exist $\alpha \precc \beta$ in $I$ and $u_\alpha\in\vlat{E}_\alpha$ so that $e_\alpha(u_\alpha) = u$ and $e_{\alpha , \beta}(u_\alpha)\geq 0$ in $\vlat{E}_\beta$.  Combining this observation with (ii) we see that $u\geq 0$ if and only if there exist $\alpha\in I$ and $0\leq u_\alpha \in\vlat{E}_\alpha$ so that $u=e_\alpha(u_\alpha)$.
\end{itemize}
\end{remark}

\subsection{Examples of direct limits}\label{Subsection:  Examples of inductive limits}

In \cite{Filter1988} a number of examples are presented of naturally occurring vector lattices which can be expressed as direct limits in categories of vector lattices.  We provide further examples which will be used in Section \ref{Section:  Applications}.

\begin{example}\label{Exm:  Inductive limit main example}
Let $\vlat{E}$ be a vector lattice.  Let $(\vlat{E}_\alpha)_{\alpha\in I}$ be an upward directed collection of ideals in $\vlat{E}$ such that $\vlat{E}_\alpha \subseteq \vlat{E}_\beta$ if and only if $\alpha \precc \beta$.  Assume that $\displaystyle\bigcup \vlat{E}_\alpha=\vlat{E}$.  For all $\alpha\precc \beta$ in $I$, let $e_{\alpha , \beta}:\vlat{E}_\alpha\to \vlat{E}_\beta$ and $e_\alpha:\vlat{E}_\alpha \to \vlat{E}$ be the inclusion mappings.  Then $\cal{D} \defeq ((\vlat{E}_{\alpha})_{\alpha\in I},(e_{\alpha , \beta})_{\alpha\precc \beta})$ is a direct system in $\nvlic$ and $\cal{S}\defeq (\vlat{E},(e_{\alpha})_{\alpha\in I})$ is the direct limit of $\cal{D}$ in $\nvlic$.
\end{example}

\begin{proof}
It is clear that $\cal{D}$ is a direct system in $\nvlic$ and that $\cal{S}$ is a compatible system of $\cal{D}$ in $\nvlic$.  Let $\tilde{\cal{S}} = (\tilde{\vlat{E}},(\tilde{e}_{\alpha})_{\alpha\in I})$ be any compatible system of $\cal{D}$ in $\nvlic$.  We show that there exists a unique $\nvlic$-morphism $r:\vlat{E}\to\tilde{\vlat{E}}$ so that for all $\alpha\in I$, the diagram
\[
\begin{tikzcd}[cramped]
\vlat{E} \arrow[rr, "r"] & & \tilde{\vlat{E}}\\
& \vlat{E}_\alpha \arrow[lu, "e_{\alpha}"] \arrow[ru, "\tilde{e}_{\alpha}"']
\end{tikzcd}
\]
commutes.

If $u\in\vlat{E}$ and $\alpha,\beta\in I$ are such that $u\in \vlat{E}_\alpha,\vlat{E}_\beta$, then $\tilde{e}_\alpha (u) = \tilde{e}_\beta(u)$.  Indeed, for any $\gamma \scc \alpha,\beta$ in $I$
\[
\tilde{e}_{\gamma}(u) = \tilde{e}_{\gamma}(e_{\alpha, \gamma}(u)) = \tilde{e}_{\alpha}(u)
\]
and
\[
\tilde{e}_{\gamma}(u) = \tilde{e}_{\gamma}(e_{\beta, \gamma}(u)) = \tilde{e}_{\beta}(u).
\]
Therefore the map $r:\vlat{E}\to \tilde{\vlat{E}}$ given by
\[
r(u) = \tilde{e}_{\alpha}(u) ~\text{if}~ u \in \vlat{E}_\alpha
\]
is well-defined.  It is clear that this map makes the diagram above commute.  Further, if $u,v\in \vlat{E}$ then there exists $\alpha\in I$ so that $u,v \in \vlat{E}_\alpha$.  Then for all $a,b\in\R$ we have $au+bv,u\vee v \in\vlat{E}_\alpha$ so that
\[
r(au+bv) = \tilde{e}_{\alpha}(au+bv) = a \tilde{e}_{\alpha}(u) + b \tilde{e}_{\alpha}(v) = a\thinspace r(u)+b\thinspace r(v)
\]
and
\[
r(u\vee v) = \tilde{e}_{\alpha}(u\vee v) = \tilde{e}_{\alpha}(u)\vee \tilde{e}_{\alpha}(v) = r(u)\vee r(v).
\]
Hence $r$ is a lattice homomorphism.  A similar argument shows that $r$ is interval preserving.  To see that $r$ is a normal lattice homomorphism, let $A\downarrow 0$ in $\vlat{E}$.  Without loss of generality, assume that there exists $0\leq u_0 \in \vlat{E}$ so that $u\leq u_0$ for all $u\in A$.  Then $A\subseteq \vlat{E}_\alpha$ for some $\alpha\in I$ so that $r[A]=\tilde{e}_\alpha [A]$. Hence, because $\tilde{e}_\alpha$ is a normal lattice homomorphism,  $ \inf r[A]  = 0$.  Therefore $r$ is a $\nvlic$-morphism.

It remains to show that $r$ is the unique $\nvlic$-morphism making the diagram above commute.  Suppose that $\tilde{r}$ is any such morphism.  Let $u\in \vlat{E}$.  There exists $\alpha\in I$ so that $u\in \vlat{E}_\alpha$.  We have $\tilde{r}(u) = \tilde{r}(e_\alpha(u)) = \tilde{e}_\alpha(u) = r(u)$, which completes the proof.
\end{proof}

The remaining examples in this section may readily been seen to be special cases of Example \ref{Exm:  Inductive limit main example}.  Therefore we omit the proofs.

\begin{example}\label{Exm:  Inductive limit of principle ideals}
Let $\vlat{E}$ be a vector lattice.  For every $0<u\in \vlat{E}$ let $\vlat{E}_u$ be the ideal generated by $u$ in $\vlat{E}$.  For all $0<u\leq v$ let $e_{u , v}:\vlat{E}_u\to \vlat{E}_v$ and $e_u:\vlat{E}_u\to \vlat{E}$ be the inclusion mappings.  Let $I$ be an upward directed subset of $\vlat{E}^+\mysetminus\{0\}$ so that $\vlat{E}=\displaystyle \bigcup \vlat{E}_u$.  Then $\cal{D} \defeq ((\vlat{E}_{u})_{u\in I},(e_{u , v})_{u\leq v})$ is a direct system in $\nvlic$ and $\cal{S}\defeq (\vlat{E},(e_{u})_{u\in I})$ is the direct limit of $\cal{D}$ in $\nvlic$.
\end{example}

\begin{example}\label{Exm:  Locally supported p-summable functions as inductive limit}
Let $(X,\Sigma,\mu)$ be a complete $\sigma$-finite measure space.  Let $\Xi \defeq (X_n)$ be an increasing sequence (w.r.t. inclusion) of measurable sets with positive measure so that $X =\displaystyle \bigcup X_n$. For $n\leq m$ in $\N$ let $e_{n,m}:\vlat{L}^p(X_n)\to \vlat{L}^p(X_m)$ be defined (a.e.) by setting
\[
	e_{n,m}(u)(t)\defeq \left\{ \begin{array}{lll}
u(t) & \text{ if } & t\in X_n \smallskip \\
0 & \text{ if } & t\in X_m\!\setminus X_n \\
\end{array}\right.
\]
for each $u\in\vlat{L}^p(X_n)$.  Further, define
\[
\vlat{L}^p_{\Xi-c}(X) \defeq \left\lbrace u\in\vlat{L}^p(X) ~:~ u=0 \text{ a.e. on } X\setminus X_n \text{ for some } n\in\N \right\rbrace.
\]
For $n\in\N$ let $e_n:\vlat{L}^p(X_n)\to \vlat{L}^p_{\Xi-c}(X)$ be given by
\[
	e_{n}(u)(t)\defeq \left\{ \begin{array}{lll}
u(t) & \text{ if } & t\in X_n \smallskip \\
0 & \text{ if } & t\in X\setminus X_n \\
\end{array}\right.
\]
for all $u \in \vlat{L}^p(X_n)$.  The following statements are true. \begin{itemize}
    \item[(i)] $\cal{D}^p_{\Xi-c}\defeq \left((\vlat{L}^p(X_n))_{n\in\N}, (e_{n , m})_{n\leq m}\right)$ is a direct system in $\nvlic$, and $e_{n,m}$ is injective for all $n\leq m$ in $\N$.\smallskip
    \item[(ii)] $\cal{S}^p_{\Xi-c}\defeq \left(\vlat{L}^p_{\Xi-c}(X),(e_n)_{n\in\N}\right)$ is the direct limit of $\cal{D}^p_{\Xi-c}$ in $\nvlic$.
\end{itemize}
\end{example}

\begin{example}\label{Exm:  Compactly supported measures as inductive limit}
Let $X$ be a locally compact Hausdorff space.  Let $\Gamma \defeq (X_\alpha)_{\alpha\in I}$ be an upward directed (with respect to inclusion) collection of non-empty open precompact subsets of $X$ so that $\displaystyle\bigcup X_\alpha = X$.  For each $\alpha\in I$, let $\vlat{M}(\bar X_\alpha)$ be the space of Radon measures on $\bar X_\alpha$ and $\vlat{M}_c(X)$ the space of compactly supported Radon measures on $X$.  For all $\alpha \precc \beta$ in $I$, let $e_{\alpha , \beta}:\vlat{M}(\bar X_\alpha)\to \vlat{M}(\bar X_\beta)$ be defined by setting
\[
	e_{\alpha , \beta}(\mu)(B) \defeq \mu(B\cap \bar X_\alpha) \text{ for all } \mu\in \vlat{M}(\bar X_\alpha) \text{ and } B\in\borel{\bar X_{\beta}}.
\]
Likewise, for $\alpha\in I$, define $e_\alpha:\vlat{M}(\bar X_\alpha)\to\vlat{M}_c(X)$ by setting
\[
e_{\alpha}(\mu)(B) \defeq \mu(B\cap \bar X_\alpha) \text{ for all } \mu\in \vlat{M}(X_\alpha) \text{ and } B\in\borel{X}.
\]
The following statements are true. \begin{itemize}
    \item[(i)] $\cal{D}_{\Gamma}\defeq \left((\vlat{M}(\bar X_\alpha)_{\alpha\in I},(e_{\alpha , \beta})_{\alpha\precc \beta}\right)$ is a direct system in $\nvlic$ and $e_{\alpha , \beta}$ is injective for all $\alpha\precc \beta$ in $I$.
        \smallskip
    \item[(ii)] $\cal{S}_{\Gamma}\defeq \left(\vlat{M}_c(X),(e_\alpha)_{\alpha \in I}\right)$ is the direct limit of $\cal{D}_{\Gamma}$ in $\nvlic$.
\end{itemize}
\end{example}


\begin{example}\label{Exm:  Compactly supported NORMAL measures as inductive limit}
Let $X$ be a locally compact Hausdorff space.  Let $\Gamma \defeq (X_\alpha)_{\alpha\in I}$ be an upward directed (with respect to inclusion) collection of open precompact subsets of $X$ so that $\displaystyle\bigcup X_\alpha  =X$.  For each $\alpha\in I$, let $\vlat{N}(\bar X_\alpha)$ be the space of normal Radon measures on $\bar X_\alpha$ and $\vlat{N}_c(X)$ the space of compactly supported normal Radon measures on $X$.  For all $\alpha\precc \beta$ in $I$, let $e_{\alpha , \beta}: \vlat{N}(\bar X_\alpha)\to \vlat{N}(\bar X_\beta)$ be defined by setting
\[
	e_{\alpha , \beta}(\mu)(B) \defeq \mu(B\cap \bar X_\alpha) \text{ for all } \mu\in \vlat{N}(\bar X_\alpha) \text{ and } B\in\borel{\bar X_\beta}.
\]Likewise, for $\alpha\in I$, define $e_\alpha:\vlat{N}(\bar X_\alpha)\to\vlat{N}_c(X)$ by setting
\[
	e_{\alpha}(\mu)(B) \defeq \mu(B\cap \bar X_\alpha) \text{ for all } \mu\in \vlat{N}(X_\alpha) \text{ and } B\in\borel{X}.
\]
The following statements are true. \begin{itemize}
    \item[(i)] $\cal{E}_{\Gamma} \defeq \left((\vlat{N}(\bar X_\alpha)_{\alpha\in I},(e_{\alpha , \beta})_{\alpha\precc \beta}\right)$ is a direct system in $\nvlic$ and $e_{\alpha , \beta}$ is injective for all $\alpha\precc \beta$ in $I$.
    \smallskip
    \item[(ii)] $\cal{T}_{\Gamma}\defeq \left(\vlat{N}_c(X),(e_\alpha)_{\alpha \in I}\right)$ is the direct limit of $\cal{E}_{\Gamma}$ in $\nvlic$.
\end{itemize}
\end{example}

\section{Inverse limits}\label{Section:  Projective limits}

In this section we discuss inverse systems and inverse limits in categories of vector lattices, which are the categorical dual concepts of direct systems and direct limits.  Below we present the definitions of inverse systems and inverse limits in these categories.  As is the case in the previous section, these definitions are specializations of the corresponding definitions in general categories, see for instance \cite[Chapter 5]{Awodey2010} or \cite[Chapter III]{MacLane1998}.

\begin{defn}\label{Defn:  Projective system}
Let ${\bf C}$ be a category of vector lattices, $I$ a directed set, $\vlat{E}_\alpha$ a vector lattice for each $\alpha\in I$, and $p_{\beta , \alpha}:\vlat{E}_\beta\to \vlat{E}_\alpha$ a ${\bf C}$-morphism for all $\beta \scc \alpha$ in $I$.  The ordered pair $\cal{I} \defeq \left((\vlat{E}_\alpha)_{\alpha \in I},(p_{\beta , \alpha})_{\beta \scc \alpha}\right)$ is an \emph{inverse system} in ${\bf C}$ if, for all $\alpha \precc\beta\precc\gamma$ in $I$, the diagram
\[
\begin{tikzcd}[cramped]
\vlat{E}_\gamma \arrow[rd, "p_{\gamma, \beta}"'] \arrow[rr, "p_{\gamma, \alpha}"] & & \vlat{E}_\alpha\\
& \vlat{E}_\beta\arrow[ru, "p_{\beta, \alpha}"']
\end{tikzcd}
\]
commutes in ${\bf C}$.
\end{defn}

\begin{defn}\label{Defn:  Compatible system for projective limit}
Let ${\bf C}$ be a category of vector lattices and $\cal{I} \defeq \left((\vlat{E}_\alpha)_{\alpha \in I},(p_{\beta , \alpha})_{\beta \scc \alpha}\right)$ an inverse system in ${\bf C}$.  Let $\vlat{E}$ be a vector lattice and for every $\alpha\in I$, let $p_\alpha: \vlat{E} \to \vlat{E}_\alpha$ be a ${\bf C}$-morphism.  The ordered pair $\cal{S} \defeq (\vlat{E}, (p_\alpha)_{\alpha\in I})$ is a \emph{compatible system} of $\cal{I}$ in ${\bf C}$ if, for all $\alpha\precc\beta$ in $I$, the diagram
\[
\begin{tikzcd}[cramped]
\vlat{E} \arrow[rd, "p_\beta"'] \arrow[rr, "p_{\alpha}"] & & \vlat{E}_\alpha\\
& \vlat{E}_\beta\arrow[ru, "p_{\beta, \alpha}"']
\end{tikzcd}
\]
commutes in ${\bf C}$.
\end{defn}

\begin{defn}\label{Defn:  Projective limit}
Let ${\bf C}$ be a category of vector lattices and $\cal{I} \defeq \left((\vlat{E}_\alpha)_{\alpha \in I},(p_{\beta , \alpha})_{\beta \scc \alpha}\right)$ an inverse system in ${\bf C}$.  The \emph{inverse limit} of $\cal{I}$ in ${\bf C}$ is a compatible system $\cal{S}\defeq (\vlat{E}, (p_\alpha)_{\alpha\in I})$ so that for any compatible system $\tilde{\cal{S}}\defeq (\tilde{\vlat{E}}, (\tilde{p}_\alpha)_{\alpha\in I})$ in ${\bf C}$ there exists a unique ${\bf C}$-morphism $s: \tilde{\vlat{E}}\to \vlat{E}$ so that, for all $\alpha\in I$, the diagram
\[
\begin{tikzcd}[cramped]
\tilde{\vlat{E}} \arrow[rd, "\tilde{p}_{\alpha}"']\arrow[rr, "s"] & & \vlat{E}\arrow[ld, "p_{\alpha}"]\\
& \vlat{E}_\alpha
\end{tikzcd}
\]
commutes in ${\bf C}$.  The inverse limit of $\cal{I}$ is denoted by $\proj{\cal{I}}$ or $\proj \vlat{E}_\alpha$.
\end{defn}

Since inverse limits are terminal objects in a certain derived category, they are unique up to a unique isomorphism when they exist, see for instance \cite[Corollary 3.2]{BucurDeleanu1968}

\subsection{Existence of inverse limits}\label{Subsection:  Existence of projective limits}

Our first task is to establish the existence of inverse limits in various categories of vector lattices.  The basic result, akin to Filter's result for direct systems, is the following.

\begin{thm}\label{Thm:  Existence of Projective Limit}
Let $\cal{I} \defeq \left( (\vlat{E}_\alpha)_{\alpha\in I}, (p_{\beta, \alpha})_{\beta \scc\alpha}\right)$ be an inverse system in $\vlc$. Define the set
\[
\vlat{E} \defeq \left\lbrace u \in \dsprod \vlat{E}_\alpha ~:~ \pi_\alpha(u) = p_{\beta,\alpha}(\pi_\beta(u)) \text{ for all } \alpha\precc \beta \text{ in } I \right\rbrace. 
\]
For every $\alpha\in I$ define $p_\alpha \defeq \left.\pi_\alpha\right|_{\vlat{E}}$. The following statements are true. \begin{enumerate}
    \item[(i)] $\vlat{E}$ is a vector sublattice of $\dsprod \vlat{E}_\alpha$.
    \item[(ii)] The pair $\cal{S}\defeq ( \vlat{E} , (p_\alpha)_{\alpha\in I})$ is the inverse limit of $\cal{I}$ in $\vlc$.
\end{enumerate}
\end{thm}

\begin{proof}[Proof of (i).]
We verify that $\vlat{E}$ is a sublattice of $\dsprod \vlat{E}_\alpha$; that it is a linear subspace follows by a similar argument, as the reader may readily verify.  Consider $u$ and $ v$ in $\vlat{E}$.  Then $\pi_\alpha(u \vee v) = \pi_\alpha(u) \vee \pi_\alpha(v)$ for all $\alpha\in I$.  Fix any $\alpha,\beta \in I$ so that $\beta \scc \alpha$.  Then
\[
p_{\beta , \alpha}(\pi_\beta(u \vee v)) = p_{\beta , \alpha}(\pi_\beta(u))\vee p_{\beta , \alpha}(\pi_\beta(u)) = \pi_\alpha(u) \vee \pi_\alpha(v)=\pi_\alpha(u\vee v).
\]
Therefore $u \vee v \in \vlat{E}$.  Similarly, $u\wedge v \in \vlat{E}$ so that $\vlat{E}$ is a sublattice of $\dsprod \vlat{E}_\alpha$.
\end{proof}

\begin{proof}[Proof of (ii).]
From the definitions of $\vlat{E}$ and the $p_\alpha$ it is clear that $\cal{S}$ is a compatible system of $\cal{I}$ in $\vlc$.  Let $\tilde{\cal{S}} \defeq (\tilde{\vlat{E}},(\tilde{p}_\alpha)_{\alpha\in I})$ be any compatible system of $\cal{I}$ in $\vlc$.  Define $s:\tilde{\vlat{E}}\to \vlat{E}$ by setting $s(u) \defeq (\tilde{p}_{\alpha}(u))_{\alpha\in I}$.  Let $\beta \scc \alpha$ in $I$.  Because $\tilde{\cal{S}}$ is a compatible system
\[
p_{\beta , \alpha} (\tilde{p}_\beta (u)) = \tilde{p}_\alpha (u), ~ u\in \tilde{\vlat{E}}.
\]
Therefore $s(u)\in \vlat{E}$ for all $u\in \tilde{\vlat{E}}$.  Because each $\tilde{p}_\alpha$ is a lattice homomorphism, so is $s$.  By the definitions of $s$ and the $p_\alpha$, respectively, it follows that $p_\alpha \circ s = \tilde{p}_\alpha$ for every $\alpha \in I$.  We show that $s$ is the unique lattice homomorphism with this property.  To this end, let $\tilde{s}:\tilde{\vlat{E}}\to \vlat{E}$ be a lattice homomorphism so that $p_\alpha \circ \tilde{s}=\tilde{p}_\alpha$ for every $\alpha\in I$.  Fix $u\in \tilde{\vlat{E}}$.  Then for every $\alpha\in I$,
\[
\pi_\alpha (\tilde{s}(u)) = p_\alpha (\tilde{s}(u)) = \tilde{p}_\alpha (u) = \pi_\alpha (s(u)).
\]
Hence $s=\tilde{s}$ and therefore $\proj{\cal{I}} = (\vlat{E},(p_\alpha)_{\alpha \in I})$ in $\vlc$.
\end{proof}

\begin{thm}\label{Thm:  Existence of Projective Limit NVL}
Let $\cal{I} \defeq \left( (\vlat{E}_\alpha)_{\alpha\in I}, (p_{\beta, \alpha})_{\beta\scc\alpha}\right)$ be an inverse system in $\nvlc$ and $\cal{S} \defeq ( \vlat{E}, (p_\alpha)_{\alpha\in I} )$ its inverse limit in $\vlc$. The following statements are true.
\begin{enumerate}
    \item[(i)] Let $A\subseteq \vlat{E}$ and assume that $\inf A = u$ or $\sup A = u$ in $\dsprod \vlat{E}_\alpha$.  Then $u\in \vlat{E}$.
    \item[(ii)] If $\vlat{E}_\alpha$ is Dedekind complete for every $\alpha\in I$ then $\cal{S}$ is the inverse limit of $\cal{I}$ in $\nvlc$.
\end{enumerate}
\end{thm}

\begin{proof}[Proof of (i)]
It is sufficient to consider infima of downward directed subsets of $\vlat{E}$. Let $A\subseteq \vlat{E}$ and assume that $A\downarrow u$ in $\dsprod\vlat{E}_\alpha$. By Theorem~\ref{Thm:  Properties of product of vector lattices.}~(i), for every $\alpha\in I$, $p_\alpha [A]=\pi_\alpha[A]\downarrow \pi_\alpha(u)$ in $\vlat{E}_\alpha$.  For $\beta \scc \alpha$ in $I$,
\[
\pi_\alpha(u) = \inf p_\alpha [A] = \inf p_{\beta , \alpha}[p_\beta[A]] = p_{\beta , \alpha}(\inf p_\beta[A]) = p_{\beta , \alpha}(\pi_\beta(u));
\]
the second to last identity follows from the fact that $p_{\beta , \alpha}$ is a normal lattice homomorphism.  Therefore $u\in \vlat{E}$.
\end{proof}

\begin{proof}[Proof of (ii)]
First, we prove that the $p_\alpha$ are normal lattice homomorphisms.  Let $A\downarrow 0$ in $\vlat{E}$.  Since $\vlat{E}_\alpha$ is Dedekind complete for every $\alpha\in I$, so is $\dsprod \vlat{E}_\alpha$.  Therefore $A\downarrow u$ in $\dsprod \vlat{E}_\alpha$ for some $u\in \dsprod \vlat{E}_\alpha$.  Then $u\in \vlat{E}$ so that $A\downarrow u$ in $\vlat{E}$.  But $A\downarrow  0$ in $\vlat{E}$, hence $u =  0$.  Therefore $\inf p_\alpha [A] = \pi_\alpha(u) = 0$ for every $\alpha \in I$.

From the above it follows that $\cal{S}$ is a compatible system in $\nvlc$.  It remains to show that $\cal{S}$ satisfies Definition \ref{Defn:  Projective limit} in $\nvlc$.  Let $\tilde{\cal{S}} = (\tilde{\vlat{E}},(\tilde{p}_\alpha)_{\alpha\in I})$ be a compatible system in $\nvlc$.  Based on Theorem \ref{Thm:  Existence of Projective Limit} we need only show that $s:\tilde{\vlat{E}}\to \vlat{E}$ defined by setting $s(u) \defeq (\tilde{p}_{\alpha}(u))_{\alpha\in I}$ for every $u \in \tilde{\vlat{E}}$ is a normal lattice homomorphism.

Let $A\downarrow 0$ in $\tilde{\vlat{E}}$.  Then, since each $\tilde{p}_\alpha$ is a normal lattice homomorphism, $\pi_\alpha[s[A]] = p_{\alpha}[s[A]] = \tilde{p}_\alpha [A]\downarrow 0$ in $\vlat{E}_\alpha$ for every $\alpha \in I$.  Hence $s[A]\downarrow  0$ in $\dsprod \vlat{E}_\alpha$, therefore also in $\vlat{E}$.  Therefore $s$ is a normal lattice homomorphism, hence a $\nvlc$-morphism, so that $\proj{\cal{I}}=(\vlat{E},(p_\alpha)_{\alpha \in I})$ in $\nvlc$.
\end{proof}


\begin{remark}
Let $\cal{I} \defeq \left( (\vlat{E}_\alpha)_{\alpha\in I}, (p_{\beta, \alpha})_{\beta\scc \alpha}\right)$ be an inverse system in a category of vector lattices, and $\cal{S} \defeq \left(\vlat{E},(p_\alpha)_{\alpha\in I}\right)$ its inverse limit in $\vlc$.  We occasionally suppress the projections $p_\alpha$ and simply write $\vlat{E}=\proj{\cal{I}}$ or `$\vlat{E}$ is the inverse limit of $\cal{I}$'.
\end{remark}

\subsection{Permanence properties}\label{Subsection:  Permanence properties}

In this section we establish some permanence properties for inverse limits, along the same vein as those for direct limits given in Theorem \ref{Thm:  Inductive Limit Permanence}.  These follow easily from the construction of inverse limits given in Theorem \ref{Thm:  Existence of Projective Limit} and the properties of products of vector lattices given in Theorem \ref{Thm:  Properties of product of vector lattices.}.

\begin{thm}\label{Thm:  Proj Limits Archimedean and RU Complete}
Let $\cal{I} \defeq \left( (\vlat{E}_\alpha)_{\alpha\in I}, (p_{\beta, \alpha})_{\beta \scc \alpha}\right)$ be an inverse system in $\vlc$ and $\cal{S} \defeq \left(\vlat{E},(p_\alpha)_{\alpha\in I}\right)$ its inverse limit in $\vlc$.  The following statements are true.
\begin{enumerate}
    \item[(i)] If $\vlat{E}_\alpha$ is Archimedean for every $\alpha\in I$ then so is $\vlat{E}$.
    \item[(ii)] If $\vlat{E}_\alpha$ is Archimedean and relatively uniformly complete for every $\alpha\in I$ then $\vlat{E}$ is relatively uniformly complete.
\end{enumerate}
\end{thm}

\begin{proof}
We note that (i) follows immediately from Theorems \ref{Thm:  Properties of product of vector lattices.} (ii) and the construction of an inverse limit in $\vlc$.

For (ii), assume that $\vlat{E}_\alpha$ is Archimedean and relatively uniformly complete for every $\alpha \in I$.  We show that every relatively uniformly Cauchy sequence in $\vlat{E}$ is relatively uniformly convergent.  Because $\vlat{E}$ is Archimedean by (i), it follows from \cite[Theorem 39.4]{LuxemburgZaanen1971RSI} that it suffices to consider increasing sequences.  Let $(u_n)$ be an increasing, relatively uniformly Cauchy sequence in $\vlat{E}$.  Then for every $\alpha\in I$, $(p_\alpha(u_n))$ is an increasing sequence in $\vlat{E}_\alpha$.  According to \cite[Theorem 59.3]{LuxemburgZaanen1971RSI}, $(p_\alpha(u_n))$ is relatively uniformly Cauchy in $\vlat{E}_\alpha$.  Because each $\vlat{E}_\alpha$ is relatively uniformly complete, there exists $u_\alpha\in \vlat{E}_\alpha$ so that $(p_\alpha(u_n))$ converges relatively uniformly to $u_\alpha$.  In fact, because $(p_\alpha (u_n))$ is increasing, $u_\alpha = \sup\{p_\alpha(u_n) ~:~n\in\N\}$.  Therefore $u\defeq (u_\alpha)=\sup \{u_n ~:~n\in\N\}$ in $\dsprod \vlat{E}_\alpha$.  By Theorem \ref{Thm:  Existence of Projective Limit NVL}~(i), $u\in \vlat{E}$ so that $u=\displaystyle\sup\{ u_n ~:~ n\in \N\}$ in $\vlat{E}$.  Therefore $(u_n)$ converges relatively uniformly to $u$  by \cite[Lemma 39.2]{LuxemburgZaanen1971RSI}.  We conclude that $\vlat{E}$ is relatively uniformly complete.
\end{proof}

\begin{thm}\label{Thm:  Projective limit permanence properties}
Let $\cal{I} \defeq \left( (\vlat{E}_\alpha)_{\alpha\in I}, (p_{\beta, \alpha})_{\beta\scc\alpha}\right)$ be an inverse system in $\nvlc$ and $\cal{S} \defeq \left(\vlat{E},(p_\alpha)_{\alpha\in I}\right)$ its inverse limit in $\vlc$.  The following statements are true. \begin{enumerate}
    \item[(i)] If $\vlat{E}_\alpha$ is $\sigma$-Dedekind complete for every $\alpha\in I$ then so is $\vlat{E}$.
    \item[(ii)] If $\vlat{E}_\alpha$ is Dedekind complete for every $\alpha\in I$ then so is $\vlat{E}$.
    \item[(iii)] If $\vlat{E}_\alpha$ is laterally complete for every $\alpha\in I$ then so is $\vlat{E}$.
    \item[(iv)] If $\vlat{E}_\alpha$ is universally complete for every $\alpha\in I$ then so is $\vlat{E}$.
\end{enumerate}
\end{thm}

\begin{proof}
We prove (ii). The statements in (i) and (iii) follow by almost identical arguments, and (iv) follows immediately from (ii) and (iii).

Let $D\subseteq \vlat{E}$ be an upwards directed set bounded above by $u \in \vlat{E}$.  For every $\alpha\in I$ the set $D_\alpha \defeq p_\alpha\left[ D \right]$ is bounded above in $\vlat{E}_\alpha$ by $\pi_\alpha(u) \in \vlat{E}_\alpha$.  Since $\vlat{E}_\alpha$ is Dedekind complete for every $\alpha\in I$, $v_\alpha \defeq \sup D_\alpha$ exists in $\vlat{E}_\alpha$ for all $\alpha\in I$.  We have that $\sup D = \left( v_\alpha \right)$ in $\dsprod \vlat{E}_\alpha$.  By Theorem \ref{Thm:  Existence of Projective Limit NVL}~(i), $v\defeq \left( v_\alpha \right) \in \vlat{E}$.  Because $\vlat{E}$ is a sublattice of $\dsprod \vlat{E}_\alpha$ it follows that $v=\sup D$ in $\vlat{E}$.
\end{proof}

\subsection{Examples of inverse limits}\label{Subsection:  Examples of projective limits}

In this section we present a number of examples of inverse systems and their limits in categories of vector lattices.  These will be used in Section \ref{Section:  Applications}.  Our first example is related to Example \ref{Exm:  Locally supported p-summable functions as inductive limit}.

\begin{example}\label{Exm:  Lploc projective limit}
Let $(X,\Sigma,\mu)$ be a complete $\sigma$-finite measure space.  Let $\Xi \defeq (X_n)$ be an increasing sequence (w.r.t. inclusion) of measurable sets with positive measure so that $X = \displaystyle \bigcup X_n$.  For $1 \leq p \leq \infty$ let $\vlat{L}^p_{\Xi-\loc}(X)$ denote the set of (equivalence classes of) measurable functions $u:X\to\R$ so that $u\onefunction_{X_n}\in \vlat{L}^p(X_n)$ for every $n\in\N$.  For $m \geq n$ in $\N$ let $r_{m , n}:\vlat{L}^p(X_m)\to \vlat{L}^p(X_n)$ and $r_n:\vlat{L}^p_{\Xi-\loc}(X)\to \vlat{L}^p(X_n)$ be the restriction maps.  The following statements are true.\begin{enumerate}
    \item[(i)] $\cal{I}^p_{\Xi-\loc} \defeq ((\vlat{L}^p(X_n))_{n\in\N},(r_{m , n})_{m \geq n})$ is an inverse system in $\nvlic$, and $r_{m,n}$ is surjective for all $m\geq n$ in $\N$.
    \item[(ii)] $\cal{S}^p_{\Xi-\loc}\defeq (\vlat{L}^p_{\Xi-\loc}(X),(r_n)_{n\in \N})$ is a compatible system of $\cal{I}^p_{\Xi-\loc}$ in $\nvlic$.
    \item[(iii)] $\cal{S}^p_{\Xi-\loc}$ is the inverse limit of $\cal{I}^p_{\Xi-\loc}$ in $\nvlc$.
\end{enumerate}
\end{example}


\begin{proof}
That (i) and (ii) are true is clear.  We prove (iii).

Because $\vlat{L}^p(X_n)$ is Dedekind complete for every $n\in\N$, $\proj{\cal{I}^p_{\Xi-\loc}}\defeq (\vlat{F},(p_n)_{n\in\N})$ exists in $\nvlc$ by Theorem \ref{Thm:  Existence of Projective Limit NVL}~(ii).  Since $\cal{S}^p_{\Xi-\loc}$ is a compatible system of $\cal{I}^p_{\Xi-\loc}$ in $\nvlc$ there exists a unique normal lattice homomorphism $s:\vlat{L}^p_{\Xi-\loc}(X)\to \vlat{F}$ so that the diagram
\[
\begin{tikzcd}[cramped]
\vlat{L}^p_{\Xi-\loc}(X) \arrow[rd, "r_{n}"']\arrow[rr, "s"] & & \vlat{F}\arrow[ld, "p_{n}"]\\
& \vlat{L}^p(X_n)
\end{tikzcd}
\]
commutes for every $n\in\N$.  We show that $s$ is bijective.  To see that $s$ is injective, suppose that $s(u)=0$ for some $u\in \vlat{L}^p_{\Xi-\loc}(X)$.  Then $r_n(u)=0$ for every $n\in\N$; that is, the restriction of $u$ to each set $X_n$ is $0$.  Since $\displaystyle\bigcup X_n = X$ it follows that $u=0$.  To see that $s$ is surjective, consider $u\in\vlat{F}$.  If $m \geq n$ then $p_n(u)=r_{m,n}(p_m(u))$; that is, $p_n(u)=p_m(u)$ a.e. on $X_n$.  Therefore $v:X\to\R$ given by
\[
v(x) \defeq p_n(u)(x) \text{ if } x\in X_n
\]
is a.e. well-defined on $X=\displaystyle \bigcup X_n$.  For $n\in\N$, $v$ restricted to $X_n$ is $p_n(u)\in \vlat{L}^p(X_n)$.  Therefore $v\in\vlat{L}^p_{\Xi-\loc}(X)$.  Furthermore, $p_n(s(v)) = r_n(v) = p_n(u)$ for all $n\in\N$ so that $s(v)=u$. We conclude that $s$ is a lattice isomorphism.
\end{proof}

Our second example is a companion result for Examples \ref{Exm:  Compactly supported measures as inductive limit} and \ref{Exm:  Compactly supported NORMAL measures as inductive limit}.

\begin{example}\label{Exm:  Continuous functions projective limit}
Let $X$ be a topological space and $\cal{O}\defeq \{O_\alpha ~:~ \alpha\in I\}$ collection of non-empty open subsets of $X$ which is upward directed with respect to inclusion; that is, $\alpha \precc \beta$ if and only if $O_\alpha \subseteq O_\beta$.  Assume that $\bigcup O_\alpha$ is dense and $\cont$-embedded in $X$.  For $\beta\scc \alpha$, denote by $r_{\beta , \alpha}:\cont(\bar O_\beta)\to \cont(\bar O_\alpha)$ and $r_\alpha:\cont(X)\to\cont(\bar O_\alpha)$ the restriction maps.  The following statements are true. \begin{enumerate}
    \item[(i)] $\cal{I}_\cal{O} \defeq ((\cont(\bar O_\alpha))_{\alpha\in I},(r_{\beta , \alpha})_{\beta \scc \alpha})$ is an inverse system in $\vlc$.
    \item[(ii)] $\cal{S}_\cal{O} \defeq (\cont(X),(r_\alpha)_{\alpha\in I})$ is a compatible system of $\cal{I}_{\cal{O}}$ in $\vlc$.
    \item[(iii)] $\cal{S}_\cal{O}$ is the inverse limit of $\cal{I}_{\cal{O}}$ in $\vlc$.
    \item[(iv)] If $X$ is a Tychonoff space and $O_\alpha$ is precompact for every $\alpha\in I$ then $\cal{I}_\cal{O}$ is an inverse system in $\nvlic$, $\cal{S}_\cal{O}$ is a compatible system of $\cal{I}_{\cal{O}}$ in $\nvlic$, and $r_{\beta,\alpha}$ is surjective for all $\beta \scc \alpha$ in $I$.
\end{enumerate}
\end{example}

\begin{proof}
That (i), (ii) and (iii) are true follows from arguments similar to those used in the proof of Example \ref{Exm:  Lploc projective limit}.  We therefore omit the proofs of these statements.  We only note that for (iii), we use the fact that every $u\in\cont(\displaystyle\bigcup O_\alpha)$ has a unique continuous and real-valued extension to $X$; that is, restriction from $X$ to $\displaystyle\bigcup O_\alpha$ defines a lattice isomorphism from $\cont(\displaystyle\bigcup O_\alpha)$ onto $\contX$.

To verify the first two statements in (iv) it is sufficient to show that the $r_\alpha$ and $r_{\alpha,\beta}$ are order continuous and interval preserving.  That these maps are order continuous follows from \cite[Theorem 3.4]{KandicVavpeticPositivity2019}.  That they are interval preserving follows from the fact that every compact subset of a Tychonoff space is $\cont^\ast$-embedded.  We show that the $r_\alpha$ are interval preserving, the proof for $r_{\alpha,\beta}$ being identical.  Consider an $\alpha\in I$, $u\in \cont(X)^+$ and $v\in\cont(\bar O_\alpha)$ so that $0\leq v\leq r_\alpha(u)$.  Because $\bar O_\alpha$ is $\cont^\ast$-embedded in $X$ there exists a continuous function $v'\in \cont (X)$ so that $r_\alpha (v')=v$.  Let $w\defeq (0\vee v')\wedge u$.  Then $0\leq w\leq u$ and, because $r_\alpha$ is a lattice homomorphism, $r_\alpha(w)=v$.  Therefore $[0,r_\alpha (u)] = r_\alpha [[0,u]]$.

For every $\beta \scc \alpha$ in $I$, $\bar O_\alpha$ is $\cont^\ast$-embedded in $\bar O_\beta$ so that $r_{\beta,\alpha}$ is surjective.
\end{proof}


Our next example is of a more general nature.  It is an essential ingredient in our solution of the decomposition problem for $\cont(X)$ mentioned in Section \ref{Section:  Introduction}.

\begin{example}\label{Exm:  Projective limits of bands}
Let $\vlat{E}$ be an Archimedean vector lattice.  Denote by $\bands{\vlat{E}}$ the Boolean algebra of projection bands in $\vlat{E}$.\footnote{$\bands{\vlat{E}}$ is ordered by inclusion.}  Let $\mathrm{M}$ be a non-trivial ideal in $\bands{\vlat{E}}$; that is, $\mathrm{M}\subset \bands{\vlat{E}}$ is downward closed, upward directed and does not consist of the trivial band $\{0\}$ only.  For notational convenience we express $\mathrm{M}$ as indexed by a directed set $I$, $\mathrm{M} = \{\vlat{B}_\alpha ~:~ \alpha \in I\}$, so that $\alpha \precc \beta$ if and only if $\vlat{B}_{\alpha}\subseteq \vlat{B}_{\beta}$.

For $\vlat{B}_\alpha \subseteq \vlat{B}_\beta$ in $\mathrm{M}$, denote by $P_\alpha$ the band projection of $\vlat{E}$ onto $\vlat{B}_\alpha$ and by $P_{\beta , \alpha}$ the band projection of $\vlat{B}_\beta$ onto $\vlat{B}_\alpha$; that is, $P_{\beta , \alpha} = \left.P_{\alpha}\right|_{\vlat{B}_\beta}$.  The following statements are true.
\begin{enumerate}
    \item[(i)] $\cal{I}_{\mathrm{M}} \defeq (\mathrm{M} , (P_{\beta,\alpha})_{\beta\scc \alpha})$ is an inverse system in $\nvlic$ and $\tilde{\cal{S}} \defeq (\vlat{E},(P_\alpha)_{\alpha\in I})$ is a compatible system of $\cal{I}_\vlat{M}$ in $\nvlic$.
    \item[(ii)] $\proj{\cal{I}_{\mathrm{M}}}\defeq \left(\vlat{F},(p_\alpha)_{\alpha \in I}\right)$ exists in $\vlc$.  If $\vlat{E}$ is Dedekind complete then $\left(\vlat{F},(p_\alpha)_{\alpha \in I}\right)$ is the inverse limit of $\cal{I}_{\mathrm{M}}$ in $\nvlc$.
    \item[(iii)]  $P_{\mathrm{M}}:\vlat{E}\ni u \mapsto (P_\alpha(u))_{\alpha\in I}\in \vlat{F}$ is the unique lattice homomorphism so that the diagram
        \[
        \begin{tikzcd}[cramped]
        \vlat{E} \arrow[rd, "P_{\alpha}"']\arrow[rr, "P_{\mathrm{M}}"] & & \vlat{F}\arrow[ld, "p_{\alpha}"]\\
        & \vlat{B}_\alpha
        \end{tikzcd}
        \]
        commutes for every $\alpha\in I$.  Furthermore, $P_{\mathrm{M}}[\vlat{E}]$ an order dense sublattice of $\vlat{F}$.  If $\vlat{E}$ is Dedekind complete then $P_\vlat{M}[\vlat{E}]$ is an ideal in $\vlat{F}$.
    \item[(iv)] $P_{\mathrm{M}}$ is injective if and only if $\{P_\alpha : \alpha\in I\}$ separates the points of $\vlat{E}$.  In this case, $P_{\mathrm{M}}$ is a lattice isomorphism onto an order dense sublattice of $\vlat{F}$.
\end{enumerate}
\end{example}


\begin{proof}
Since band projections are both interval preserving and order continuous, (i) follows immediately from Proposition \ref{Prop: Properties of band projections}.  The statement in (ii) follows immediately from (i) and Theorems \ref{Thm:  Existence of Projective Limit} and \ref{Thm:  Existence of Projective Limit NVL}~(ii).  That (iv) is true is a direct consequence of the definition of $P_{\mathrm{M}}$.

We proceed to prove (iii).  Since $P_\alpha$ is a lattice homomorphism for every $\alpha\in I$, $P_\vlat{M}$ is a lattice homomorphism into $\dsprod \vlat{B}_\alpha$.  If $u\in\vlat{E}$ and $\alpha \precc \beta$ then $P_{\beta,\alpha}(P_\beta ( u)) = P_\alpha (u)$ by Proposition \ref{Prop: Properties of band projections} (iii).  Hence $P_\vlat{M}[\vlat{E}]$ is a sublattice of $\vlat{F}$.  It follows from the construction of $\vlat{F}$ as a sublattice of $\dsprod \vlat{B}_\alpha$ given in Theorem \ref{Thm:  Existence of Projective Limit} that $p_\alpha \circ P_{\vlat{M}}=P_\alpha$ for all $\alpha\in I$.

Let $0 < u=(u_\alpha)\in \vlat{F}$.  There exists $\alpha_0\in I$ so that $u_{\alpha_0}>0$ in $\vlat{B}_{\alpha_0}\subseteq \vlat{E}$.  Then $0<P_\vlat{M}(u_{\alpha_0}) \leq u$ in $\vlat{F}$.  Hence $P_\vlat{M}[\vlat{E}]$ is order dense in $\vlat{F}$.

Assume that $\vlat{E}$ is Dedekind complete.  We show that $P_\vlat{M}[\vlat{E}]$ is an ideal in $\vlat{F}$.  Consider $v\in \vlat{E}^+$ and $u=(u_\alpha)\in \vlat{F}^+$ so that $0\leq u \leq P_\vlat{M}(v)$.  Then $u_\alpha \leq P_\alpha(v) \leq v$ for all $\alpha\in I$.  Let $w=\sup \{u_\alpha ~:~ \alpha \in I\}$ in $\vlat{E}$.  We claim that $P_\vlat{M}(w)=u$.  Because $u_\alpha \leq w$ for all $\alpha \in I$, $u_\alpha = P_\alpha(u_\alpha) \leq P_\alpha(w)$.  Therefore $u\leq P_\vlat{M}(w)$.  For the reverse inequality we note that for all $\beta\in I$,
\[
P_\beta(w) = \sup\{ P_\beta (u_\alpha) ~:~ \alpha \in I\}.
\]
We claim that $P_{\beta}(u_\alpha) \leq u_\beta$ for all $\alpha,\beta\in I$.  It follows from this claim that $P_\beta(w) \leq u_\beta$ so that $P_\vlat{M}(w) \leq u$.  Thus we need only verify that, indeed, $P_{\beta}(u_\alpha) \leq u_\beta$ for all $\alpha,\beta \in I$.  To this end, fix $\alpha,\beta\in I$.  Let $\gamma\in I$ be a mutual upper bound for $\alpha$ and $\beta$.  Because $u=(u_\alpha) \in\vlat{F}$, $\tilde{\cal{S}}$ is compatible with $\cal{I}_\vlat{M}$ and $u_\gamma,u_\alpha\in \vlat{E}$ we have
\[
P_\beta(u_\alpha) = P_\beta (P_{\gamma, \alpha}(u_\gamma)) \leq P_\beta(u_\gamma) = P_{\gamma , \beta} (P_\gamma(u_\gamma)) = P_{\gamma , \beta}(u_\gamma) = u_\beta.
\]
This completes the proof.
\end{proof}

\begin{remark}\label{Remark:  Vector lattice as projective limit of bands}
Let $\vlat{E}$, $\bands{\vlat{E}}$, $\mathrm{M}$, $\cal{I}_{\mathrm{M}}$, $P_{\mathrm M}$ and $\tilde{\cal{S}}$ be as in Example \ref{Exm:  Projective limits of bands}.  Assume that $\{P_\alpha ~:~ \alpha\in I\}$ separates the points of $\vlat{E}$.  It may happen that $P_{\mathrm M}$ maps $\vlat{E}$ onto $\proj{\cal{I}_\vlat{M}}$, but this is not always the case.  If this is the case, then $\proj{\cal{I}_\vlat{M}}=\tilde{\cal{S}}$ in $\vlc$, or, if $\vlat{E}$ is Dedekind complete, in $\nvlc$.  A sufficient, but not necessary, condition for $P_\mathrm{M}$ to map $\vlat{E}$ onto $\vlat{F}$ is that $\vlat{E}\in\mathrm{M}$.
\begin{enumerate}
    \item[(i)] Consider the vector lattice $\R^\omega$ of all functions from $\N$ to $\R$. For $F\subseteq \N$ let
    \[
	\vlat{B}_F \defeq \{u\in \R^\omega ~:~ \supp (u) \subseteq F\}.
    \]
    Then $\mathrm{M} \defeq \{\vlat{B}_F ~:~ \emptyset\neq F\subseteq \N \text{ is finite} \}$ is a proper, non-trivial ideal in $\rm{ {\bf B}}_{\R^\omega}$ and $\{P_F : \emptyset\neq F\subseteq \N \text{ finite} \}$ separates the points of $\R^\omega$.  It is easy to see that $P_{\mathrm{M}}$ maps $\R^\omega$ onto $\proj{\cal{I}_{\mathrm{M}}}$.
    \item[(ii)] Consider the vector lattice $\ell^1$.  As in (i), for $F\subseteq \N$ define
    \[
	\vlat{B}_F \defeq \{u\in \ell^1 ~:~ \supp (u) \subseteq F\}
    \]
    Then $\mathrm{M} \defeq  \{\vlat{B}_F ~:~ \emptyset\neq F\subseteq \N \text{ is finite} \}$ is a proper, non-trivial ideal in $\rm{ {\bf B}}_{\ell^1}$ and $\proj{\cal{I}_\mathrm{M}}$ is $\R^\omega$.  In this case, $P_{\mathrm{M}}[\ell^1]$ is a proper subspace of $\proj{\cal{I}_\mathrm{M}}$.
\end{enumerate}
\end{remark}

Based on Remark \ref{Remark:  Vector lattice as projective limit of bands} we ask the following question:  Given a Dedekind complete vector lattice $\vlat{E}$, does there exist a proper ideal $\mathrm{M}$ in $\bands{\vlat{E}}$ so that $P_\mathrm{M}:\vlat{E}\to \proj{\cal{I}_\mathrm{M}}$ is an isomorphism onto $\proj{\cal{I}_\mathrm{M}}$?  We do not pursue this question any further here, except to note the following example.

\begin{example}
Let $X$ be an extremally disconnected Tychonoff space.  Let $\cal{O} \defeq \{O_\alpha : \alpha \in I\}$ be a proper, non-trivial ideal in the Boolean algebra ${\bf R}_X$ of clopen subsets of $X$.  Assume that $\bigcup O_\alpha$ is dense and $\cont$-embedded in $X$.  Then $\vlat{M}\defeq \{\cont (O_\alpha) ~:~ \alpha\in I\}$ is a proper, non-trivial ideal in $\bands{\contX}$ and $P_\vlat{M}:\contX \to \proj{\cal{I}_\vlat{M}}$ is a lattice isomorphism onto $\proj{\cal{I}_\vlat{M}}$.
\end{example}

\begin{proof}
The Boolean algebras ${\bf R}_X$ and $\bands{\contX}$ are isomorphic.  In particular, the isomorphism is given by
\[
{\bf R}_X\ni O \longmapsto \vlat{B}_O=\{u\in \contX ~:~ \supp(u)\subseteq O\},
\]
see  \cite[Theorem 12.9]{deJongevanRooijRieszSpaces1977}.  We note that for $O\in{\bf R}_X$ the band $\vlat{B}_O$ may be identified with $\cont(O)$, and the band projection onto $\vlat{B}_O$ is given by restriction to $O$.    Therefore $\vlat{M}$ is a proper, non-trivial ideal in $\bands{\contX}$.  It follows from Example \ref{Exm:  Continuous functions projective limit} that $\proj{\cal{I}_\vlat{M}}=\contX$, i.e. $\cal{I}_\vlat{M}:\contX \to \proj{\cal{I}_\vlat{M}}$ is a lattice isomorphism onto $\proj{\cal{I}_\vlat{M}}$.
\end{proof}

\section{Dual spaces}\label{Section:  Dual spaces}

The results presented in this section form the technical heart of the paper.  Roughly speaking, we will show, under fairly general assumptions, that the order (continuous) dual of a direct limit is an inverse limit.  On the other hand, more restrictive conditions are needed to show that the order (continuous) dual of an inverse limit is a direct limit.  These results form the basis of the applications given in Section \ref{Section:  Applications}.

\subsection{Duals of direct limits}\label{Subsection:  Duals of inductive limits}

\begin{defn}\label{Defn:  Dual system of inductive system}
Let $\cal{D}\defeq \left( (\vlat{E}_\alpha)_{\alpha\in I}, (e_{\alpha, \beta})_{\alpha\precc\beta}\right)$ be a direct system in $\vlic$.  The \emph{dual system} of $\cal{D}$ is the pair $\cal{D}^\sim \defeq \left( (\vlat{E}^\sim_\alpha)_{\alpha\in I}, (e^\sim_{\alpha, \beta})_{\alpha\precc\beta}\right)$.

If $\cal{D}$ is a direct system in $\nvlic$, define the \emph{order continuous dual system} of $\cal{D}$ as the pair $\ordercontn{\cal{D}} \defeq \left( (\ordercontn{(\vlat{E}_\alpha)})_{\alpha\in I}, (e^\sim_{\alpha, \beta})_{\alpha\precc\beta}\right)$ with $e^\sim_{\alpha, \beta}: \ordercontn{(E_\beta)} \to \ordercontn{(\vlat{E}_\alpha)}$.
\end{defn}

\begin{prop}\label{Prop:  Dual system of inductive system is projective system}
Let $\cal{D}\defeq \left( (\vlat{E}_\alpha)_{\alpha\in I}, (e_{\alpha, \beta})_{\alpha\precc\beta}\right)$ be a direct system in $\vlic$.  The following statements are true. \begin{enumerate}
    \item[(i)] The dual system $\cal{D}^\sim$ is an inverse system in $\nvlic$.
    \item[(ii)] If $\cal{D}$ is a direct system in $\nvlic$ then the order continuous dual system $\ordercontn{\cal{D}}$ is an inverse system in $\nvlic$.
\end{enumerate}
\end{prop}

\begin{proof}
We present the proof of (i).  That (ii) is true follows by a similar argument, so we omit the proof.

The maps $e_{\alpha, \beta}: \vlat{E}_\alpha\to \vlat{E}_\beta$ are interval preserving lattice homomorphisms for all $\alpha \precc \beta$. By Theorem \ref{Thm:  Adjoints of interval preserving vs lattice homomorphisms} the adjoint maps $e^\sim_{\alpha, \beta}:\vlat{E}_\beta^\sim \to \vlat{E}_\alpha^\sim$ are normal interval preserving lattice homomorphisms.  Fix $\alpha, \beta, \gamma\in I$ such that $\alpha \precc \beta \precc \gamma$.  Since $\cal{D}$ is a direct system in $\vlic$, $e_{\alpha , \gamma} = e_{\beta , \gamma } \circ e_{\alpha , \beta}$ so that $e_{\alpha , \gamma}^\sim = e_{\alpha , \beta}^\sim \circ e_{\beta , \gamma }^\sim$. Thus the dual system $\cal{D}^\sim = \left( (\vlat{E}^\sim_\alpha)_{\alpha\in I}, (e^\sim_{\alpha, \beta})_{\alpha\precc\beta}\right)$ is an inverse system in $\nvlic$.
\end{proof}

\begin{prop}\label{Prop:  Dual compactible systems inductive limit}
Let $\cal{D} \defeq \left( (\vlat{E}_\alpha)_{\alpha\in I}, (e_{\alpha, \beta}) \right)$ be a direct system in $\vlic$ and  $\cal{S}\defeq \left( \vlat{E}, (e_\alpha)_{\alpha\in I} \right)$ a compatible system of $\cal{D}$ in $\vlic$. The following statements are true. \begin{enumerate}
    \item[(i)] $\cal{S}^\sim \defeq \left( \vlat{E}^\sim, (e^\sim_\alpha)_{\alpha\in I} \right)$ is a compatible system for the inverse system $\cal{D}^\sim$ in $\nvlic$.
    \item[(ii)] If $\cal{D}$ is a direct system in $\nvlic$ and $\cal{S}$ is a compatible system in $\nvlic$, then $\ordercontn{\cal{S}} \defeq \left( \ordercontn{\vlat{E}}, (e^\sim_\alpha)_{\alpha\in I} \right)$ is a compatible system for the inverse system $\ordercontn{\cal{D}}$ in $\nvlic$.
\end{enumerate}
\end{prop}

\begin{proof}
Again, we only prove (i) as the proof of (ii) is similar.  By Theorem \ref{Thm:  Adjoints of interval preserving vs lattice homomorphisms}, $e_{\alpha}^\sim: \vlat{E}^\sim \to \vlat{E}^\sim_\alpha$ is a normal interval preserving lattice homomorphism for every $\alpha\in I$.  Furthermore, if $\alpha \precc \beta $ then $e_\alpha = e_\beta \circ e_{\alpha , \beta}$ so that $e_\alpha^\sim = e_{\alpha , \beta}^\sim \circ e_\beta^\sim$.  Therefore $\cal{S}^\sim$ is a compatible system of $\cal{D}^\sim$ in $\nvlic$.
\end{proof}

The main results of this section are the following.

\begin{thm}\label{Thm:  Dual of ind sys in VLIC is proj of duals in NVL}
Let $\cal{D}\defeq \left( (\vlat{E}_\alpha)_{\alpha\in I}, (e_{\alpha, \beta})_{\alpha\precc\beta}\right)$ be a direct system in $\vlic$, and let $\cal{S} \defeq \left( \vlat{E},(e_\alpha)_{\alpha\in I}\right)$ be the direct limit of $\cal{D}$ in $\vlic$.  The following statements are true.
\begin{itemize}
\item[(i)] $\proj{\cal{D}^\sim}\defeq \left(\vlat{F},(p_\alpha)_{\alpha\in I}\right)$ exists in $\nvlc$.
\item[(ii)] $\left(\ind{\cal{D}}\right)^\sim \cong \proj{\cal{D}^\sim}$ in $\nvlc$; that is, there exists a lattice isomorphism $T: \vlat{E}^\sim \to \vlat{F}$ such that the following diagram commutes for all $\alpha\in I$.
\end{itemize}
\begin{eqnarray}\label{EQ:  Thm-Dual of ind sys in VLIC is proj of duals in NVL diagram}
\begin{tikzcd}[cramped]
\vlat{E}^\sim \arrow[rd, "e^\sim_\alpha"'] \arrow[rr, "T"] & & \vlat{F} \arrow[dl, "p_\alpha"]\\
& \orderdual{\vlat{E}}_\alpha
\end{tikzcd}
\end{eqnarray}
\end{thm}

\begin{proof}
That (i) is true follows from Proposition \ref{Prop:  Dual system of inductive system is projective system} and Theorem \ref{Thm:  Existence of Projective Limit NVL}~(ii) because $\orderdual{\vlat{E}}_\alpha$ is Dedekind complete for every $\alpha\in I$.

We prove (ii).  By Proposition \ref{Prop:  Dual compactible systems inductive limit}, $\cal{S}^\sim \defeq \left( \orderdual{\vlat{E}},(e^\sim_\alpha)_{\alpha\in I}\right)$ is a compatible system for $\cal{D}^\sim$ in $\nvlic$, hence also in $\nvlc$.  Therefore there exists a unique normal lattice homomorphism $T:\orderdual{\vlat{E}}\to \vlat{F}$ so that the diagram (\ref{EQ:  Thm-Dual of ind sys in VLIC is proj of duals in NVL diagram}) commutes.  We show that $T$ is bijective.

To see that $T$ is injective, let $\psi\in \vlat{E}^\sim$ and suppose that $T(\psi)=0$.  Consider any $u\in \vlat{E}$.  There exist $\alpha\in I$ and $u_\alpha\in \vlat{E}_\alpha$ so that $u=e_\alpha(u_\alpha)$, see Remark \ref{Remark:  Inductive limit notation}.  Then $\psi(u) = \psi(e_\alpha(u_\alpha)) = e^\sim_\alpha(\psi)(u_\alpha) = p_\alpha(T(\psi))(u) = 0$.  This holds for all $u\in\vlat{E}$ so that $\psi=0$.  Therefore $T$ is injective.

It remains to show that $T$ maps $\vlat{E}^\sim$ onto $\vlat{F}$.  To this end, consider $(\varphi_\alpha) \in \vlat{F}^+$.  We construct a functional $0 \leq \varphi\in \vlat{E}^\sim$ so that $T(\varphi) = (\varphi_\alpha)$.

Let $u\in\vlat{E}$.  Consider any $\alpha,\beta\in I$, $u_\alpha\in\vlat{E}_\alpha$ and $u_\beta\in\vlat{E}_\beta$ so that $e_\alpha(u_\alpha)=u=e_\beta(u_\beta)$, see Remark \ref{Remark:  Inductive limit notation}.   We claim that $\varphi_\alpha(u_\alpha) = \varphi_\beta(u_\beta)$.  Indeed, there exists $\gamma \scc \alpha,\beta$ in $I$ so that $e_{\alpha , \gamma}(u_\alpha) = e_{\beta , \gamma}(u_\beta)$.  Furthermore, $e_{\gamma}(e_{\alpha , \gamma}(u_\alpha))=u=e_{\gamma}(e_{\beta , \gamma}(u_\beta))$.  Because $(\varphi_\alpha) \in\vlat{F}$ we have $\varphi_\alpha = e_{\alpha , \gamma}^\sim (\varphi_\gamma)$ and $\varphi_\beta = e_{\beta , \gamma}^\sim (\varphi_\gamma)$.  Hence
\[
\varphi_\alpha (u_\alpha) = \varphi_\gamma (e_{\alpha , \gamma}(u_\alpha)) = \varphi_\gamma (e_{\beta , \gamma}(u_\beta)) = \varphi_\beta (u_\beta).
\]
Thus our claim is verified.

For $u\in\vlat{E}$ define $\varphi(u) \defeq \varphi_\alpha(u_\alpha)$ if $u=e_\alpha (u_\alpha)$.  By our above claim, $\varphi$ is a well-defined map from $\vlat{E}$ into $\R$. To see that $\varphi$ is linear, consider $u,v\in \vlat{E}$ and $a,b\in\R$.  Let $u=e_\alpha (u_\alpha)$ and $v=e_\beta (v_\beta)$ where $\alpha,\beta\in I$, $u_\alpha\in\vlat{E}_\alpha$ and $v_\beta\in \vlat{E}_\beta$.  There exists $\gamma \scc \alpha,\beta$ in $I$ so that
\[
au+bv = e_{\gamma}(a e_{\alpha , \gamma}(u_\alpha) + b e_{\beta , \gamma}(v_\beta)).
\]
Then
\[
\varphi(au+bv) = \varphi_{\gamma}(a e_{\alpha , \gamma}(u_\alpha) + b e_{\beta , \gamma}(v_\beta)) = a\varphi_\gamma(e_{\alpha , \gamma}(u_\alpha)) + b \varphi_\gamma(e_{\beta , \gamma}(v_\beta)).
\]
But $e_\gamma (e_{\alpha , \gamma}(u_\alpha)) = e_\alpha(u_\alpha)=u$ and $e_\gamma (e_{\beta , \gamma}(v_\beta)) = e_\beta(v_\beta)=v$.  Hence $\varphi_\gamma(e_{\alpha , \gamma}(u_\alpha))= \varphi(u)$ and $\varphi_\gamma(e_{\beta , \gamma}(v_\beta)) = \varphi(v)$.  Therefore $\varphi(au+bv) = a\varphi(u) + b\varphi(v)$.

We show that $\varphi$ is positive.  If $0\leq u\in \vlat{E}$ then there exist $\alpha\in I$ and $0\leq u_\alpha \in\vlat{E}_\alpha$ so that $u=e_\alpha(u_\alpha)$, see Remark \ref{Remark:  Inductive limit notation}.  Then $\varphi(u) = \varphi_\alpha( u_\alpha )\geq 0$, the final inequality following from the fact that $(\varphi_\alpha)\in\vlat{F}^+$.

It follows from the definition of $\varphi$ and the commutativity of the diagram (\ref{EQ:  Thm-Dual of ind sys in VLIC is proj of duals in NVL diagram}) that $p_{\alpha}(T(\varphi))= e_\alpha^\sim(\varphi)=\varphi_\alpha$ for every $\alpha\in I$.  Hence $T(\varphi) = (\varphi_\alpha)$ so that $T$ is surjective.
\end{proof}

\begin{thm}\label{Thm:  Order continuous dual of ind sys in NRIP is proj of duals in NRiesz}
Let $\cal{D}\defeq \left( (\vlat{E}_\alpha)_{\alpha\in I}, (e_{\alpha, \beta})_{\alpha\precc\beta}\right)$ be a direct system in $\nvlic$, and let $\cal{S} \defeq \left( \vlat{E},(e_\alpha)_{\alpha\in I}\right)$ be the direct limit of $\cal{D}$ in $\vlic$.  The following statements are true.
\begin{itemize}
\item[(i)] $\proj{\ordercontn{\cal{D}}}\defeq \left(\vlat{G},(p_\alpha)_{\alpha\in I}\right)$ exists in $\nvlc$.
\item[(ii)] If $e_{\alpha , \beta}$ is injective for all $\alpha \precc \beta$ in $I$ then $\ordercontn{\left(\ind{\cal{D}}\right)} \cong \proj{\ordercontn{\cal{D}}}$ in $\nvlc$; that is, there exists a lattice isomorphism $S: \ordercontn{\vlat{E}} \to \vlat{G}$ such that the following diagram commutes for all $\alpha\in I$.
\end{itemize}
\begin{eqnarray}\label{EQ:  Thm-Order continuous dual of ind sys in NRIP is proj of duals in NRiesz diagram}
\begin{tikzcd}[cramped]
\ordercontn{\vlat{E}} \arrow[rd, "e^\sim_\alpha"'] \arrow[rr, "S"] & & \vlat{G} \arrow[dl, "p_\alpha"]\\
& \ordercontn{(\vlat{E}_\alpha)}
\end{tikzcd}
\end{eqnarray}
\end{thm}

\begin{proof}
The proof proceeds in a similar fashion to that of Theorem \ref{Thm:  Dual of ind sys in VLIC is proj of duals in NVL}. That (i) is true follows from Proposition \ref{Prop:  Dual system of inductive system is projective system} and Theorem \ref{Thm:  Existence of Projective Limit NVL}~(ii).

For the proof of (ii), assume that $e_{\alpha , \beta}$ is injective for all $\alpha \precc \beta$ in $I$.  Then $\cal{S}$ is the direct limit of $\cal{D}$ in $\nvlic$ by Theorem \ref{Thm:  Existence of Inductive Limits in NVLI}.  Hence, by Proposition \ref{Prop:  Dual compactible systems inductive limit} (ii), $\ordercontn{\cal{S}}$ is a compatible system of $\ordercontn{\cal{D}}$ in $\nvlic$, hence in $\nvlc$.  Therefore there exits a unique normal lattice homomorphism $S:\ordercontn{\vlat{E}}\to \vlat{G}$ so that the diagram (\ref{EQ:  Thm-Order continuous dual of ind sys in NRIP is proj of duals in NRiesz diagram}) commutes.

It follows by exactly the same reasoning as employed in the proof of Theorem \ref{Thm:  Dual of ind sys in VLIC is proj of duals in NVL} that $S$ is injective.  It remains to verify that $S$ maps $\ordercontn{\vlat{E}}$ onto $\vlat{G}$.  Let $(\varphi_\alpha)\in\vlat{G}^+$.  As in the proof of Theorem \ref{Thm:  Dual of ind sys in VLIC is proj of duals in NVL} we define a positive functional $\varphi \in \vlat{E}^\sim$ by setting, for each $u\in \vlat{E}$,
\[
\varphi(u)\defeq \varphi_\alpha (u_\alpha) \text{ if } u = e_\alpha (u_\alpha).
\]
We claim that $\varphi$ is order continuous.  To see that this is so, let $A\downarrow 0$ in $\vlat{E}$.  Without loss of generality, we may assume that $A$ is bounded above by some $0\leq w\in \vlat{E}$.  By Remark \ref{Remark:  Inductive limit notation} (ii) there exist $\alpha\in I$ and $0\leq w_\alpha\in \vlat{E}_\alpha$ so that $e_\alpha(w_\alpha)=w$, and, by Remark \ref{Remark:  Inductive limit notation} (iii), $e_\alpha$ is injective.  Because $e_\alpha$ is also interval preserving, there exists for every $u\in A$ a unique $0\leq u_\alpha \leq w_\alpha$ in $\vlat{E}_\alpha$ so that $e_\alpha (u_\alpha)=u$.  Let $A_\alpha \defeq \{u_\alpha ~:~ u\in A\}$.  Then $A_\alpha \downarrow 0$ in $\vlat{E}_\alpha$. Indeed, let $0\leq v\in\vlat{E}_\alpha$ be a lower bound for $A_\alpha$.  Then $0\leq e_\alpha(v) \leq e_\alpha(u_\alpha)=u$ for all $u\in A$.  Because $A\downarrow 0$ in $\vlat{E}$ it follows that $e_\alpha (v)=0$, hence $v=0$.  By definition of $\varphi$ and the order continuity of $\varphi_\alpha$ we now have $\varphi[A]=\varphi_\alpha[A_\alpha]\downarrow 0$.  Hence $\varphi\in \ordercontn{\vlat{E}}$.

By definition of $\varphi$ and the commutativity of the diagram (\ref{EQ:  Thm-Order continuous dual of ind sys in NRIP is proj of duals in NRiesz diagram}) it follows that $S(\varphi)=(\varphi_\alpha)$.  Therefore $S$ is surjective.
\end{proof}

\begin{remark}\label{Remark:  Direct limit with trivial dual}
Let $\cal{D}\defeq \left( (\vlat{E}_\alpha)_{\alpha\in I}, (e_{\alpha, \beta})_{\alpha\precc\beta}\right)$ be a direct system in $\vlic$, and let $\cal{S} \defeq \left( \vlat{E},(e_\alpha)_{\alpha\in I}\right)$ be the direct limit of $\cal{D}$ in $\vlic$.  In general, it does not follow from $\preann{\orderdual{(\vlat{E}_\alpha)}}=\{0\}$ for all $\alpha\in I$ that $\preann{\orderdual{\vlat{E}}}=\{0\}$, even if all the $\vlat{E}_\alpha$ are non-trivial and the $e_\alpha$ injective.  Indeed, it is well known that $\vlat{L}^0[0,1]$, the space of Lebesgue measurable functions on the unit interval $[0,1]$, has trivial order dual, see for instance \cite[Example 85.1]{Zaanen1983RSII}.  However, by Example \ref{Exm:  Inductive limit of principle ideals}, $\vlat{L}^0[0,1]$ can be expressed as the direct limit of its principal ideals, each of which has a separating order dual.
\end{remark}

In view of the above remark, the following proposition is of interest.

\begin{prop}\label{Prop:  Separating order dual of direct limit}
Let $\cal{D}\defeq \left( (\vlat{E}_\alpha)_{\alpha\in I}, (e_{\alpha, \beta})_{\alpha\precc\beta}\right)$ be a direct system in $\vlic$, and let $\cal{S} \defeq \left( \vlat{E},(e_\alpha)_{\alpha\in I}\right)$ be the direct limit of $\cal{D}$ in $\vlic$.  Assume that for every $\alpha\in I$, $e_\alpha$ is injective and $e_\alpha[\vlat{E}_\alpha]$ is a projection band in $\vlat{E}$.  The following statements are true. \begin{enumerate}
    \item[(i)] If $\preann{\orderdual{(\vlat{E}_\alpha)}}=\{0\}$ for every $\alpha\in I$ then $\preann{\orderdual{\vlat{E}}}=\{0\}$.
    \item[(ii)] If $\preann{\ordercontn{(\vlat{E}_\alpha)}}=\{0\}$ for every $\alpha\in I$ then $\preann{\ordercontn{\vlat{E}}}=\{0\}$.
\end{enumerate}
\end{prop}

\begin{proof}
The proofs of (i) and (ii) are identical, except that for (ii) we note that for all $\alpha\in I$, $e_\alpha$ and $e_\alpha^{-1}$ are order continuous by Proposition \ref{Prop:  Interval Preserving vs Lattice Homomorphism} (i).  We therefore omit the proof of (ii).

 Assume that $\preann{\orderdual{(\vlat{E}_\alpha)}}=\{0\}$ for every $\alpha\in I$.  Let $u\in\vlat{E}$ be non-zero.  Then there exist $\alpha\in I$ and a non-zero $u_\alpha\in \vlat{E}_\alpha$ so that $e_\alpha(u_\alpha)=u$, see Remark \ref{Remark:  Inductive limit notation}.  By assumption there exists $\varphi_\alpha \in\orderdual{\vlat{E}_\alpha}$ so that $\varphi_\alpha(u_\alpha)\neq 0$.  Denote by $P_\alpha:\vlat{E}\to e_\alpha[\vlat{E}_\alpha]$ the projection onto $e_\alpha[\vlat{E}_\alpha]$.  We note that $e_\alpha$ is a lattice isomorphism onto $e_\alpha[\vlat{E}_\alpha]$.  Let $\varphi \defeq (e_\alpha^{-1} \circ P_\alpha)^\sim(\varphi_\alpha)$.  Then $\varphi\in\orderdual{\vlat{E}}$ and, because $u \in e_\alpha[\vlat{E}_\alpha]$, $\varphi(u) = \varphi_\alpha(e_\alpha^{-1}(P_\alpha(u)))=\varphi_\alpha(u_\alpha)\neq 0$.  Hence $\preann{\orderdual{\vlat{E}}}=\{0\}$.
\end{proof}

\subsection{Duals of inverse limits}\label{Subsection:  Duals of projective limits}

We now turn to duals of inverse limits. For inverse systems over $\N$, we prove results analogous to Theorems \ref{Thm:  Dual of ind sys in VLIC is proj of duals in NVL} and \ref{Thm:  Order continuous dual of ind sys in NRIP is proj of duals in NRiesz}. We identify the main obstacle to more general results for inverse systems over arbitrary index sets: Positive (order continuous) functionals defined on a proper sublattice of a vector lattice $\vlat{E}$ do not necessarily extend to $\vlat{E}$.


\begin{defn}\label{Defn:  Dual system of projective system}
Let $\cal{I}\defeq \left( (\vlat{E}_\alpha)_{\alpha\in I}, (p_{\beta, \alpha})_{\beta\scc \alpha}\right)$ be an inverse system in $\vlic$. The \emph{dual system} of $\cal{I}$ is the pair $\cal{I}^\sim \defeq \left( (\vlat{E}^\sim_\alpha)_{\alpha\in I}, (p^\sim_{\beta , \alpha})_{\beta\scc \alpha}\right)$.

If $\cal{I}$ is an inverse system in $\nvlic$, define the \emph{order continuous dual system} of $\cal{I}$ as the pair $\ordercontn{\cal{I}} \defeq \left( (\ordercontn{(\vlat{E}_\alpha)})_{\alpha\in I}, (p^\sim_{\beta , \alpha})_{\beta\scc \alpha}\right)$ with $p^\sim_{\beta , \alpha}: \ordercontn{(E_\alpha)} \to \ordercontn{(\vlat{E}_\beta)}$.
\end{defn}

The following preliminary results, analogous to Propositions \ref{Prop:  Dual system of inductive system is projective system} and \ref{Prop:  Dual compactible systems inductive limit}, are proven in the same way as the corresponding results for direct limits.  As such, we omit the proofs.

\begin{prop}\label{Prop:  Dual system of projective system is inductive system}
Let $\cal{I}\defeq \left( (\vlat{E}_\alpha)_{\alpha\in I}, (p_{\beta, \alpha})_{\beta\scc \alpha}\right)$ be an inverse system in $\vlic$.  The following statements are true. \begin{enumerate}
    \item[(i)] The dual system $\cal{I}^\sim$ is a direct system in $\nvlic$.
    \item[(ii)] If $\cal{I}$ is an inverse system in $\nvlic$ then the order continuous dual system $\ordercontn{\cal{I}}$ is a direct system in $\nvlic$.
\end{enumerate}
\end{prop}

\begin{prop}\label{Prop:  Dual compactible systems projective limit}
Let $\cal{I}\defeq \left( (\vlat{E}_\alpha)_{\alpha\in I}, (p_{\beta, \alpha})_{\beta\scc \alpha}\right)$ be an inverse system in $\vlic$ and  $\cal{S}\defeq \left( \vlat{E}, (p_\alpha)_{\alpha\in I} \right)$ a compatible system of $\cal{I}$ in $\vlic$. The following statements are true. \begin{enumerate}
    \item[(i)] $\cal{S}^\sim \defeq \left( \vlat{E}^\sim, (p^\sim_\alpha)_{\alpha\in I} \right)$ is a compatible system for the direct system $\cal{I}^\sim$ in $\nvlic$.
    \item[(ii)] If $\cal{I}$ is an inverse system in $\nvlic$ and $\cal{S}$ is a compatible system in $\nvlic$, then $\ordercontn{\cal{S}} \defeq \left( \ordercontn{\vlat{E}}, (p^\sim_\alpha)_{\alpha\in I} \right)$ is a compatible system for the direct system $\ordercontn{\cal{I}}$ in $\nvlic$.
\end{enumerate}
\end{prop}

\begin{lem}\label{Lem:  Countable projective system with surjective projections}
Let $\cal{I}\defeq \left( (\vlat{E}_n)_{n\in \N}, (p_{m, n})_{m\geq n}\right)$ be an inverse system in $\vlic$ and let $\cal{S} \defeq \left( \vlat{E}, (p_n)_{n\in \N} \right)$ be the inverse limit of $\cal{I}$ in $\vlc$.  Assume that $p_{m , n}$ is a surjection for all $m \geq n$ in $\N$. Then $p_n$ is surjective and interval preserving for every $n\in\N$. In particular, $\cal{S}$ is a compatible system of $\cal{I}$ in $\vlic$.
\end{lem}

\begin{proof}
Fix $n_0\in \N$. Consider any $u_{n_0}\in \vlat{E}_{n_0}$.  For $n<n_0$ let $u_n = p_{n_0 , n}(u_{n_0})$.  Because $p_{n_0+1 , n_0}$ is a surjection, there exists $u_{n_0+1}\in \vlat{E}_{n_0+1}$ so that $p_{n_0+1 , n_0}(u_{n_0+1}) = u_{n_0}$.  Inductively, for each $n>n_0$ there exists $u_n\in\vlat{E}_n$ so that $p_{n , n-1}(u_n) = u_{n-1}$.

We show that $(u_n)\in\vlat{E}$.  Let $n<m$ in $\N$.  By the definition of an inverse system it follows that $p_{m , n} = p_{n+1 , n} \circ p_{n+2 , n+1}\circ  \cdots \circ p_{m-1 , m-2}\circ p_{m,m-1}$.  It thus follows that $p_{m , n}(u_m) = u_n$ so that $(u_n)\in \vlat{E}$.  We have $p_{n_0}((u_n))=u_{n_0}$ so that $p_{n_0}$ is a surjection.  It follows from Proposition \ref{Prop:  Interval Preserving vs Lattice Homomorphism} (ii) that $p_{n_0}$ is interval preserving. Since $\cal{S}$ is a compatible system of $\cal{I}$ in $\vlc$ and the $p_n$ are interval preserving, we conclude that $\cal{S}$ is a compatible system of $\cal{I}$ in $\vlic$.
\end{proof}

\begin{thm}\label{Thm:  Dual of proj sys in VLIC is ind of duals in NVLIC}
Let $\cal{I}\defeq \left( (\vlat{E}_n)_{n\in \N}, (p_{m, n})_{m\geq n}\right)$ be an inverse system in $\vlic$, and let $\cal{S} \defeq \left( \vlat{E}, (p_n)_{n\in \N} \right)$ be the inverse limit of $\cal{I}$ in $\vlc$. Assume that $p_{m , n}$ is a surjection for all $m \geq n$ in $\N$.  Then the following statements are true. \begin{enumerate}
    \item[(i)] $\ind{\cal{I}^\sim}\defeq \left(\vlat{F},(e_n)_{n\in \N}\right)$ exists in $\nvlic$.
    \item[(ii)] $\left(\proj{\cal{I}}\right)^\sim \cong \ind{\cal{I}^\sim}$ in $\nvlic$; that is, there exists a lattice isomorphism $T: \vlat{F} \to \vlat{E}^\sim$ such that the following diagram commutes for all $n\in \N$.
\end{enumerate}
\[
\begin{tikzcd}[cramped]
\vlat{F} \arrow[rr, "T"] & & \vlat{E}^\sim\\
& \left(\vlat{E}_n\right)^\sim \arrow[lu, "e_{n}"] \arrow[ru, "p_n^\sim"']
\end{tikzcd}
\]
\end{thm}

\begin{proof}
By Proposition \ref{Prop:  Dual system of projective system is inductive system}, $\cal{I}^\sim$ is a direct system in $\nvlic$.   Because the $p_{m , n}$ are surjections their adjoints are injective.  Thus by Theorem \ref{Thm:  Existence of Inductive Limits in NVLI}, $\ind{\cal{I}^\sim}$ exists in $\nvlic$.

We proceed to prove (ii).  Because the maps $p_{m , n}^\sim: \left(\vlat{E}_n\right)^\sim \to \left(\vlat{E}_m\right)^\sim$ are injective, so are the maps $e_n:\left(\vlat{E}_n\right)^\sim \to \vlat{F}$, see Remark \ref{Remark:  Inductive limit notation}.  By Lemma \ref{Lem:  Countable projective system with surjective projections}, each $p_n:\vlat{E}\to\vlat{E}_n$ is surjective and interval preserving, and $\cal{S}$ is a compatible system of $\cal{I}$ in $\vlic$. Therefore $p_n^\sim:\left(\vlat{E}_n\right)^\sim \to\vlat{E}^\sim$ is an injection for every $n$ in $\N$.

By Proposition \ref{Prop:  Dual compactible systems projective limit}, $\cal{S}^\sim = (\vlat{E}^\sim , (p_n^\sim)_{n\in\N})$ is a compatible system of $\cal{I}^\sim$ in $\nvlic$.  Therefore there exists a unique interval preserving normal lattice homomorphism $T : \vlat{F}\to \vlat{E}^\sim$ so that the diagram
\[
\begin{tikzcd}[cramped]
\vlat{F} \arrow[rr, "T"] & & \vlat{E}^\sim\\
& \left(\vlat{E}_n \right) \arrow[lu, "e_{n}"] \arrow[ru, "p_n^\sim"']
\end{tikzcd}
\]
commutes for all $n\in\N$.  We show that $T$ is a lattice isomorphism.

Our first goal is to establish that $T$ is injective.  Consider $\varphi\in \vlat{F}$ so that $T(\varphi)=0$.  There exist an $n\in\N$ and a unique $\varphi_{n}\in \orderdual{\left(\vlat{E}_n\right)}$ so that $e_{n}(\varphi_n) = \varphi$.  Then $p_n^\sim (\varphi_n) = T (e_n(\varphi_n)) = T(\varphi)=0$.  But $p_n^\sim$ is injective so that $\varphi_n=0$, hence $\varphi=e_n(\varphi_n)=0$.

It remains to show that $T$ maps $\vlat{F}$ onto $\vlat{E}^\sim$.  This follows from
\[
	\vlat{E}^\sim = \displaystyle\bigcup p_n^\sim\left[ \left(\vlat{E}_n\right)^\sim \right],
\]
a fact which we now establish.  Suppose that $\vlat{E}^\sim \neq \displaystyle\bigcup p_n^\sim[\left(\vlat{E}_n\right)^\sim]$.  Because $p_n^\sim$ is an interval preserving lattice homomorphism for every $n\in\N$, each $p_n^\sim[\left(\vlat{E}_n\right)^\sim]$, and hence $\displaystyle\bigcup p_n^\sim[\left(\vlat{E}_n\right)^\sim]$, is a solid subset of $\vlat{E}^\sim$.  Therefore, because $\vlat{E}^\sim \neq \displaystyle\bigcup p_n^\sim[\left(\vlat{E}_n\right)^\sim]$, there exists $0\leq \psi\in \vlat{E}^\sim \setminus \displaystyle\bigcup p_n^\sim[\left(\vlat{E}_n\right)^\sim]$.  By Proposition \ref{Prop:  Image of adjoint of lattice homomorphism} (i), $p_n^\sim[\left(\vlat{E}_n\right)^\sim]=\ann{\ker(p_n)}$ for every $n\in \N$ so that $\psi\notin \ann{\ker(p_n)}$ for $n\in\N$.  Hence, for every $n\in\N$, there exists $0\leq u^{(n)}\in\ker(p_n)$ so that $\psi(u^{(n)})=1$.  We claim that there exists $w\in\vlat{E}$ so that $w \geq u^{(1)}+\cdots + u^{(n)}$ for all $n\in\N$.  This claim leads to $\psi(w) \geq \psi(u^{(1)}+\cdots + u^{(n)}) = n$ for every $n\in\N$, which is impossible, so that $\vlat{E}^\sim = \displaystyle\bigcup p_n^\sim\left[ \left(\vlat{E}_n\right)^\sim \right]$.

For each $n\in\N$, write $u^{(n)} = (u_m^{(n)})\in\vlat{E}^+ \subseteq \dsprod\vlat{E}_m$.  Let $w_m \defeq u_m^{(1)} + \cdots + u_m^{(m)}$ for every $m\in\N$, and $w\defeq (w_m)$.  If $n>m$ then $u_m^{(n)} = p_{n , m}(p_n(u^{(n)})) = 0$ because $u^{(n)}\in\ker(p_n)$.  Since $u_{m}^{(n)}\geq 0$ for all $m,n\in\N$ we therefore have $w_m \geq u^{(1)}_m+\ldots +u^{(n)}_m$ for all $m,n\in\N$ so that $w\geq u^{(1)}+ \cdots + u^{(n)}$ for every $n\in\N$.  To see that $w\in \vlat{E}$ consider $m_1\geq m_0$ in $\N$.  Then
\[
p_{m_1 , m_0}(w_{m_1}) = p_{m_1 , m_0}(u_{m_1}^{(1)}) + \cdots + p_{m_1 , m_0}(u_{m_1}^{(m_1)}).
\]
But $u^{(n)}=(u_{m}^{(n)})\in\vlat{E}$ for all $n\in\N$, so
\[
p_{m_1 , m_0}(w_{m_1}) = u_{m_0}^{(1)} + \cdots + u_{m_0}^{(m_1)}.
\]
Finally, because $u_m^{(n)}=0$ for all $n>m$ in $\N$ we have
\[
p_{m_1 , m_0}(w_{m_1}) = u_{m_0}^{(1)} + \cdots + u_{m_0}^{(m_0)} = w_{m_0}.
\]
Hence $w\in \vlat{E}$, which verifies our claim.  This completes the proof.
\end{proof}

\begin{thm}\label{Thm:  Order continuous dual of proj sys in VLIC is ind of duals in NVLIC}
Let $\cal{I}\defeq \left( (\vlat{E}_n)_{n\in \N}, (p_{m, n})_{m\geq n}\right)$ be an inverse system  in $\nvlic$, and let $\cal{S} \defeq \left( \vlat{E}, (p_n)_{n\in \N} \right)$ be the inverse limit of $\cal{I}$ in $\vlc$.  Assume that $p_n$ is order continuous and $\vlat{E}_n$ is Archimedean for each $n\in \N$, and that $p_{m , n}$ is a surjection for all $m \geq n$ in $\N$.  The following statements are true.
\begin{enumerate}
    \item[(i)] $\ind{\ordercontn{\cal{I}}}\defeq \left(\vlat{G},(e_n)_{n\in \N}\right)$ exists in $\nvlic$.
    \item[(ii)] $\ordercontn{\left(\proj{\cal{I}}\right)} \cong \ind{\ordercontn{\cal{I}}}$ in $\nvlic$; that is, there exists a lattice isomorphism $S: \vlat{G} \to \ordercontn{\vlat{E}}$ such that the following diagram commutes for all $n\in \N$.
\end{enumerate}
\[
\begin{tikzcd}[cramped]
\vlat{G} \arrow[rr, "S"] & & \ordercontn{\vlat{E}}\\
& \ordercontn{(\vlat{E}_n)} \arrow[lu, "e_{n}"] \arrow[ru, "p_n^\sim"']
\end{tikzcd}
\]
\end{thm}


\begin{proof}
The existence of $\ind{\ordercontn{\cal{I}}}$ in $\nvlic$ follows by the same reasoning as given in Theorem \ref{Thm:  Dual of proj sys in VLIC is ind of duals in NVLIC}.

For (ii), as in the proof of Theorem \ref{Thm:  Dual of proj sys in VLIC is ind of duals in NVLIC}, we see that $e_n:\ordercontn{(\vlat{E}_n)} \to \vlat{G}$ and $p_n^\sim:\ordercontn{(\vlat{E}_n)}\to\ordercontn{\vlat{E}}$ are injective interval preserving normal lattice homomorphisms for all $n\in \N$.  In addition, $\cal{S}$ is a compatible system for $\cal{I}$ in $\nvlic$.

By Proposition \ref{Prop:  Dual compactible systems projective limit} (ii), $\ordercontn{\cal{S}} = (\ordercontn{\vlat{E}} , (p_n^\sim)_{n\in\N})$ is a compatible system of $\ordercontn{\cal{I}}$ in $\nvlic$. Therefore there exists a unique interval preserving normal lattice homomorphism $S : \vlat{G}\to \ordercontn{\vlat{E}}$ so that the diagram
\[
\begin{tikzcd}[cramped]
\vlat{G} \arrow[rr, "S"] & & \ordercontn{\vlat{E}}\\
& \ordercontn{(\vlat{E}_n)} \arrow[lu, "e_{n}"] \arrow[ru, "p_n^\sim"']
\end{tikzcd}
\]
commutes for all $n\in\N$. Exactly the same argument as used in the proof of Theorem \ref{Thm:  Dual of proj sys in VLIC is ind of duals in NVLIC} shows that $S$ is a lattice isomorphism, this time making use of Proposition \ref{Prop:  Image of adjoint of lattice homomorphism} (ii).
\end{proof}

Theorems \ref{Thm:  Dual of proj sys in VLIC is ind of duals in NVLIC} and \ref{Thm:  Order continuous dual of proj sys in VLIC is ind of duals in NVLIC} cannot be generalised to systems over an arbitrary directed set $I$.  Indeed, the assumption that the inverse system $\cal{I}$ is indexed by the natural numbers is used in essential ways to show that the lattice homomorphisms $T$ and $S$ in Theorems \ref{Thm:  Dual of proj sys in VLIC is ind of duals in NVLIC} and \ref{Thm:  Order continuous dual of proj sys in VLIC is ind of duals in NVLIC}, respectively, are both injective and surjective.  The injectivity of $S$ and $T$ follows from the surjectivity of the maps $p_n$, which in turn follows from Lemma \ref{Lem:  Countable projective system with surjective projections} where the total ordering of $\N$ is used explicitly.  We are not aware of any conditions on an inverse system $\cal{I}$ in $\vlc$, indexed over an arbitrary directed set, which implies that the projections from $\proj{\cal{I}}$ into the component spaces are surjective.  Furthermore, the method of proof for surjectivity of $S$ and $T$ cannot be generalised to systems over arbitrary directed sets.  As we show next, this issue is related to the extension of order bounded linear functionals.

\begin{thm}\label{Thm:  Characterisation dual of projective limit is inductive limit of duals}
Let $\cal{I}\defeq \left( (\vlat{E}_\alpha)_{\alpha\in I}, (p_{\beta, \alpha})_{\beta \scc \alpha}\right)$ be an inverse system in $\vlic$ and $\cal{S} \defeq \left( \vlat{E}, (p_\alpha)_{\alpha\in I} \right)$ its inverse limit in $\vlc$.   Assume that $p_{\beta , \alpha}$ and $p_\alpha$ are surjections for all $\beta\scc \alpha$ in $I$.  Then the following statements are true.
\begin{enumerate}
    \item[(i)] $\ind{\cal{I}^\sim}\defeq \left(\vlat{F},(e_\alpha)_{\alpha\in I}\right)$ exists in $\nvlic$.
    \item[(ii)] There exists an injective interval preserving normal lattice homomorphism $T:\vlat{F}\to \vlat{E}^\sim$ so that the diagram
        \[
        \begin{tikzcd}[cramped]
        \vlat{F} \arrow[rr, "T"] & & \vlat{E}^\sim\\
        & \vlat{E}_\alpha^\sim \arrow[lu, "e_{\alpha}"] \arrow[ru, "p_\alpha^\sim"']
        \end{tikzcd}
        \]
        commutes for every $\alpha \in I$.
    \item[(iii)] If $T$ is a bijection, hence a lattice isomorphism, then every order bounded linear functional on $\vlat{E}$ has an order bounded linear extension to $\dsprod \vlat{E}_\alpha$.  The converse is true if $I$ has non-measurable cardinal.
\end{enumerate}
\end{thm}

\begin{proof}
That (i) and (ii) are true follow as in the proof of Theorem \ref{Thm:  Dual of proj sys in VLIC is ind of duals in NVLIC}.  We verify (iii).

Let $\iota:\vlat{E}\to\dsprod \vlat{E}_\alpha$ be the inclusion map.  The diagram
\[
\begin{tikzcd}[cramped]
\vlat{E} \arrow[rd, "p_\alpha"'] \arrow[rr, "\iota"] & & \dsprod\vlat{E}_\alpha \arrow[dl, "\pi_\alpha"]\\
& \vlat{E}_\alpha
\end{tikzcd}
\]
commutes for every $\alpha\in I$, and therefore the diagram
\[
\begin{tikzcd}[cramped]
\left(\dsprod\vlat{E}_\alpha\right)^\sim  \arrow[rr, "\iota^\sim"] & & \vlat{E}^\sim  \\
& \vlat{E}_\alpha^\sim \arrow[lu, "\pi_{\alpha}^\sim"] \arrow[ru, "p_\alpha^\sim"']
\end{tikzcd}
\]
also commutes for each $\alpha\in I$.  Hence, for all $\alpha\in I$, the diagram
\[
\begin{tikzcd}[cramped]
\left(\dsprod\vlat{E}_\alpha\right)^\sim  \arrow[rr, "\iota^\sim"] & &  \vlat{E}^\sim &\\
& \vlat{E}^\sim_\alpha \arrow[ul, "\pi_\alpha^\sim"] \arrow[ur, "p_\alpha^\sim"'] \arrow[rr, "e_\alpha"'] & & \vlat{F} \arrow[ul, "T"']
\end{tikzcd}
\]
commutes.

Assume that $T$ is a lattice isomorphism, and therefore a surjection.  Let $\varphi\in \vlat{E}^\sim$.  There exists a $\psi\in \vlat{F}$ so that $T(\psi)=\varphi$.  By Remark \ref{Remark:  Inductive limit notation}, there exist $\alpha\in I$ and $\psi_\alpha\in\vlat{E}^\sim_\alpha$ so that $e_\alpha(\psi_\alpha)=\psi$.  Then
\[
\iota^\sim (\pi_\alpha^\sim(\psi_\alpha)) = p_\alpha^\sim (\psi_\alpha) = T(e_\alpha(\psi_\alpha)) = \varphi.
\]
Therefore $\iota^\sim$ is a surjection; that is, every $\varphi\in \vlat{E}^\sim$ has an order bounded linear extension to $\dsprod \vlat{E}_\alpha$.

Assume that $I$ has non-measurable cardinal, and every order bounded linear functional on $\vlat{E}$ extends to an order bounded linear functional on $\dsprod \vlat{E}_\alpha$.  Then $\iota^\sim$, which acts as restriction of functionals on $\dsprod \vlat{E}_\alpha$ to $\vlat{E}$, is a surjection.  Fix $\varphi\in \vlat{E}^\sim$.  There exists $\psi \in \left(\dsprod\vlat{E}_\alpha\right)^\sim$ so that $\varphi = \iota^\sim (\psi)$.  By Theorem \ref{Thm:  Properties of product of vector lattices.} (iv) there exist $\alpha_1,\ldots,\alpha_n\in I$ and $\psi_1\in \vlat{E}_{\alpha_1}^\sim,\ldots,\psi_n\in \vlat{E}_{\alpha_n}^\sim$ so that $\psi = \pi_{\alpha_1}^\sim (\psi_{\alpha_1}) + \ldots + \pi_{\alpha_n}^\sim (\psi_{\alpha_n})$.  Then
\[
\varphi =\iota^\sim \left( \sum_{i=1}^n \pi_{\alpha_i}^\sim(\psi_i)\right) =  \sum_{i=1}^n \iota^\sim(\pi_{\alpha_i}^\sim(\psi_i)) = \sum_{i=1}^n p_{\alpha_i}^\sim(\psi_i) = \sum_{i=1}^n T(e_{\alpha_i}(\psi_i)) = T\left( \sum_{i=1}^n e_{\alpha_i}(\psi_i)\right).
\]
Therefore $T$ is surjective, and hence a lattice isomorphism.
\end{proof}

A similar result holds for the order continuous dual of an inverse limit.  We omit the proof of the next theorem, which is virtually identical to that of Theorem \ref{Thm:  Characterisation dual of projective limit is inductive limit of duals}.  Note, however, that unlike in Theorem \ref{Thm:  Characterisation dual of projective limit is inductive limit of duals}, we make no assumption on the cardinality of $I$.

\begin{thm}\label{Thm:  Characterisation order continuous dual of projective limit is inductive limit of duals}
Let $\cal{I}\defeq \left( (\vlat{E}_\alpha)_{\alpha\in I}, (p_{\beta, \alpha})_{\beta \scc \alpha}\right)$ be an inverse system in $\nvlic$ and $\cal{S} \defeq \left( \vlat{E}, (p_\alpha)_{\alpha\in I} \right)$ its inverse limit in $\vlc$.   Assume that $p_{\beta , \alpha}$ and $p_\alpha$ are surjections for all $\beta \scc \alpha$ in $I$, and that each $p_\alpha$ is order continuous.  Then the following statements are true.
\begin{enumerate}
    \item[(i)] $\ind{\ordercontn{\cal{I}}}\defeq \left( \vlat{G}, (e_\alpha)_{\alpha \in I} \right)$ exists in $\nvlic$.
    \item[(ii)] There exists an injective and interval preserving normal lattice homomorphism $S:\vlat{G}\to \ordercontn{\vlat{E}}$ so that the diagram
        \[
        \begin{tikzcd}[cramped]
        \vlat{G} \arrow[rr, "S"] & & \ordercontn{\vlat{E}}\\
        & \ordercontn{(\vlat{E}_\alpha)} \arrow[lu, "e_{\alpha}"] \arrow[ru, "p_\alpha^\sim"']
        \end{tikzcd}
        \]
        commutes for every $\alpha \in I$.
    \item[(iii)] $S$ is a lattice isomorphism if and only if every order continuous linear functional on $\vlat{E}$ has an order continuous linear extension to $\dsprod \vlat{E}_\alpha$.
\end{enumerate}
\end{thm}

The following two results are immediate consequences of Theorems \ref{Thm:  Characterisation dual of projective limit is inductive limit of duals} and \ref{Thm:  Characterisation order continuous dual of projective limit is inductive limit of duals}, respectively.

\begin{cor}\label{Cor:  Dual of projectime limit is inductive limit}
Let $\cal{I}\defeq \left( (\vlat{E}_\alpha)_{\alpha\in I}, (p_{\beta, \alpha})_{\alpha\precc \beta}\right)$ be an inverse system in $\vlic$, $\cal{S} \defeq \left( \vlat{E}, (p_\alpha)_{\alpha\in I} \right)$ its inverse limit in $\vlc$ and $(\vlat{F},(e_\alpha)_{\alpha\in I})$ the direct limit of $\cal{I}^\sim$ in $\nvlic$.   Assume that $p_{\beta , \alpha}$ and $p_\alpha$ are surjections for all $\beta\scc\alpha$ in $I$.  If $\vlat{E}$ is majorising in $\dsprod \vlat{E}_\alpha$ then $\left(\proj{\cal{I}}\right)^\sim \cong \ind{\cal{I}^\sim}$ in $\nvlic$; that is, there exists a lattice isomorphism $T: \vlat{F} \to \vlat{E}^\sim$ such that the diagram
\[
\begin{tikzcd}[cramped]
\vlat{F} \arrow[rr, "T"] & & \vlat{E}^\sim\\
& \vlat{E}_\alpha^\sim \arrow[lu, "e_{\alpha}"] \arrow[ru, "p_\alpha^\sim"']
\end{tikzcd}
\]
commutes for all $\alpha\in I$.
\end{cor}


\begin{proof}
According to \cite[Theorem 1.32]{AliprantisBurkinshaw2006}, every order bounded functional on $\vlat{E}$ extends to an order bounded functional on $\dsprod \vlat{E}_\alpha$.  Therefore the result follows directly from Theorem \ref{Thm:  Characterisation dual of projective limit is inductive limit of duals}.
\end{proof}

\begin{cor}\label{Cor:  Dual of projectime limit is inductive limit}
Let $\cal{I}\defeq \left( (\vlat{E}_\alpha)_{\alpha\in I}, (p_{\beta, \alpha})_{\alpha\precc \beta}\right)$ be an inverse system in $\nvlic$, $\cal{S} \defeq \left( \vlat{E}, (p_\alpha)_{\alpha\in I} \right)$ its inverse limit in $\vlc$ and $(\vlat{F},(e_\alpha)_{\alpha\in I})$ the direct limit of $\ordercontn{\cal{I}}$ in $\nvlic$.   Assume that $p_{\beta , \alpha}$ and $p_\alpha$ are surjections for all $\beta\scc\alpha$ in $I$, and that each $p_\alpha$ is order continuous.  If $\vlat{E}$ is majorising and order dense in $\dsprod \vlat{E}_\alpha$ then $\ordercontn{\left(\proj{\cal{I}}\right)} \cong \ind{\ordercontn{\cal{I}}}$ in $\nvlic$; that is, there exists a lattice isomorphism $S: \vlat{F} \to \ordercontn{\vlat{E}}$ such that the diagram
\[
\begin{tikzcd}[cramped]
\vlat{F} \arrow[rr, "S"] & & \ordercontn{\vlat{E}}\\
& \ordercontn{\left(\vlat{E}_\alpha\right)} \arrow[lu, "e_{\alpha}"] \arrow[ru, "p_\alpha^\sim"']
\end{tikzcd}
\]
commutes for all $\alpha\in I$.
\end{cor}


\begin{proof}
According to \cite[Theorem 1.65]{AliprantisBurkinshaw2006}, every order continuous functional on $\vlat{E}$ extends to an order continuous functional on $\dsprod \vlat{E}_\alpha$.  Therefore the result follows directly from Theorem \ref{Thm:  Characterisation order continuous dual of projective limit is inductive limit of duals}.
\end{proof}

In contradistinction with direct limits, the inverse limit construction always preserves the property of having a separating order (continuous) dual.

\begin{prop}\label{Prop:  Separating order dual of inverse limit}
Let $\cal{I}\defeq \left( (\vlat{E}_\alpha)_{\alpha\in I}, (p_{\beta, \alpha})_{\beta \scc \alpha}\right)$ be an inverse system in $\vlc$ and $\cal{S} \defeq \left( \vlat{E}, (p_\alpha)_{\alpha\in I} \right)$ its inverse limit in $\vlc$.  Then the following statements are true.  \begin{enumerate}
    \item[(i)] If $\preann{\orderdual{(\vlat{E}_\alpha)}}=\{0\}$ for every $\alpha \in I$ then $\preann{\orderdual{\vlat{E}}}=\{0\}$.
    \item[(ii)] If $\preann{\ordercontn{(\vlat{E}_\alpha)}}=\{0\}$ and $p_\alpha$ is order continuous for every $\alpha \in I$ then $\preann{\ordercontn{\vlat{E}}}=\{0\}$.
\end{enumerate}
\end{prop}

\begin{proof}
The proofs of (i) and (ii) are identical.  Hence we omit the proof of (ii).

Assume that $\preann{\orderdual{(\vlat{E}_\alpha)}}=\{0\}$ for every $\alpha \in I$.  Let $u\in\vlat{E}$ be non-zero.  Then there exists $\alpha \in I$ so that $p_\alpha(u)\neq 0$.  Since $\preann{\orderdual{(\vlat{E}_\alpha)}}=\{0\}$, there exists $\varphi\in \orderdual{(\vlat{E}_\alpha)}$ so that $\varphi(p_\alpha (u))\neq 0$; that is, $p_\alpha^\sim(\varphi)(u)\neq 0$.  Hence $\preann{\orderdual{\vlat{E}}}=\{0\}$.
\end{proof}

\section{Applications}\label{Section:  Applications}

In this section we apply the duality results for direct and inverse limits obtained in Section \ref{Section:  Dual spaces}.  In particular, we consider order (continuous) duals of some of the function spaces which are expressed as direct and inverse limits in Sections \ref{Subsection:  Examples of inductive limits} and \ref{Subsection:  Examples of projective limits}, respectively.  This is followed by an investigation of perfect spaces. We show that, under certain conditions, the direct and inverse limits of perfect spaces are perfect.  We then specialise these results to the case of $\cont(X)$ and obtain a solution to the decomposition problem mentioned in the introduction.  Finally, we show that an  Archimedean vector lattice has a relatively uniformly complete order predual if and only if it can be expressed, in a suitable way, as an inverse limit of spaces of Radon measures on compact Hausdorff spaces.

The following two simple propositions are used repeatedly.  These results are proved in \cite[p. 193, p.~205]{Bourbakie_Theory_of_sets} in the context of direct and inverse systems of sets.  The arguments in \cite{Bourbakie_Theory_of_sets} suffice to verify the results in the vector lattice context, so we do not repeat them here.

\begin{prop}\label{Prop:  Morphisms between inductive limits}
Let $\cal{D}\defeq \left( (\vlat{E}_\alpha)_{\alpha\in I}, (e_{\alpha , \beta})_{\alpha\precc \beta}\right)$ and $\cal{D}'\defeq \left( (\vlat{E}_\alpha')_{\alpha\in I}, (e_{\alpha , \beta}')_{\alpha\precc \beta}\right)$ be direct systems in $\vlc$ with direct limits $\cal{S} \defeq \left( \vlat{E}, (e_\alpha)_{\alpha\in I} \right)$ and $\cal{S}' \defeq \left( \vlat{E}', (e_\alpha')_{\alpha\in I} \right)$ in $\vlc$.  Assume that for every $\alpha\in I$ there exists a lattice homomorphism ${T_\alpha :\vlat{E}_\alpha \to \vlat{E}_\alpha'}$ so that the diagram
\begin{eqnarray}
\begin{tikzcd}[cramped]
\vlat{E}_\alpha \arrow[rr, "T_\alpha"] \arrow[dd, "e_{\alpha , \beta}"'] & & \vlat{E}_\alpha' \arrow[dd, "e_{\alpha , \beta}' "]\\
 & & \\
\vlat{E}_\beta \arrow[rr, "T_\beta"']  & & \vlat{E}_\beta'
\end{tikzcd}\label{EQ:  Prop-Morphisms between inductive limits 1}
\end{eqnarray}
commutes for all $\alpha\precc \beta$ in $I$.  The following statements are true. \begin{enumerate}
    \item[(i)] There exists a unique lattice homomorphism $T:\vlat{E}\to\vlat{E}'$ so that the diagram
    \begin{eqnarray}
    \begin{tikzcd}[cramped]
    \vlat{E}_\alpha \arrow[rr, "T_\alpha"] \arrow[dd, "e_{\alpha}"'] & & \vlat{E}_\alpha' \arrow[dd, "e_{\alpha}' "]\\
     & & \\
    \vlat{E} \arrow[rr, "T"']  & & \vlat{E}'
    \end{tikzcd}\label{EQ:  Prop-Morphisms between inductive limits 2}
    \end{eqnarray}
    commutes for every $\alpha\in I$.
    \item[(ii)] If $T_\alpha$ is a lattice isomorphism for every $\alpha\in I$, then so is $T$.
\end{enumerate}
\end{prop}

\begin{prop}\label{Prop:  Morphisms between projective limits}
Let $\cal{I}\defeq \left( (\vlat{E}_\alpha)_{\alpha\in I}, (p_{\beta, \alpha})_{\beta \scc \alpha}\right)$ and $\cal{I}'\defeq \left( (\vlat{E}_\alpha')_{\alpha\in I}, (p_{\beta, \alpha}')_{\beta \scc \alpha}\right)$ be inverse systems in $\vlc$ with inverse limits $\cal{S} \defeq \left( \vlat{E}, (p_\alpha)_{\alpha\in I} \right)$ and $\cal{S}' \defeq \left( \vlat{E}', (p_\alpha')_{\alpha\in I} \right)$ in $\vlc$.  Assume that for every $\alpha\in I$ there exists a lattice homomorphism $T_\alpha :\vlat{E}_\alpha \to \vlat{E}_\alpha'$ so that the diagram
\begin{eqnarray}
\begin{tikzcd}[cramped]
\vlat{E}_\beta \arrow[rr, "T_\beta"] \arrow[dd, "p_{\beta, \alpha}"'] & & \vlat{E}_\beta' \arrow[dd, "p_{\beta,\alpha}' "]\\
 & & \\
\vlat{E}_\alpha \arrow[rr, "T_\alpha"']  & & \vlat{E}_\alpha'
\end{tikzcd}\label{EQ:  Prop-Morphisms between projective limits 1}
\end{eqnarray}
commutes for all $\alpha\precc \beta$ in $I$.  The following statements are true. \begin{enumerate}
    \item[(i)] There exists a unique lattice homomorphism $T:\vlat{E}\to\vlat{E}'$ so that the diagram
    \begin{eqnarray}
    \begin{tikzcd}[cramped]
    \vlat{E} \arrow[rr, "T"] \arrow[dd, "p_{\alpha}"'] & & \vlat{E}' \arrow[dd, "p_{\alpha}' "]\\
     & & \\
    \vlat{E}_\alpha \arrow[rr, "T_\alpha"']  & & \vlat{E}_\alpha'
    \end{tikzcd}\label{EQ:  Prop-Morphisms between projective limits 2}
    \end{eqnarray}
    commutes for every $\alpha\in I$.
    \item[(ii)] If $T_\alpha$ is a lattice isomorphism for every $\alpha\in I$, then so is $T$.
\end{enumerate}
\end{prop}

\subsection{Duals of function spaces}\label{Subsection:  Duals of function spaces}

In this section we apply the duality results in Section \ref{Section:  Dual spaces} to the examples in Sections {\ref{Subsection:  Examples of inductive limits} and \ref{Subsection:  Examples of projective limits} to obtain characterizations of the order and order continuous duals of some function spaces.  These results follow immediately from the corresponding examples and the appropriate duality result.

\begin{thm}
Let $(X,\Sigma,\mu)$ be a complete $\sigma$-finite measure space.  Let $\Xi \defeq (X_n)$ be an increasing sequence (w.r.t. inclusion) of measurable sets with positive measure so that $X=\displaystyle \bigcup X_n$.  Let $1\leq p <\infty$ and $1\leq q\leq \infty$ satisfy $\frac{1}{p}+\frac{1}{q}=1$.  For $n\in\N$ let $e_n$ and $r_n$ be as in Examples \ref{Exm:  Locally supported p-summable functions as inductive limit} and \ref{Exm:  Lploc projective limit}, respectively.

For every $n\in\N$, let $T_n:\vlat{L}^q(X_n)\to \vlat{L}^p(X_n)^\sim$ be the usual (isometric) lattice isomorphism,
\[
T_n(u)(v) = \int_{X_n} uv\thinspace d\mu ,~~ u\in \vlat{L}^q(X_n),~ v\in \vlat{L}^p(X_n).
\]
There exists a unique lattice isomorphism $T:\vlat{L}^q_{\Xi-\loc}(X)\to \vlat{L}^p_{\Xi-c}(X)^\sim$ so that the diagram
\[
\begin{tikzcd}[cramped]
\vlat{L}^q_{\Xi-\loc}(X) \arrow[rr, "T"] \arrow[dd, "r_n"'] & & \vlat{L}^p_{\Xi-c}(X)^\sim \arrow[dd, "e_n^\sim"]\\
 & & \\
\vlat{L}^q(X_n) \arrow[rr, "T_n"']  & & \vlat{L}^p(X_n)^\sim
\end{tikzcd}
\]
commutes for every $n\in\N$.
\end{thm}

\begin{proof}
The result follows immediately from Examples \ref{Exm:  Locally supported p-summable functions as inductive limit} and \ref{Exm:  Lploc projective limit}, Theorem \ref{Thm:  Dual of ind sys in VLIC is proj of duals in NVL} and Proposition \ref{Prop:  Morphisms between projective limits}.
\end{proof}

\begin{thm}
Let $(X,\Sigma,\mu)$ be a complete $\sigma$-finite measure space.  Let $\Xi \defeq (X_n)$ be an increasing sequence (w.r.t. inclusion) of measurable sets with positive measure so that $X = \displaystyle \bigcup X_n$.  Let $1\leq p\leq \infty$ and $1\leq q\leq \infty$ satisfy $\frac{1}{p}+\frac{1}{q}=1$.  For $n\in\N$ let $e_n$ and $r_n$ be as in Examples \ref{Exm:  Locally supported p-summable functions as inductive limit} and \ref{Exm:  Lploc projective limit}, respectively.

For every $n\in\N$, let $S_n:\vlat{L}^q(X_n)\to \ordercontn{\vlat{L}^p(X_n)}$ be the usual (isometric) lattice isomorphism,
\[
S_n(u)(v) = \int_{X_n}uv\thinspace d\mu,~~ u\in \vlat{L}^q(X_n),~ v\in \vlat{L}^p(X_n).
\]
There exists a unique lattice isomorphism $S:\vlat{L}^q_{\Xi-\loc}(X)\to \ordercontn{\vlat{L}^p_{\Xi-c}(X)}$ so that the diagram
\[
\begin{tikzcd}[cramped]
\vlat{L}^q_{\Xi-\loc}(X) \arrow[rr, "S"] \arrow[dd, "r_n"'] & & \ordercontn{\vlat{L}^p_{\Xi-c}(X)} \arrow[dd, "e_n^\sim"]\\
 & & \\
\vlat{L}^q(X_n) \arrow[rr, "S_n"']  & & \ordercontn{\vlat{L}^p(X_n)}
\end{tikzcd}
\]
commutes for every $n\in\N$.
\end{thm}

\begin{proof}
We recall that the mappings $e_{n,m}$ in Example \ref{Exm:  Locally supported p-summable functions as inductive limit} are injective for all $n\leq m$ in $\N$.  Therefore the result follows immediately from Examples \ref{Exm:  Locally supported p-summable functions as inductive limit} and \ref{Exm:  Lploc projective limit}, Theorem \ref{Thm:  Order continuous dual of ind sys in NRIP is proj of duals in NRiesz} and Proposition \ref{Prop:  Morphisms between projective limits}.
\end{proof}

\begin{thm}
Let $(X,\Sigma,\mu)$ be a complete $\sigma$-finite measure space.  Let $\Xi \defeq (X_n)$ be an increasing sequence (w.r.t. inclusion) of measurable sets with positive measure so that $X=\displaystyle \bigcup X_n$.  Let $1\leq p<\infty$ and $1\leq q\leq \infty$ satisfy $\frac{1}{p}+\frac{1}{q}=1$.  For $n\in\N$ let $e_n$ and $r_n$ be as in Examples \ref{Exm:  Locally supported p-summable functions as inductive limit} and \ref{Exm:  Lploc projective limit}, respectively.

For every $n\in\N$, let $T_n:\vlat{L}^q(X_n)\to \vlat{L}^p(X_n)^\sim$ be the usual (isometric) lattice isomorphism,
\[
T_n(u)(v) = \int_{X_n}uv\thinspace d\mu, ~~ u\in \vlat{L}^q(X_n),~ v\in \vlat{L}^p(X_n).
\]
There exists a unique lattice isomorphism $R:\vlat{L}_{\Xi-c}^q(X)\to \vlat{L}^p_{\Xi-\loc}(X)^\sim$ so that the diagram
\[
\begin{tikzcd}[cramped]
\vlat{L}^q(X_n) \arrow[rr, "T_n"] \arrow[dd, "e_n"'] & & \vlat{L}^p(X_n)^\sim \arrow[dd, "r_n^\sim"]\\
 & & \\
\vlat{L}^q_{\Xi-c}(X) \arrow[rr, "R"']  & & \vlat{L}^p_{\Xi-\loc}(X)^\sim
\end{tikzcd}
\]
commutes for every $n\in\N$.
\end{thm}

\begin{proof}
We recall that the mappings $p_{m,n}$ in Example \ref{Exm:  Lploc projective limit} are surjective for all $m\geq n$ in $\N$.  Therefore the result follows immediately from Examples \ref{Exm:  Locally supported p-summable functions as inductive limit} and \ref{Exm:  Lploc projective limit}, Theorem \ref{Thm:  Dual of proj sys in VLIC is ind of duals in NVLIC} and Proposition \ref{Prop:  Morphisms between inductive limits}.
\end{proof}

\begin{thm}
Let $(X,\Sigma,\mu)$ be a complete $\sigma$-finite measure space.  Let $\Xi \defeq (X_n)$ be an increasing sequence (w.r.t. inclusion) of measurable sets with positive measure so that $X=\displaystyle \bigcup X_n$.  Let $1\leq p\leq \infty$ and $1\leq q\leq \infty$ satisfy $\frac{1}{p}+\frac{1}{q}=1$.    For $n\in\N$ let $e_n$ and $r_n$ be as in Examples \ref{Exm:  Locally supported p-summable functions as inductive limit} and \ref{Exm:  Lploc projective limit}, respectively.

For every $n\in\N$, let $S_n:\vlat{L}^q(X_n)\to \ordercontn{\vlat{L}^p(X_n)}$ be the usual (isometric) lattice isomorphism,
\[
S_n(u)(v) = \int_{X_n}uv\thinspace d\mu, ~~ u\in \vlat{L}^q(X_n),~ v\in \vlat{L}^p(X_n).
\]
There exists a unique lattice isomorphism $Q:\vlat{L}^p_{\Xi-c}(X) \to \ordercontn{\vlat{L}^q_{\Xi-\loc}(X)}$ so that the diagram
\[
\begin{tikzcd}[cramped]
\vlat{L}^q(X_n) \arrow[rr, "S_n"] \arrow[dd, "e_n"'] & & \ordercontn{\vlat{L}^p(X_n)} \arrow[dd, "r_n^\sim"]\\
 & & \\
\vlat{L}^q_{\Xi-c}(X) \arrow[rr, "Q"']  & & \ordercontn{\vlat{L}^p_{\Xi-\loc}(X)}
\end{tikzcd}
\]
commutes for every $n\in\N$.
\end{thm}

\begin{proof}
Because the mappings $p_{m,n}$ in Example \ref{Exm:  Lploc projective limit} are surjective for all $m\geq n$ in $\N$, the result follows immediately from Examples \ref{Exm:  Locally supported p-summable functions as inductive limit} and \ref{Exm:  Lploc projective limit}, Theorem \ref{Thm:  Order continuous dual of proj sys in VLIC is ind of duals in NVLIC} and Proposition \ref{Prop:  Morphisms between inductive limits}.
\end{proof}

The next two results are special cases of Theorems \ref{Thm:  Riesz Representation Theorem for C(X)} and \ref{Thm:  Order continuous functionals on C(X) are normal measures}, respectively.

\begin{thm}
Let $X$ be a locally compact and $\sigma$-compact Hausdorff space.  Let $\Gamma \defeq (X_n)$ be an increasing sequence (with respect to inclusion) of open precompact sets in $X$ so that $X = \displaystyle\bigcup X_n$.  For $n\in\N$ let $e_n$ and $r_n$ be as in Examples \ref{Exm:  Compactly supported measures as inductive limit} and \ref{Exm:  Continuous functions projective limit}, respectively.

For every $n\in \N$, let $T_n:\vlat{M}(\bar X_n)\to \cont(\bar X_n)^\sim$ denote the usual (isometric) lattice isomorphism,
\[
T_n(\mu)(u) = \int_{X_n} u\thinspace d\mu, ~~ \mu\in\vlat{M}(\bar X_n),~ u\in \cont(\bar X_n).
\]
There exists a unique lattice isomorphism $T:\vlat{M}_c(X)\to \cont(X)^\sim$ so that the diagram
\[
\begin{tikzcd}[cramped]
\vlat{M}(\bar X_n) \arrow[rr, "T_n"] \arrow[dd, "e_n"'] & & \cont(\bar X_n)^\sim \arrow[dd, "r_n^\sim"]\\
 & & \\
\vlat{M}_c(X) \arrow[rr, "T"']  & & \vlat{C}(X)^\sim
\end{tikzcd}
\]
commutes for every $n\in\N$.
\end{thm}

\begin{proof}
Recall that the $r_n$ are surjective.  The result follows immediately from Examples \ref{Exm:  Compactly supported measures as inductive limit} and \ref{Exm:  Continuous functions projective limit}, Theorem \ref{Thm:  Dual of proj sys in VLIC is ind of duals in NVLIC} and Proposition \ref{Prop:  Morphisms between inductive limits}.
\end{proof}

\begin{thm}
Let $X$ be a locally compact and $\sigma$-compact Hausdorff space.  Let $\Gamma \defeq (X_n)$ be an increasing sequence (with respect to inclusion) of open precompact sets in $X$ so that $X = \displaystyle\bigcup X_n$.  For $n\in\N$ let $e_n$ and $r_n$ be as in Examples \ref{Exm:  Compactly supported NORMAL measures as inductive limit} and \ref{Exm:  Continuous functions projective limit}, respectively.

For every $n\in \N$, let $S_n:\vlat{N}(\bar X_n)\to \ordercontn{\cont(\bar X_n)}$ denote the (isometric) lattice isomorphism,
\[
	S_n(\mu)(u) = \int_{X_n} u\thinspace d\mu, ~~\mu\in\vlat{N}(\bar X_n),~ u\in \cont(\bar X_n).
\]
There exists a unique lattice isomorphism $S:\vlat{N}_c(X)\to \ordercontn{\cont(X)}$ so that the diagram
\[
\begin{tikzcd}[cramped]
\vlat{N}(\bar X_n) \arrow[rr, "S_n"] \arrow[dd, "e_n"'] & & \ordercontn{\cont(\bar X_n)} \arrow[dd, "r_n^\sim"]\\
 & & \\
\vlat{N}_c(X) \arrow[rr, "S"']  & & \ordercontn{\cont(X)}
\end{tikzcd}
\]
commutes for every $n\in\N$.
\end{thm}

\begin{proof}
The result follows immediately from Examples \ref{Exm:  Compactly supported NORMAL measures as inductive limit} and \ref{Exm:  Continuous functions projective limit}, Theorem \ref{Thm:  Order continuous dual of proj sys in VLIC is ind of duals in NVLIC} and Proposition \ref{Prop:  Morphisms between inductive limits}.
\end{proof}

\subsection{Perfect spaces}\label{Subsection:  Perfect spaces}

Recall that a vector lattice $\vlat{E}$ is \emph{perfect} if the canonical embedding $\vlat{E}\ni u\longmapsto \Psi_u \in \ordercontnbidual{\vlat{E}}$ is a lattice isomorphism \cite[p. 409]{Zaanen1983RSII}.  We say that a vector lattice $\vlat{E}$ is \emph{an order continuous dual}, or has an \emph{order continuous predual} if there exists a vector lattice $\vlat{F}$ so that $\vlat{E}$ and $\ordercontn{\vlat{F}}$ are isomorphic vector lattices. From the definition it is clear that every perfect vector lattice has an order continuous predual.  On the other hand, see \cite[Theorem 110.3]{Zaanen1983RSII}, $\ordercontn{\vlat{F}}$ is perfect for any vector lattice $\vlat{F}$.  Therefore, if $\vlat{E}$ has an order continuous predual then $\vlat{E}$ is perfect; that is, $\vlat{E}$ is perfect if and only if it has an order continuous predual.

This section is mainly concerned with obtaining a decomposition theorem for perfect vector lattices, i.e. for vector lattices with an order continuous predual, akin to Theorem \ref{Thm:  C(K) Dual space char order dual version}.  This result follows as an application of Example \ref{Exm:  Projective limits of bands} and the duality results in Section \ref{Section:  Dual spaces}.

\begin{lem}\label{Lem: Disjointification of order cts functionals}
Let $\vlat{E}$ be a vector lattice and $0 \leq \varphi, \psi \in \ordercontn{\vlat{E}}$. The following statements are true.
\begin{enumerate}
    \item[(i)] There exist functionals $0\leq \varphi_1,\psi_1 \in\ordercontn{\vlat{E}}$ so that $\varphi_1\wedge \psi_1=0$, $\varphi_1\leq \varphi$, $\psi_1\leq \psi$ and $\varphi\vee \psi = \varphi_1\vee \psi_1$.
    \item[(ii)] If $\vlat{E}$ has the principal projection property and $\varphi$ is strictly positive, then for all $u\in\vlat{E}$, if $\eta(u) = 0$ for all functionals $0 \leq \eta \leq \varphi$ then $u = 0$.
\end{enumerate}
\end{lem}

\begin{proof}
The statement in (i) follows from \cite[Lemma 1.28 (ii) \& Exercise 1.2.E1]{Meyer-Nieberg1991}.

We prove the contrapositive of (ii).  Let $u\neq 0$ in $\vlat{E}$.  Without loss of generality assume that $u^+\neq 0$.  Denote by $\vlat{B}$ the band generated by $u^+$ in $\vlat{E}$.  Define $\eta \defeq \varphi \circ P_{\vlat{B}}$.  Then $\eta$ is order continuous, $0 \leq \eta \leq \varphi$ and $\eta(u) = \varphi(u^+) \neq 0$.
\end{proof}

\begin{thm}\label{Thm:  Projective limit of perfect spaces is perfect}
Let $\cal{I}\defeq \left( (\vlat{E}_\alpha)_{\alpha\in I}, (p_{\beta, \alpha})_{\beta \scc \alpha}\right)$ be an inverse system in $\nvlic$, and let $\cal{S}\defeq \left( \vlat{E}, (p_\alpha)_{\alpha\in I} \right)$ be its inverse limit in $\vlc$.  Assume that $p_{\beta , \alpha}$ is surjective for all $\beta \scc \alpha$ in $I$.  If $\vlat{E}_\alpha$ is perfect for every $\alpha\in I$ then so is $\vlat{E}$.
\end{thm}

\begin{proof}
By Proposition \ref{Prop:  Dual system of projective system is inductive system} the pair $\ordercontn{\cal{I}}\defeq \left( (\ordercontn{(\vlat{E}_\alpha)})_{\alpha\in I}, (p_{\beta, \alpha}^\sim)_{\alpha\precc \beta}\right)$ is a direct system in $\nvlic$.  Because every $p_{\beta,\alpha}$ is surjective, each $p_{\beta,\alpha}^\sim$ is injective.  Hence, by Theorem \ref{Thm:  Existence of Inductive Limits in NVLI}, the direct limit of $\ordercontn{\cal{I}}$ exists in $\nvlic$.  Let $\cal{T}\defeq \left( \vlat{F},(e_\alpha)_{\alpha \in I}\right)$ be the direct limit of $\ordercontn{\cal{I}}$ in $\nvlic$.

By Proposition \ref{Prop:  Dual system of inductive system is projective system} the pair $\ordercontnbidual{\cal{I}}\defeq \left( (\ordercontnbidual{(\vlat{E}_\alpha)})_{\alpha\in I}, (p_{\beta, \alpha}^{\sim \sim})_{\alpha\precc \beta}\right)$ is an inverse system in $\nvlic$, and $\ordercontn{\cal{T}}\defeq \left( \ordercontn{\vlat{F}},(e_\alpha^\sim)_{\alpha \in I}\right)$ is the inverse limit of $\ordercontnbidual{\cal{I}}$ in $\nvlc$ by Theorem \ref{Thm:  Order continuous dual of ind sys in NRIP is proj of duals in NRiesz}.  For every $\alpha\in I$, let $\sigma_\alpha: \vlat{E}_\alpha \to \ordercontnbidual{(\vlat{E}_\alpha)}$ denote the canonical lattice isomorphism.  We observe that the diagram
\[
\begin{tikzcd}[cramped]
\vlat{E}_\beta \arrow[rr, "\sigma_\beta"] \arrow[dd, "p_{\beta, \alpha}"'] & & \ordercontnbidual{(\vlat{E}_\beta)} \arrow[dd, "p^{\sim\sim}_{\beta, \alpha}"]\\
 & & \\
\vlat{E}_\alpha \arrow[rr, "\sigma_\alpha"']  & & \ordercontnbidual{(\vlat{E}_\alpha)}
\end{tikzcd}
\]
commutes for all $\beta\scc\alpha$ in $I$.  By Proposition \ref{Prop:  Morphisms between projective limits}, there exists a unique lattice isomorphism $\Sigma:\vlat{E}\to \ordercontn{\vlat{F}}$ so that the diagram
\[
\begin{tikzcd}[cramped]
\vlat{E} \arrow[rr, "\Sigma"] \arrow[dd, "p_{\alpha}"'] & & \ordercontn{\vlat{F}} \arrow[dd, "e^{\sim}_{\alpha}"]\\
 & & \\
\vlat{E}_\alpha \arrow[rr, "\sigma_\alpha"']  & & \ordercontnbidual{(\vlat{E}_\alpha)}
\end{tikzcd}
\]
commutes for every $\alpha \in I$.  Since $\ordercontn{\vlat{F}}$ is perfect, we conclude that $\vlat{E}$ is also perfect.
\end{proof}

We now come to the main results of this section, namely, decomposition theorems for perfect vector lattices.  Recall the terminology and notation introduced in Example \ref{Exm:  Projective limits of bands}.

\begin{thm}\label{Thm:  Perfect spaces as projective limits of carriers}
Let $\vlat{E}$ be a Dedekind complete vector lattice.  Let $\vlat{M}_{\mathrm n}\subseteq \bands{\vlat{E}}$ consist of the carriers of all positive, order continuous functionals on $\vlat{E}$; that is,
\[
\vlat{M}_{\mathrm n} \defeq \{\carrier{\varphi} ~:~ 0\leq \varphi\in\ordercontn{\vlat{E}}\}.
\]
For $\carrier{\varphi} \subseteq \carrier{\psi}$ in $\vlat{M}_{\mathrm n}$, denote by $P_{\varphi}$ the band projection of $\vlat{E}$ onto $\carrier{\varphi}$ and by $P_{\psi , \varphi}$ the band projection of $\carrier{\psi}$ onto $\carrier{\varphi}$.  The following statements are true.
\begin{enumerate}
    \item[(i)] $\vlat{M}_{\mathrm n}$ is an ideal in $\bands{\vlat{E}}$.
    \item[(ii)] $\vlat{M}_{\mathrm n}$ is a non-trivial ideal in $\bands{\vlat{E}}$ if and only if $\vlat{E}$ admits a non-zero order continuous functional.
    \item[(iii)] $\vlat{M}_{\mathrm n}$ is a proper ideal in $\bands{\vlat{E}}$ if and only if $\vlat{E}$ does not admit a strictly positive order continuous functional.
    \item[(iv)] $P_{\vlat{M}_{\mathrm n}}$ is injective if and only if $\preann{\ordercontn{\vlat{E}}}=\{0\}$.
    \item[(v)] If $\vlat{E}$ is perfect then $P_{\vlat{M}_{\mathrm n}}$ is a lattice isomorphism.
\end{enumerate}
\end{thm}

\begin{proof}[Proof of (i)]
For $0\leq \psi,\varphi\in\ordercontn{\vlat{E}}$, we have $\carrier{\psi},\carrier{\varphi}\subseteq \carrier{\varphi\vee \psi}\in \vlat{M}_{\mathrm n}$ and therefore $\vlat{M}_{\mathrm n}$ is upwards directed.

Let $\vlat{B}\in\bands{\vlat{E}}$ and $0\leq \varphi\in\ordercontn{\vlat{E}}$ such that $\vlat{B}\subseteq \carrier{\varphi}$.  Define $\psi \defeq \varphi \circ P_{\vlat{B}}$.  Then $\psi \geq 0$ and by the order continuity of band projections, $\psi\in \ordercontn{\vlat{E}}$.  We show that $\nullid{\psi} = \vlat{B}^d$.  For $u\in \vlat{B}^d$, $P_\vlat{B}(|u|) = 0$ so that $\psi(|u|) = \varphi\left( P_{\vlat{B}}(|u|) \right) = 0$.  Therefore $\vlat{B}^d\subseteq \nullid{\psi}$.  For the reverse inclusion, let $v\in \nullid{\psi}$.  Then $\varphi\left( P_{\vlat{B}}(|v|) \right) = 0$ so that $P_{\vlat{B}}(|v|) \in \nullid{\varphi} \subseteq \vlat{B}^d$.  Hence $P_{\vlat{B}}(|v|)=0$ so that $v\in \vlat{B}^d$.  We conclude that $\vlat{B}=\carrier{\psi}$.  Therefore $\vlat{B}\in\vlat{M}_{\mathrm n}$ so that $\vlat{M}_{\mathrm n}$ is downward closed, hence an ideal in $\bands{\vlat{E}}$.
\end{proof}

\begin{proof}[Proof of (ii)]
This is clear.
\end{proof}

\begin{proof}[Proof of (iii)]
A functional $0\leq \varphi\in\ordercontn{\vlat{E}}$ is strictly positive if and only if $\nullid{\varphi}=\{0\}$, if and only if $\carrier{\varphi} = \vlat{E}$; hence the result follows.
\end{proof}

\begin{proof}[Proof of (iv)]
According to Example \ref{Exm:  Projective limits of bands} (iii),  $P_{\vlat{M}_{\mathrm n}}$ is injective if and only if $\{P_\varphi ~:~ 0\leq \varphi\in \ordercontn{\vlat{E}}\}$ separates the points of $\vlat{E}$.  It therefore suffices to prove that $\preann{\ordercontn{\vlat{E}}}=\{0\}$ if and only if $\{P_{\varphi} : 0\leq \varphi\in\ordercontn{\vlat{E}}\}$ separates the points of $\vlat{E}$.

Assume that $\preann{\ordercontn{\vlat{E}}}=\{0\}$.  Fix $u\in \vlat{E}$ with $u\neq 0$.  Then there exists $\varphi\in \ordercontn{\vlat{E}}$ such that $\varphi(u) \neq 0$.  Therefore $0 < |\varphi(u)| \leq |\varphi|(|u|)$.  Hence $u\not\in \nullid{|\varphi|}$ and thus $P_{|\varphi|}(u) \neq 0$.

Conversely, assume that $\{P_{\varphi} : 0\leq \varphi\in\ordercontn{\vlat{E}}\}$ separates the points of $\vlat{E}$.   Let $0<v\in \vlat{E}^+$.  There exists $0 \leq \varphi \in \ordercontn{\vlat{E}}$ such that $P_{\varphi}(v) > 0$.  Since every positive functional is strictly positive on its carrier, it follows that $\varphi(v) \geq \varphi\left( P_{\varphi}(v) \right) > 0$.  Now consider any non-zero $w\in \vlat{E}$.  There exists $0 \leq \varphi \in \ordercontn{\vlat{E}}$ such that $\varphi(w^+) \neq 0$.  Let $B$ denote the band generated by $w^+$ in $\vlat{E}$ and define the functional $\psi \defeq \varphi\circ P_{B}$.  Then $0\leq \psi\in \ordercontn{\vlat{E}}$ and  $\psi(w) = \varphi(w^+) \neq 0$.
\end{proof}

\begin{proof}[Proof of (v)]
It follows from Example \ref{Exm:  Projective limits of bands} (ii) that $P_{\vlat{M}_{\mathrm n}}$ is a lattice homomorphism.  Since $\vlat{E}$ is perfect, $\preann{\ordercontn{\vlat{E}}}=\{0\}$ by \cite[Theorem 110.1]{Zaanen1983RSII} and so by (iv), $P_{\vlat{M}_{\mathrm n}}$ is injective.  We show that $P_{\vlat{M}_{\mathrm n}}$ is surjective.

Let $0\leq u = \left( u_\varphi \right) \in \proj{\cal{I}_{\vlat{M}_{\mathrm n}}}$.  Define the map $\Upsilon: (\ordercontn{\vlat{E}})^+ \to \R$ by setting $\Upsilon\left(\varphi\right) \defeq \varphi\left( u_{\varphi} \right)$ for every $\varphi \in (\ordercontn{\vlat{E}})^+$.  We claim that $\Upsilon$ is additive.  Let $0 \leq \varphi, \psi \in \ordercontn{\vlat{E}}$.  Then
\[
\begin{array}{lll}
\Upsilon\left( \varphi + \psi \right) & = & \left(\varphi + \psi\right)\left( u_{\varphi + \psi} \right) \smallskip\\
& = & \varphi\left( u_{\varphi + \psi} \right) + \psi\left( u_{\varphi + \psi} \right)\smallskip\\
& = &\varphi\circ P_{\varphi}\left(u_{\varphi + \psi} \right) + \psi\circ P_{\psi}\left( u_{\varphi + \psi} \right).
\end{array}
\]
Because $(u_{\varphi})\in \proj{\cal{I}_{\vlat{M}_{\mathrm n}}}$, $u_{\varphi+\psi}\in \carrier{\varphi+\psi}$ so that $P_{\varphi}\left(u_{\varphi + \psi} \right)= P_{\varphi+\psi , \varphi}\left(u_{\varphi + \psi} \right) = u_\varphi$ and $P_{\psi}\left(u_{\varphi + \psi} \right) = P_{\varphi+\psi , \psi}\left(u_{\varphi + \psi} \right) = u_\psi$.  Hence
\[
\Upsilon\left( \varphi + \psi \right) = \varphi\left( u_\varphi \right) + \psi\left( u_\psi \right) = \Upsilon\left( \varphi\right) + \Upsilon\left( \psi \right).
\]
By \cite[Theorem 1.10]{AliprantisBorder2006} $\Upsilon$ extends to a positive linear functional on $\ordercontn{\vlat{E}}$, which we denote by $\Upsilon$ as well.

We claim that $\Upsilon$ is order continuous.  To see this, consider any $D\downarrow 0$ in $\ordercontn{\vlat{E}}$.  Fix $\epsilon > 0$ and $\varphi\in D$.  By \cite[Theorem 1.18]{AliprantisBurkinshaw2006} there exists $\psi_0\leq \varphi$ in $D$ so that $0\leq \psi(u_{\varphi})<\epsilon$ for all $\psi\leq \psi_0$ in $D$.  Consider $\psi\leq \psi_0$.  Since $u\in \proj{\cal{I}_{\vlat{M}_{\mathrm n}}}$ we have $u_\psi = P_{\varphi,\psi} (u_\varphi)\leq u_\varphi$ so that $0\leq \psi(u_\psi) \leq \psi(u_{\varphi})<\epsilon$; that is, $0\leq \Upsilon(\psi)<\epsilon$ for all $\psi\leq \psi_0$.  Therefore $\Upsilon[D]\downarrow 0$ in $\R$ so that $\Upsilon$ is order continuous, as claimed.

Since $\vlat{E}$ is perfect, there exists $v\in \vlat{E}^+$ so that $\Upsilon\left( \varphi \right) = \varphi\left( v \right)$ for all $\varphi\in \ordercontn{\vlat{E}}$.  We claim that $P_{\vlat{M}_{\mathrm n}}(v)=u$; that is, $P_{\varphi}(v)=u_\varphi$ for every $0\leq \varphi\in \ordercontn{\vlat{E}}$.  For each $0 \leq \varphi\in \ordercontn{\vlat{E}}$ we have $\varphi(u_\varphi) =  \Upsilon\left(\varphi\right) = \varphi(v) = \varphi\left( P_{\varphi}(v) \right)$.  Let $0 \leq \eta \leq \varphi$ in $\ordercontn{\vlat{E}}$.  Then
\[
	\eta\left( u_\varphi \right) = \eta\left( P_{\eta}(u_\varphi) \right) = \eta(P_{\varphi,\eta}(u_\varphi)) = \eta\left( u_\eta \right) = \Upsilon(\eta) = \eta(v),
\]
and,
\[
\eta\left( P_{\varphi} (v) \right) = \eta\left( P_{\eta}P_{\varphi} (v) \right) = \eta\left( P_{\eta}(v) \right) = \eta(v).
\]
Thus $\eta\left( u_\varphi- P_{\varphi } (v) \right)=0$.  By Lemma \ref{Lem: Disjointification of order cts functionals} (ii), applied on $C_\varphi$, we conclude that $P_{\varphi}(v) = u_\varphi$.  This verifies our claim.  Therefore $P_{\vlat{M}_{\mathrm n}}$ maps $\vlat{E}^+$ onto $\left(\proj{\cal{I}_{\vlat{M}_{\mathrm n}}}\right)^+$ which shows that $P_{\vlat{M}_{\mathrm n}}$ is surjective.
\end{proof}

\begin{remark}\label{Remark:  Inverse limit of carriers not always perfect}
We observe that the converse of Theorem \ref{Thm:  Perfect spaces as projective limits of carriers} (v) is false.  Indeed, $\ordercontnbidual{({\mathrm c}_0)}=\ell^\infty$ so that ${\mathrm c}_0$ is not perfect.  However, there exists a strictly positive functional $\varphi\in \ordercontn{({\mathrm c}_0)}$.  Therefore ${\mathrm c}_0 = \carrier{\varphi}\in \vlat{M}_{\mathrm n}$ so that $P_{\vlat{M}_{\mathrm n}}$ maps ${\mathrm c}_0$ lattice isomorphically onto $\proj{\cal{I}_{\vlat{M}_{\mathrm n}}}$, see Remark \ref{Remark:  Vector lattice as projective limit of bands}.
\end{remark}

\begin{cor}\label{Cor:  Perfect spaces as inverse limit of perfect carriers}
Let $\vlat{E}$ be a Dedekind complete vector lattice.  Let $\vlat{M}_{p}\subseteq \bands{\vlat{E}}$ consist of the carriers of all positive, order continuous functionals on $\vlat{E}$ which are perfect; that is,
\[
\vlat{M}_{p} \defeq \{\carrier{\varphi} ~:~ 0\leq \varphi\in\ordercontn{\vlat{E}} \text{ and } \carrier{\varphi} \text{ is perfect}\}.
\]
The following statements are true. \begin{enumerate}
    \item[(i)] $\vlat{M}_{p}$ is an ideal in $\bands{\vlat{E}}$.
    \item[(ii)] $P_{\vlat{M}_p}$ is a lattice isomorphism if and only if $\vlat{E}$ is perfect.
\end{enumerate}
\end{cor}

\begin{proof}[Proof of (i)]
It follows from Theorem \ref{Thm:  Perfect spaces as projective limits of carriers} (i) and the fact that bands in a perfect vector lattice are themselves perfect that $\vlat{M}_{p}$ is downwards closed in $\bands{\vlat{E}}$.  To see that $\vlat{M}_{p}$ is upwards directed, fix $C_\varphi,C_\psi\in\vlat{M}_p$.  By Lemma \ref{Lem: Disjointification of order cts functionals} (i) there exist functionals $0 \leq \varphi_1\leq \varphi$ and $0\leq \psi_1 \leq \psi$ in $\ordercontn{\vlat{E}}$ such that $\varphi_1 \wedge \psi_1=0$ and $\varphi_1 \vee \psi_1 = \varphi \vee \psi$.  Because $0\leq \varphi_1\leq \varphi$ and $0\leq \psi_1\leq \psi$ it follows that $\carrier{\varphi_1}\subseteq \carrier{\varphi}$ and $\carrier{\psi_1}\subseteq \carrier{\psi}$.  Therefore $\carrier{\varphi_1}$ and $\carrier{\psi_1}$ are perfect.  By \cite[Theorem 90.7]{Zaanen1983RSII} we have
\[
\carrier{\varphi_1 \vee \psi_1} = \left( \carrier{\varphi_1} + \carrier{\psi_1} \right)^{dd} = \carrier{\varphi_1} + \carrier{\psi_1}.
\]
By \cite[Theorem 90.6]{Zaanen1983RSII}, since $\varphi_1 \wedge \psi_1 = 0$, we have $\carrier{\varphi_1} \perp \carrier{\psi_1}$.  Thus $\carrier{\varphi_1} \cap \carrier{\psi_1} =\{ 0 \}$ which implies $\carrier{\varphi_1 \vee \psi_1} = \carrier{\varphi_1} \oplus \carrier{\psi_1}$.  Hence it follows from Theorem \ref{Thm:  Properties of product of vector lattices.} (v) and (vii) that  $\ordercontnbidual{\left(\carrier{\varphi_1 \vee \psi_1}\right)} \cong \carrier{\varphi_1 \vee \psi_1}$; that is, $\carrier{\varphi \vee \psi}=\carrier{\varphi_1 \vee \psi_1}$ is perfect.  Since $\carrier{\varphi},\carrier{\psi}\subseteq \carrier{\varphi \vee \psi}$ it follows that $\vlat{M}_p$ is upward directed, hence an ideal in $\bands{\vlat{E}}$.
\end{proof}

\begin{proof}[Proof of (ii)]
If $\vlat{E}$ is perfect then $\vlat{M}_p=\vlat{M}_{\mathrm n}$, and so the result follows from Theorem \ref{Thm:  Perfect spaces as projective limits of carriers} (v).  Conversely, if $P_{\vlat{M}_p}$ is an isomorphism then Theorem \ref{Thm:  Projective limit of perfect spaces is perfect} implies that $\vlat{E}$ is perfect.
\end{proof}

We now consider direct limits of perfect spaces.  Due to the inherent limitations of the duality theorems for inverse limits, the results we obtain are less general than the corresponding results for inverse limits.

\begin{thm}\label{Thm: Direct limit of perfect spaces is perfect.}
Let $\cal{D} \defeq \left( (\vlat{E}_n)_{n \in \N}, (e_{n, m})_{n \leq m} \right)$ be a direct system in $\nvlic$, and let $\cal{S}\defeq \left( \vlat{E}, (e_n)_{n \in \N} \right)$ be the direct limit of $\cal{D}$ in $\vlic$.  Assume that $e^\sim_{n , m}$ is surjective for all $n \leq m$ in $\N$.  If $\vlat{E}_n$ is perfect for every $n \in \N$ then so is $\vlat{E}$.
\end{thm}

\begin{proof}
By Proposition \ref{Prop:  Dual system of inductive system is projective system}, the pair $\ordercontn{\cal{D}} \defeq \left( ( \ordercontn{( \vlat{E}_n )} )_{n \in \N}, (e^\sim_{n, m})_{n \leq m} \right)$ is an inverse system in $\nvlic$, and by Theorem \ref{Thm:  Existence of Projective Limit NVL}~(ii) the inverse limit of $\ordercontn{\cal{D}}$ exists in $\nvlc$.  Denote $\proj{\ordercontn{\cal{D}}}$ by $\cal{S}_0 \defeq \left( \vlat{F},(p_n)_{n \in \N}\right)$.

By Proposition \ref{Prop:  Dual system of projective system is inductive system}, the pair $\ordercontnbidual{\cal{D}} \defeq \left( \left( \ordercontnbidual{( \vlat{E}_n )} \right)_{n \in \N}, (e^{\sim\sim}_{n, m})_{n \leq m} \right)$ is a direct system in $\nvlic$.  Since we assumed that the $e^\sim_{n , m}$ are surjective, it follows by Theorem \ref{Thm:  Order continuous dual of proj sys in VLIC is ind of duals in NVLIC} that $\ordercontn{\left( \cal{S}_0 \right)}$ is the direct limit of $\ordercontnbidual{\cal{D}}$ in $\nvlic$. For every $n \in \N$, let $\sigma_n: \vlat{E}_n \to \ordercontnbidual{(\vlat{E}_n)}$ denote the canonical lattice isomorphism.  The diagram
\[
\begin{tikzcd}[cramped]
\vlat{E}_n \arrow[rr, "\sigma_n"] \arrow[dd, "e_{n,m}"'] & & \ordercontnbidual{( \vlat{E}_n )} \arrow[dd, "e^{\sim\sim}_{n,m} "]\\
 & & \\
\vlat{E}_m \arrow[rr, "\sigma_m"']  & & \ordercontnbidual{( \vlat{E}_m )}
\end{tikzcd}
\]
commutes for all $n\leq m$ in $\N$.  By Proposition \ref{Prop:  Morphisms between inductive limits} there exists a unique lattice isomorphism $\Sigma: \vlat{E}\to \ordercontn{\vlat{F}}$ so that the diagram
\[
	\begin{tikzcd}[cramped]
    \vlat{E}_n \arrow[rr, "\sigma_n"] \arrow[dd, "e_n"'] & & \ordercontnbidual{( \vlat{E}_n )} \arrow[dd, "e^{\sim\sim}_{n,m} "]\\
     & & \\
    \vlat{E} \arrow[rr, "\Sigma"']  & & \ordercontn{\vlat{F}}
    \end{tikzcd}
\]
commutes for every $n \in \N$.  Since $\ordercontn{\vlat{F}}$ is perfect, we conclude that $\vlat{E}$ is also perfect.
\end{proof}

\begin{cor}\label{Cor:  Inductive limit of sequence of perfects is perfect}
Let $\cal{D}\defeq \left((\vlat{E}_{n})_{n\in\N},(e_{n,m})_{n\leq m}\right)$ be a direct system in $\nvlic$, and let $\cal{S} \defeq \left( \vlat{E}, (e_n)_{n\in \N} \right)$ be the direct limit of $\cal{D}$ in $\vlic$.  Assume that $e_{n,m}$ is injective and $e_{n,m}[\vlat{E}_n]$ is a band in $\vlat{E}_m$ for all $n\leq m$ in $\N$.  If $\vlat{E}_n$ is perfect for every $n\in\N$ then so is $\vlat{E}$.
\end{cor}

\begin{proof}

We show that $e_{n,m}^\sim$ is surjective for all $n\leq m$ in $\N$. Then the result follows directly from Theorem~\ref{Thm: Direct limit of perfect spaces is perfect.}. We observe that each $\vlat{E}_n$ is Dedekind complete and thus has the projection property.  Fix $n\leq m$ in $\N$.  Let $P_{m,n}:\vlat{E}_m\to e_{n,m}[\vlat{E}_n]$ be the band projection onto $e_{n,m}[\vlat{E}_n]$.  The diagram
\[
\begin{tikzcd}[cramped]
\vlat{E}_n \arrow[rd, "e_{n,m}"'] \arrow[rr, "e_{n,m}"] & & \vlat{E}_m \arrow[dl, "P_{m,n}"]\\
& e_{n,m}[\vlat{E}_n]
\end{tikzcd}
\]
commutes.  Therefore
\[
\begin{tikzcd}[cramped]
\ordercontn{(\vlat{E}_{m})} \arrow[rr, "e_{n,m}^\sim"] & & \ordercontn{(\vlat{E}_n)}\\
& \ordercontn{(e_{n,m}[\vlat{E}_n])} \arrow[lu, "P_{m,n}^\sim"] \arrow[ru, "e_{n,m}^\sim"']
\end{tikzcd}
\]
commutes as well.  Since $e_{n,m}:\vlat{E}_n\to e_{n,m}[\vlat{E}_n]$ is an isomorphism, so is $e_{n,m}^\sim: \ordercontn{(e_{n,m}[\vlat{E}_n])}\to \ordercontn{\left( \vlat{E}_n \right)}$. It follows from the above diagram that $e_{n,m}^\sim:\ordercontn{(\vlat{E}_m)}\to \ordercontn{(\vlat{E}_n)}$ is a surjection.
\end{proof}


\begin{cor}
Let $\vlat{E}$ be a vector lattice.  Assume that there exists an increasing sequence $(\varphi_n)$ of positive order continuous functionals on $\vlat{E}$ such that $\displaystyle\bigcup \carrier{\varphi_n}=\vlat{E}$ and, for every $n\in\N$, $\carrier{\varphi_n}$ is perfect.  Then $\vlat{E}$ is perfect.
\end{cor}

\begin{proof}
For all $n\leq m$ denote by $e_{n,m}:\carrier{\varphi_n}\to\carrier{\varphi_m}$ and $e_n:\carrier{\varphi_n}\to\vlat{E}$ the inclusion maps.  By Example \ref{Exm:  Inductive limit main example}, $\cal{D}\defeq \left( (\carrier{\varphi_n})_{n\in \N},(e_{n,m})_{n\leq m}\right)$ is a direct system in $\nvlic$, and $\cal{S} \defeq \left(\vlat{E},(e_n)_{n\in\N}\right)$ is the direct limit of $\cal{D}$ in $\nvlic$.  By Corollary \ref{Cor:  Inductive limit of sequence of perfects is perfect}, $\vlat{E}$ is perfect.
\end{proof}

\subsection{Decomposition theorems for $\contX$ as a dual space}\label{Subsection:  Structure theorems for C(X) as a dual space}

This section deals with decomposition theorems for spaces $\contX$ of continuous, real valued functions which are order dual spaces.  In particular, we show that the naive generalization of Theorems \ref{Thm:  C(K) Dual space char} and \ref{Thm:  C(K) Dual space char order dual version} to the non-compact case fails, and present an alternative approach via inverse limits. Specialising Corollary \ref{Cor:  Perfect spaces as inverse limit of perfect carriers} to $\contX$ yields the desired decomposition theorem.  In order to facilitate the discussion to follow we recall some basic facts concerning the structure of the carriers of positive functionals on $\contX$.  Throughout this section, $X$ denotes a realcompact space. Recall from Section~\ref{Section:  Introduction} that the realcompactification of a Tychonoff space $Y$ is denoted as $\upsilon Y$.

Let $0\leq \varphi\in \orderdual{\contX}$.  According to Theorem \ref{Thm:  Riesz Representation Theorem for C(X)} there exists a measure $\mu_\varphi \in\vlat{M}_c(X)^+$ so that
\[
\varphi(u) = \int u \thinspace d\mu_\varphi,~~ u\in\contX.
\]
Denote by $S_\varphi$ the support of the measure $\mu_\varphi$.  The null ideal of $\varphi$ is given by
\[
\nullid{\varphi} = \{u\in\contX ~:~ u(x) = 0 \text{ for all } x\in S_{\varphi}\}.
\]
Indeed, the inclusion $\{u\in\contX ~:~ u(x) = 0 \text{ for all } x\in S_{\varphi}\}\subseteq \nullid{\varphi}$ is clear.  For the reverse inclusion, consider $u\in\contX$ so that $u(x_0)\neq 0$ for some $x_0\in S_\varphi$.  Then there exist a neighbourhood $V$ of $x_0$ and a number $\epsilon>0$ so that $|u|(x)>\epsilon$ for all $x\in V$.  Because $x_0\in S_\varphi$, $\mu_\varphi(V)>0$.  Therefore
\[
\varphi(|u|) \geq \int_V |u| \thinspace d\mu_\varphi \geq \epsilon\mu_\varphi(V)>0
\]
so that $u\notin \nullid{\varphi}$.  It therefore follows that
\[
\carrier{\varphi}=\{u\in\contX ~:~ u(x) = 0 \text{ for all } x\in X\setminus S_\varphi\}.
\]
The band $\carrier{\varphi}$ is a projection band if and only if $S_\varphi$ is open, hence compact and open, see \cite[Theorem 6.3]{KandicVavpeticPositivity2019}.  In this case we identify $\carrier{\varphi}$ with $\cont(S_\varphi)$ and the band projection $P_\varphi:\contX\to \carrier{\varphi}$ is given by restriction of $u\in\contX$ to $S_\varphi$.

\begin{prop}\label{Prop:  Carriers in C(X) are perfect}
Let $X$ be extremally disconnected.  Then $\carrier{\varphi}$ is perfect for every $0\neq \varphi\in\ordercontn{\contX}$.
\end{prop}

\begin{proof}
Let $0\neq \varphi\in\ordercontn{\contX}$. Since $\contX$ is Dedekind complete, so is $\carrier{\varphi}$. Furthermore, $|\varphi|$ is strictly positive and order continuous on $\carrier{\varphi}$. Thus $\carrier{\varphi}$ has a separating order continuous dual. By Theorem~\ref{Thm:  C(K) Dual space char order dual version}, $\carrier{\varphi} = \cont(S_\varphi)$ has a Banach lattice predual; that is, $\carrier{\varphi}$ is an order dual space. Therefore $\carrier{\varphi}$ is perfect by \cite[Theorem~110.2]{Zaanen1983RSII}.
\end{proof}

\begin{thm}\label{Thm:  C(X) perfect iff order dual iff x is Xion}
Let $X$ be a realcompact space.  Denote by $S$ the union of the supports of all order continuous functionals\footnote{Equivalently, all compactly supported normal Radon measures on $X$.} on $\contX$.  The following statements are equivalent. \begin{enumerate}
    \item[(i)] There exists a vector lattice $\vlat{E}$ so that $\cont(X)$ is lattice isomorphic to $\orderdual{\vlat{E}}$.
    \item[(ii)] $\contX$ is perfect.
    \item[(iii)] $X$ is extremally disconnected and $\upsilon S = X$; that is,
    \[
    	\contX \ni u \longmapsto \left.u\right|_{S} \in \cont (S)
    \]is a lattice isomorphism.
\end{enumerate}
\end{thm}

\begin{proof}
That (i) implies (ii) follows from \cite[Theorem~110.2]{Zaanen1983RSII}. The argument in the proof of \cite[Theorem 2]{Xiong1983} shows that (ii) implies (iii), and \cite[Theorem 1]{Xiong1983} shows that (iii) implies (i). Thus the statements (i), (ii) and (iii) are equivalent.
\end{proof}

A naive attempt to generalise Theorem \ref{Thm:  C(K) Dual space char order dual version} (iv) is to replace the $\ell^\infty$-direct sum in that result with the Cartesian product of the carriers of a maximal singular family in $\ordercontn{\cont(X)}$.  In next result and the example to follow, we show that this approach is not correct.

\begin{prop}\label{Prop:  Partial decomposition result for order dual C(X)}
Let $X$ be an extremally disconnected realcompact space, and let $\cal{F}$ be a maximal (with respect to inclusion) singular family of positive order continuous linear functionals on $\contX$.  Consider the following statements. \begin{enumerate}
    \item [(i)] The map
    \[
    \contX\ni u \longmapsto (P_{\varphi}(u))\in \dsprod_{\varphi\in\cal{F}}\carrier{\varphi}
    \]
    is a lattice isomorphism.
    \item [(ii)] $\contX$ is perfect.
    \item[(iii)] There exists a vector lattice $\vlat{E}$ so that $\cont(X)$ is lattice isomorphic to $\orderdual{\vlat{E}}$.
\end{enumerate}
Then (i) implies (ii), and (ii) and (iii) are equivalent.
\end{prop}

\begin{proof}
By Theorem \ref{Thm:  C(X) perfect iff order dual iff x is Xion}, (ii) and (iii) are equivalent.  Assume that (i) is true.  By Theorem \ref{Thm:  Properties of product of vector lattices.} (v) and (vii), $\ordercontnbidual{\contX}$ is isomorphic to $\dsprod \ordercontnbidual{(\carrier{\varphi})}$.  But each $\carrier{\varphi}$ is perfect so that $\dsprod \ordercontnbidual{(\carrier{\varphi})}$ is isomorphic to $\dsprod \carrier{\varphi}$, hence $\contX$ is isomorphic to $\ordercontnbidual{\contX}$.
\end{proof}

\begin{example}\label{Exm:  C(X) decomponsition counterexample}
As is well known, $\cont(\beta\N)=\ell^\infty$ is perfect, hence an order dual space.  For every $x\in\N$, denote by $\delta_x:\cont(\beta\N)\to\R$ the point mass centred at $x$.  Then $\cal{F}=\{\delta_x ~:~ x\in\N\}$ is a maximal singular family in $\ordercontn{\cont(\beta\N)} \cong \ell^1$.  Since $\carrier{\delta_x}=\R$ for every $x\in\N$, it follows that $\dsprod\carrier{\delta_x}=\R^\omega$.  Therefore $\dsprod\carrier{\delta_x}$ does not have a strong order unit.  Since $\cont(\beta\N)$ contains a strong order unit,
\[
\cont(\beta\N)\ni u \longmapsto (P_{\delta_x}(u))\in\dsprod \carrier{\delta_x}
\]
is not an isomorphism.
\end{example}

The final result of this section offers a solution to the decomposition problem for a space $\contX$ which is an order dual space.  We refer the reader to the notation used in Example \ref{Exm:  Projective limits of bands} and Theorem \ref{Thm:  Projective limit of perfect spaces is perfect}.

\begin{thm}\label{Thm:  Structure theorem for C(X) a dual space}
Let $X$ be an extremally disconnected realcompact space. Denote by $S$ the union of the supports of all order continuous functionals on $\contX$. The following statements are equivalent.
\begin{enumerate}
    \item[(i)] There exists a vector lattice $\vlat{E}$ so that $\cont(X)$ is lattice isomorphic to $\orderdual{\vlat{E}}$.
    \item[(ii)] $\contX$ is perfect.
    \item[(iii)] $\upsilon S = X$.
    \item[(iv)] $P_{\vlat{M}_{\rm n}}:\contX\to \proj{\cal{I}_{\vlat{M}_{\rm n}}}$ is a lattice isomorphism.
\end{enumerate}
\end{thm}

\begin{proof}
By Theorem~\ref{Thm:  C(X) perfect iff order dual iff x is Xion}, it suffices to show that (ii) and (iv) are equivalent.  Since $\carrier{\varphi}$ is perfect for every $0\leq \varphi \in \ordercontn{\contX}$ by Proposition \ref{Prop:  Carriers in C(X) are perfect}, this follows immediately from Corollary \ref{Cor:  Perfect spaces as inverse limit of perfect carriers}.
\end{proof}

\subsection{Structure theorems}\label{Subsection:  Structure theorems}

Let $\vlat{E}$ be an Archimedean vector lattice.  In Example \ref{Exm:  Inductive limit of principle ideals} it is shown that the principal ideals of $\vlat{E}$ form a direct system in $\nvlic$, and that $\vlat{E}$ can be expressed as the direct limit of this system.  In this section we exploit this result and the duality results in Section \ref{Section:  Dual spaces} to obtain structure theorems for vector lattices and their order duals.

A frequently used technique in the theory of vector lattices is to reduce a problem to one confined to a fixed principal ideal $\vlat{E}_u$ of a space $\vlat{E}$.  Once this is achieved, the problem becomes equivalent to one in a space $\cont(K)$ of continuous functions on some compact Hausdorff space $K$ via the Kakutani Representation Theorem, see \cite{Kakutani1941} or \cite[Theorem 2.1.3]{Meyer-Nieberg1991}. For instance, this technique is used in \cite[Theorem 3.8.6]{Meyer-Nieberg1991} to study tensor products of Banach lattices. The following result is essentially a formalization of this method in the language of direct limits.

\begin{thm}\label{Thm:  Structure theorem}
Let $\vlat{E}$ be an Archimedean, relatively uniformly complete vector lattice.  For all $0<u\leq v$ there exist compact Hausdorff spaces $K_u$ and $K_v$ and injective, interval preserving normal lattice homomorphisms $e_{u,v}:\cont(K_u)\to \cont(K_v)$ and $e_u:\cont(K_u)\to \vlat{E}$ so that the following is true. \begin{enumerate}
    \item[(i)] $\vlat{E}_u$ is lattice isomorphic to $\cont(K_u)$ for every $0<u\in\vlat{E}$.
    \item[(ii)] $\cal{D}_\vlat{E}\defeq \left((\cont(K_u))_{0<u\in\vlat{E}},(e_{u,v})_{u\leq v}\right)$ is a direct system in $\nvlic$
    \item [(iii)] $\cal{S}_\vlat{E}\defeq \left(\vlat{E},(e_u)_{0<u\in\vlat{E}}\right)$ is the direct limit of $\cal{D}_{\vlat{E}}$ in $\nvlic$.
    \item[(iv)] $\vlat{E}$ is Dedekind complete if and only if $K_u$ is Stonean for every $0<u\in\vlat{E}$.
    \item[(v)] If $\vlat{E}$ is perfect then $K_u$ is hyper-Stonean for every $0<u\in\vlat{E}$.
\end{enumerate}
\end{thm}

\begin{proof}
According to \cite[Proposition 1.2.13]{Meyer-Nieberg1991} every principal ideal in $\vlat{E}$ is a unital $AM$-space. Therefore the statements in (i), (ii) and (iii) follow immediately from Example \ref{Exm:  Inductive limit of principle ideals} and Kakutani's Representation Theorem for $AM$-spaces \cite{Kakutani1941}. The proof of (iv) follows immediately from Theorem \ref{Thm:  Inductive Limit Permanence} and \cite[Proposition 2.1.4]{Meyer-Nieberg1991}.

For the proof of (v), assume that $\vlat{E}$ is perfect.  Then, in particular, $\vlat{E}$ is Dedekind complete and has a separating order continuous dual.  Therefore the same is true for each $\vlat{E}_u$.  By (i), $\cont(K_u)$ is Dedekind complete and has a separating order continuous dual.  It follows from Theorems \ref{Thm:  C(K) Dual space char} and \ref{Thm:  C(K) Dual space char order dual version} that $K_u$ is hyper-Stonean.
\end{proof}

\begin{cor}\label{Cor:  Structure of duals ito spaces of measures}
Let $\vlat{E}$ be an Archimedean, relatively uniformly complete vector lattice.  There exist an inverse system $\cal{I}\defeq \left((\vlat{M}(K_\alpha))_{\alpha\in I},(p_{\beta,\alpha})_{\beta \scc \alpha}\right)$ in $\nvlic$, with each $K_\alpha$ a compact Hausdorff space, and normal lattice homomorphisms $p_\alpha :\orderdual{\vlat{E}}\to \vlat{M}(K_\alpha)$, so that $\cal{S}\defeq \left(\orderdual{\vlat{E}},(p_\alpha)_{\alpha \in I}\right)$ is the inverse limit of $\cal{I}$ in $\nvlc$.
\end{cor}

\begin{proof}
The result follows immediately from Theorems \ref{Thm:  Structure theorem} and \ref{Thm:  Dual of ind sys in VLIC is proj of duals in NVL}, and the Riesz Representation Theorem.
\end{proof}

In order to obtain a converse of Corollary \ref{Cor:  Structure of duals ito spaces of measures} we require a more detailed description of the interval preserving normal lattice homomorphisms $e_{u,v}:\cont(K_u)\to \cont(K_v)$ in Theorem \ref{Thm:  Structure theorem}.  Let $X$ and $Y$ be topological spaces and $p: X\to Y$ a continuous function.  Recall from \cite[p.~20]{Bilokopytov2021} that $p$ is \emph{almost open} if for every non-empty open subset $U$ of $X$, $\text{int}\left( \overline{p\left[ U \right]} \right) \neq \emptyset$.  It is clear that all open maps are almost open and thus every homeomorphism is almost open.

\begin{prop}\label{Prop:  Chacterization of injective interval preserving lattice homomorphisms on C(K)}
Let $K$ and $L$ be compact Hausdorff spaces and $T:\cont(K)\to\cont(L)$ a positive linear map.  $T$ is a lattice homomorphism if and only if there exist a unique $0<w\in\cont(L)$ and a unique continuous function $p:\cozeroset{w}\to K$ so that
\begin{eqnarray}
T(u)(x) = \left\{ \begin{array}{lll}
w(x)u(p(x)) & \text{if} & x\in \cozeroset{w} \bigskip \\
0 & \text{if} & x\in\zeroset{w} \\
\end{array}
\right.\label{EQ:  Prop-Chacterization of injective interval preserving lattice homomorphisms on C(K)}
\end{eqnarray}
for all $u\in\cont(K)$.  In particular, $w=T(\onefunction_K)$.

Assume that $T$ is a lattice homomorphism.  Then the following statements are true.\begin{enumerate}
    \item[(i)] $T$ is order continuous if and only if $p$ is almost open.
    \item[(ii)] $T$ is injective if and only if $p[\cozeroset{w}]$ is dense in $K$.
    \item[(iii)] $T$ is interval preserving if and only if $p[\cozeroset{w}]$ is $\cont^\ast$-embedded in $K$ and $p$ is a homeomorphism onto $p[\cozeroset{w}]$.
\end{enumerate}
\end{prop}

\begin{proof}
The first part of the result is well known, see for instance \cite[Theorem 4.25]{AbramovichAliprantis2002}.  Now suppose that $T$ is a lattice homomorphism.  The statement (i) follows from \cite[Theorem 4.4]{vanImhoff2018}, or, from \cite[Theorem 7.1 (iii)]{Bilokopytov2021}.

We prove (ii).  Assume that $p[\cozeroset{w}]$ is dense in $K$.  Let $u\in\cont(K)$ satisfy $T(u)=0$.  Then $w(x)u(p(x))=0$ for all $x\in \cozeroset{w}$.  Hence $u(z)=0$ for all $z\in p[\cozeroset{w}]$.  Since $p[\cozeroset{w}]$ is dense in $K$ it follows that $u=0$.  Thus $T$ is injective.  Conversely, suppose that $p[\cozeroset{w}]$ is not dense in $K$.  Then there exists $0<u\in\cont(K)$ so that $u(z)=0$ for all $z\in p[\cozeroset{w}]$; that is, $u(p(x))=0$ for all $x\in \cozeroset{w}$.  Hence $T(u)(x) = w(x)u(p(x))=0$ for all $x\in \cozeroset{w}$.  By definition $T(u)(x)=0$ for all $x\in \zeroset{w}$ so that $T(u)=0$.  Therefore $T$ is not injective.  Thus (ii) is proved.

Lastly we verify (iii).  Suppose that $T$ is interval preserving.  We first show that $p[\cozeroset{w}]$ is $\cont^\ast$-embedded in $K$. Consider $0\leq f\in \contb(p[\cozeroset{w}])$.  We must show that there exists a function $g\in\cont(K)$ so that $g(z)=f(z)$ for all $z\in p[\cozeroset{w}]$.  We may assume that $f\leq \onefunction_{p[\cozeroset{w}]}$.  Define $v:L\to \R$ by setting
\[
v(x)\defeq \left\{ \begin{array}{lll}
w(x)f(p(x)) & \text{if} & x\in \cozeroset{w} \bigskip \\
0 & \text{if} & x\in\zeroset{w} \\
\end{array}
\right.
\]
for every $x\in K$.  It is clear that $v$ is continuous on $\cozeroset{w}$ and on the interior of $\zeroset{w}$.  For every other point $x\in K$, continuity of $v$ follows from the inequality $0\leq v \leq w$.  From this last inequality and the fact that $T$ is interval preserving it follows that there exists $0\leq g\leq \onefunction_K$ so that $T(g)=v$.  If $x\in p[\cozeroset{w}]$ then $w(x)f(p(x)) = v(x) = T(g)(x) = w(x)g(p(x))$ so that $f(p(x))=g(p(x))$; that is, $g(z)=f(z)$ for all $z\in p[\cozeroset{w}]$.

Next we show that $p$ is a homeomorphism onto $p[\cozeroset{w}]$.  First we show that $p$ is injective.  Consider distinct $x_0,x_1\in \cozeroset{w}$ and suppose that $p(x_0)=p(x_1)$.  There exists $v \in \cont(L)$ with $0 < v \leq w = T(\onefunction_K)$ such that $v(x_0)=0$ and $v(x_1)>0$.  Because $T$ is interval preserving there exists $0<u\leq \onefunction_K$ in $\cont(K)$ so that $T(u)=v$.  Then $u(p(x_0))=0$ and $u(p(x_1))>0$, contradicting the assumption that $p(x_0)=p(x_1)$.  Therefore $p$ is injective.

It remains to verify that $p^{-1}$ is continuous.  Let $(x_i)$ be a net in $\cozeroset{w}$ and $x\in\cozeroset{w}$ so that $(p(x_i))$ converges to $p(x)$ in $K$.  Suppose that $(x_i)$ does not converge to $x$.  Passing to a subnet of $(x_i)$ if necessary, we obtain a neighbourhood $V$ of $x$ so that $x_i\notin V$ for all $i$.  Therefore there exists a function $0<v\leq w$ in $\cont(L)$ so that $v(x)>0$ and $v(x_i)=0$ for all $i$.  Because $T$ is interval preserving there exists a function $u\in\cont(K)$ so that $T(u)=v$.  In particular, $w(x)u(p(x))=v(x)>0$ so that $u(p(x))>0$, but $w(x_i)u(p(x_i))=v(x_i)=0$ so that $u(p(x_i))=0$ for all $i$.  Therefore $(u(p(x_i)))$ does not converge to $u(p(x))$, contradicting the continuity of $u$.  Hence $(x_i)$ converges to $x$ so that $p^{-1}$ is continuous.

Conversely, suppose that $p$ is a homeomorphism onto $p[\cozeroset{w}]$, and that $p[\cozeroset{w}]$ is $\cont^\ast$-embedded in $K$.  Let $0<u\in\cont(K)$ and $0\leq v\leq T(u)$ in $\cont(L)$.  Define $f:p[\cozeroset{w}]\to\R$ by setting
\[
f(z)\defeq \frac{1}{w(p^{-1}(z))}v(p^{-1}(z)),~~ z\in p[\cozeroset{w}].
\]
Because $p^{-1}:p[\cozeroset{w}] \to \cozeroset{w}$ is continuous, $f$ is continuous.  Furthermore, $0\leq f(z)\leq u(z)$ for all $z\in p[\cozeroset{w}]$.  Therefore $f$ is a bounded continuous function on $p[\cozeroset{w}]$.  By assumption there exists a continuous function $g:K\to \R$ so that $g(z)=f(z)$ for all $z\in p[\cozeroset{w}]$.  Since $0\leq f\leq u$ on $p[\cozeroset{w}]$, the function $g$ may to chosen so that $0\leq g\leq u$.  For $x\in\cozeroset{w}$ we have
\[
T(g)(x)=w(x)g(p(x))=w(x)f(p(x))=\frac{w(x)v(x)}{w(x)}=v(x),
\]
and for $x\in\zeroset{w}$ we have $v(x)=0=T(g)(x)$.  Therefore $T(g)=v$ so that $T$ is interval preserving.
\end{proof}

\begin{thm}\label{Thm:  Structure theorem for order duals}
Let $\vlat{E}$ be a vector lattice.  The following statements are equivalent. \begin{enumerate}
    \item[(i)] There exists a relatively uniformly complete Archimedean vector lattice $\vlat{F}$ so that $\vlat{E}$ is lattice isomorphic to $\vlat{F}^\sim$.
    \item[(ii)] There exists an inverse system $\cal{I}\defeq \left((\vlat{M}(K_\alpha))_{\alpha \in I},(p_{\beta,\alpha})_{\beta \scc \alpha}\right)$ in $\nvlic$, with each $K_\alpha$ a compact Hausdorff space, such that the following holds. \begin{enumerate}
            \item[(a)] For each $\beta \scc \alpha$ in $I$ there exist a function $w\in \cont(K_\beta)^+$ and a homeomorphism $t:\cozeroset{w}\to t[\cozeroset{w}]\subseteq K_\alpha$ onto a dense $\cont^\star$-embedded subspace of $K_\alpha$ so that for every $\mu\in \vlat{M}(K_\beta)$,
                \[
                p_{\beta , \alpha}(\mu)(A) = \int_{p^{-1}[A]} w \thinspace d\mu,~~ A\in\borel{K_\alpha}.
                \]
            \item[(b)] For every $\alpha\in I$ there exists a normal lattice morphism ${p_\alpha :\vlat{E}\to\vlat{M}(K_\alpha)}$ such that $\proj{\cal{I}} = \left( \vlat{E},(p_\alpha)_{\alpha \in I} \right)$ in $\nvlc$.
        \end{enumerate}
\end{enumerate}
\end{thm}

\begin{proof}[Proof that (i) implies (ii)]
By Theorem \ref{Thm:  Structure theorem} there exist a direct system \linebreak $\cal{D} \defeq \left(  (\cont(K_\alpha))_{\alpha\in I}, (e_{\alpha,\beta})_{\alpha \precc \beta} \right)$ in $\nvlic$, with each $K_\alpha$ a compact Hausdorff space, and interval preserving normal lattice homomorphisms ${e_\alpha: \cont(K_\alpha)\to \vlat{F}}$ so that $\cal{S}\defeq \left(\vlat{F},(e_\alpha)_{\alpha\in I}\right)$ is the direct limit of $\cal{D}$ in $\nvlic$.  Note that $e_{\alpha,\beta}$ is injective for all $\alpha \precc \beta$ in $I$.  By Theorem \ref{Thm:  Dual of ind sys in VLIC is proj of duals in NVL} and the Riesz Representation Theorem \cite[Theorem 18.4.1]{Semadeni1971}, $\orderdual{\cal{S}} \defeq \left( \vlat{E},(e_\alpha^\sim)_{\alpha \in I}\right)$ is the inverse limit of the inverse system $\orderdual{\cal{D}} \defeq \left(\vlat{M}(K_\alpha),(e_{\alpha,\beta}^\sim)_{\alpha \precc \beta} \right)$ in $\nvlc$.  Thus the claim in (b) holds.

Fix $\beta \scc \alpha$ in $I$.  We show that $e_{\alpha,\beta}^\sim$ is of the from given in (a).  By Proposition \ref{Prop:  Chacterization of injective interval preserving lattice homomorphisms on C(K)} there exist $w\in \cont(K_\beta)^+$ and a homeomorphism $t:\cozeroset{w}\to t[\cozeroset{w}]\subseteq K_\alpha$ onto a dense $\cont^\star$-embedded subspace of $K_\alpha$ so that
\[
 e_{\alpha,\beta}(u) (x) = \left\{ \begin{array}{lll}
w(x)u(t(x)) & \text{if} & x\in \cozeroset{w} \bigskip \\
0 & \text{if} & x\in\zeroset{w} \\
\end{array}
\right.
\]
for all $u\in \cont(K_\alpha)$.  Let $T:\cont(K_\alpha)\to \contb(\cozeroset{w})$ and $M_w:\contb(\cozeroset{w})\to \cont(K_\beta)$ be given by $T(u)=u\circ t$ and $M_w (v) = wv$ for all $u\in\cont(K_\alpha)$ and $v\in \contb(\cozeroset{w})$, with $wv$ defined as identically zero outside $\cozeroset{w}$.  Then $T$ and $M_w$ are positive operators and $e_{\alpha , \beta} = M_w\circ T$; hence $e_{\alpha,\beta}^\sim = T^\sim\circ M_w^\sim$.  It follows from \cite[Theorems 3.6.1 \& 9.1.1]{Bogachev2007} that $T^\sim( \mu)(A) = \mu(t^{-1}[A])$ for every $\mu\in \vlat{M}(\cozeroset{w})$ and $A\in\borel{K_\alpha}$.  The Riesz Representation Theorem shows that, for each $\nu \in \vlat{M}(K_\beta)$ and every Borel set $B$ in $\cozeroset{w}$,
\[
M_w^\sim(\nu)(B) = \int_Bw \thinspace d\nu.
\]
Hence for $\mu\in \vlat{M}(K_\beta)$ and $A\in\borel{K_\alpha}$,
\[
e_{\alpha,\beta}^\sim (\mu) (A) = \int_{t^{-1}[A]} w \thinspace d\mu
\]
as claimed.
\end{proof}

\begin{proof}[Proof that (ii) implies (i)]
Fix $\beta \scc \alpha$ in $I$ and consider the function $w\in \cont(K_\beta)^+$ and  the homeomorphism $t:\cozeroset{w}\to t[\cozeroset{w}]\subseteq K_\alpha$ given in (b).  Define the map $e_{\alpha,\beta}:\cont(K_\alpha)\to \cont(K_\beta)$ as
\[
e_{\alpha,\beta}(u)(x) = \left\{ \begin{array}{lll}
w(x)u(t(x)) & \text{if} & x\in \cozeroset{w} \bigskip \\
0 & \text{if} & x\in\zeroset{w} \\
\end{array}
\right.
\]
We show that $\cal{D}\defeq \left((\cont(K_\alpha))_{\alpha\in I},(e_{\alpha,\beta})_{\alpha \precc \beta}\right)$ is a direct system in $\nvlic$.

It follows by Proposition \ref{Prop:  Chacterization of injective interval preserving lattice homomorphisms on C(K)} that each $e_{\alpha,\beta}$ is an injective interval preserving normal lattice homomorphism.  It remains to show that $e_{\alpha,\gamma} = e_{\beta,\gamma}\circ e_{\alpha,\beta}$ for all $\alpha\precc \beta \precc \gamma$ in $I$.  An argument similar to that in the proof that (i) implies (ii) shows that $e_{\alpha,\beta}^\sim = p_{\beta,\alpha}$ for all $\alpha \precc\beta$; hence $e_{\alpha,\beta}^{\sim\sim} = p_{\beta,\alpha}^\sim$.  By Proposition \ref{Prop:  Dual system of projective system is inductive system}, $\cal{I}^\sim \defeq \left( (\vlat{M}(K_\alpha)^\sim)_{\alpha\in I}, (p^\sim_{\beta , \alpha})_{\beta\scc \alpha}\right)$ is a direct system in $\nvlic$ and therefore $e_{\alpha,\gamma}^{\sim\sim}=e_{\beta,\gamma}^{\sim\sim}\circ e_{\alpha,\beta}^{\sim\sim}$ for all $\alpha \precc \beta \precc \gamma$ in $I$.  Since $\cont(K_\alpha)$ has a separating order dual for every $\alpha\in I$, it follows that $e_{\alpha,\gamma}=e_{\beta,\gamma}\circ e_{\alpha,\beta}$.  Hence $\cal{D}$ is a direct system in $\nvlic$.

Since each $e_{\alpha,\beta}$ is injective, $\ind{\cal{D}}\defeq \left(\vlat{F},(e_\alpha)_{\alpha\in I}\right)$ exists in $\nvlic$ by Theorem \ref{Thm:  Existence of Inductive Limits in NVLI}.  Since $\cont(K_\alpha)$ is Archimedean and relatively uniformly complete for each $\alpha\in I$ it follows from Theorem \ref{Thm:  Inductive Limit Permanence} (i) and (v) that $\vlat{F}$ is also Archimedean and relatively uniformly complete.  Because $e_{\alpha,\beta}^\sim = p_{\beta,\alpha}$ for all $\alpha\precc\beta$ in $I$, $\orderdual{\cal{D}}=\cal{I}$.  Therefore, by Theorem \ref{Thm:  Dual of ind sys in VLIC is proj of duals in NVL}, there exists a lattice isomorphism $T: \vlat{F}^\sim \to \vlat{E}$ such that the diagram
\[
\begin{tikzcd}[cramped]
\vlat{F}^\sim \arrow[rd, "e^\sim_\alpha"'] \arrow[rr, "T"] & & \vlat{E} \arrow[dl, "p_\alpha"]\\
& \vlat{M}(K_\alpha)
\end{tikzcd}
\]
commutes for all $\alpha\in I$.  This completes the proof.
\end{proof}

\bibliographystyle{amsplain}
\bibliography{Limitsofvectorlattices}

\end{document}